\RequirePackage{fix-cm}

\documentclass[smallextended]{svjour3}       
\smartqed  
\usepackage{graphicx}
\journalname{Journal of Statistical Physics}

\usepackage[nosumlimits]{amsmath}
\usepackage{amsfonts}
\usepackage{latexsym}

\usepackage{algorithm}
\usepackage{algorithmic}

\usepackage{url}
\usepackage{xcolor}

\newcommand{\dd}[2]{\frac{\diff#1}{\diff#2}}

\DeclareMathOperator{\diff}{d}

\def\dd{{\color{red}\mathsf d}}
 
\usepackage{enumitem}

\usepackage[a4paper, left=2cm, right=2cm, top=2cm, bottom=2cm]{geometry}

\begin{document}

\title{Data assimilation for a quasi-geostrophic model with circulation-preserving stochastic transport noise
}
\subtitle{}


\author{Colin Cotter \and
        Dan Crisan \and
        Darryl Holm   \and
        Wei Pan    \and
        Igor Shevchenko
}


\institute{C. Cotter \at
              \email{colin.cotter@imperial.ac.uk}           
           \and
           D. Crisan \at
           \email{d.crisan@imperial.ac.uk}
           \and
           D. Holm \at
           \email{d.holm@imperial.ac.uk}
           \and
           W. Pan \at
           \email{wei.pan@imperial.ac.uk}
           \and
           I. Shevchenko \at
           \email{i.shevchenko@imperial.ac.uk}
           \\
Department of Mathematics, Imperial College London, London, SW7 2AZ, UK           
}

\date{Received: date / Accepted: date}

\maketitle

\begin{abstract}
This paper contains the latest installment of the authors' project (see~\cite{CCHWS2019_1,CCHWS2019_2,CCHWS2019_3}) 
on developing ensemble based data assimilation methodology for high dimensional fluid dynamics models. 
The algorithm presented here is a particle filter that combines model reduction, tempering, jittering, and nudging. 
The methodology is tested on a two-layer quasi-geostrophic model for a $\beta$-plane channel flow
with $O(10^6)$ degrees of freedom out of which only a minute fraction are noisily observed. The model is reduced by following the stochastic variational approach for geophysical fluid dynamics introduced in~\cite{holm2015variational} as a framework for deriving stochastic parametrisations for unresolved scales. 
The reduction is substantial: the computations are done only for $O(10^4)$ degrees of freedom.  We introduce a stochastic time-stepping scheme for the two-layer model and prove its consistency in time. Then, we analyze the effect of the different procedures (tempering combined with jittering and nudging) on the performance of the data assimilation
procedure using the reduced model, as well as how the dimension of the observational data (the number of "weather stations") and the data assimilation step affect the accuracy and uncertainty of the results. 

\end{abstract}

\section{Introduction}
\label{sec:intro}

In recent years, there has been an increased scientific effort in developing ensemble based data assimilation as an alternative to variational data assimilation which is currently used in operation centres for numerical weather prediction\footnote{However, 
see~\cite{Potthast_et_al2019} for a recent application of particle filters within an operational framework.}. Such methods can be more suited for fully nonlinear systems and complex observation operators. The work presented in this paper is
part of this wider effort (see the survey paper~\cite{VanLeeuwen2019} and the references therein for recent developments in this direction). 

The cornerstone of the current work is the introduction of stochastic parametrization to model uncertainity via the so-called Stochastic Advection by Lie Transport (SALT) approach~\cite{holm2015variational}.  The stochasticity is introduced into the advection part of
the dynamics via a constrained variational principle. This is a general approach for deriving stochastic partial differential equation (SPDE) models for geophysical 
fluid dynamics (GFD). In this work we apply it 
to the N-layer quasi-geostrophic model (see Section 2 for details).  By adding stochasticity into the advection operator, one can model uncertain transport behaviour. The uncertainty in our case occurs as we assimilate the data coming from observing a high resolution model, but use low resolution realisations of the model. This model reduction is crucial as it enables us to complete the task by using fairly modest computational resources\footnote{We used a stand-alone workstation with 128GB of RAM and 2x6-core Intel Xeon E5-2643v4 3.4GHz processors.}. 
The stochastic term used to model the missing uncertainties is calibrated by using a data driven approach as described in~\cite{CCHWS2019_2}. 

This paper complements the work done in \cite{CCHWS2019_3} where the "true state" is chosen to be the solution of the Euler equation with forcing and damping. We choose the quasi-geostrophic (QG) equations for this work as it has a qualitatively different behaviour from Euler: As one can see in Fig.~\ref{fig:qf_all_grids}, the solution of the QG equations is far less homogeneous than the that of the Euler equation, exhibiting multiple large-scale zonally elongated jets as well as small-scale vortices. Indeed, one of the findings of our work is that the formation of jets is heavily influenced by the size of the grid: the coarser the grid the less jets are formed. Nevertheless, once data assimmilation is applied to the coarser model (stochastically parametrized and properly calibrated), the number of jets can be preserved. The occurrence of the jets makes the data assimilation problem harder. Whilst in \cite{CCHWS2019_3} it sufficed to use only tempering and jittering to assimilate the data, in the current work we obtained far better results only after we added the nudging procedure to the two already used in \cite{CCHWS2019_3}. Another difference from the work done in \cite{CCHWS2019_3} was the choice of the initial ensemble, which here chosen as a set of independent realizations from the solution of the stochastically perturbed QG equation. We found this to be a more natural alternative to the one used in \cite{CCHWS2019_3}.         

The use of the combination of tempering and jittering is theoretically justified. Indeed in~\cite{BCJ2016} it is shown that the use of the two procedures can produce particle filters suitable for solving high dimensional problems. More precisely, it is proved that the effective sample size of the ensemble of particles remains under control as the dimension $d$ of the underlying system increases with a computational cost that is at most quadratic in $d$. By contrast, a generic (bootstrap) particle filter would require a computational cost that is exponential in $d$. 

As is usually the case in data assimilation, the particle filter proceeds by alternating between forecast and analysis cycles. In each analysis step, observations of the current (and possibly past) state of a system are combined with the results from a prediction model (the forecast) to produce an analysis. The tempering and jittering are used to complete the analysis step, whilst the nudging procedure is used in the forecast step. In the absence of nudging, the ensemble particles have trajectories that are independent solutions of the stochastic QG equations. Nudging consists in adding a drift to the trajectories of the particles with the aim of maximising the likelihood of their positions given the observation data. This introduces a bias in the system that is corrected at the analysis step. It follows that also the nudging procedure is theoretically justified through a standard convergence argument, see for example~\cite{Crisan2002ASO}. It follows that the data assimilation algorithm presented in this paper will give an asymptotically (as the number of particles increases) consistent approximation of the posterior distribution of the state given the data. That does not mean that the empirical distribution of the ensemble is a good approximation of the posterior. The size of the ensemble is 100 and this is certainly not enough to approximate a posterior distribution in a state space of dimension $O(10^4)$. However, it offers a sound theoretical basis for the algorithms presented here. We give further details of this issue in Section~\ref{sec:da_methods}.         
 
The paper is structured as follows. 
In Section~\ref{sec:ham} we describe the deterministic N-layer QG equations and Hamiltonian formulation for the stochastic
multi-layer QG model.
Section~\ref{sec:2d_qg} presents the deterministic and stochastic QG equations, and numerical methods for the QG model.
In Section~\ref{sec:2d_qg_consistency} we prove that the stochastic CABARET scheme is consistent with the stochastic QG equation
in the mean square sense in time.
In Section~\ref{sec:da_methods} we discuss different procedures used: Bootstrap Particle Filter, jittering, tempering and nudging procedures. 
In Section~\ref{sec:num_res} we present and discuss the numerical experiments and results, and study
how the data assimilation methods influence the quality of the forecast given by the stochastic QG model. The following is a summary of the main numerical experiments contained in this paper:
\begin{itemize}
    \item Dependence of the relative bias and the ensemble mean $l_2$-norm relative error between the true deterministic solution and its
    stochastic parameterisation on the data assimilation 
    step (Figs.~\ref{fig:av_relerr_v_129x65} and~\ref{fig:av_relerr_v_129x65_dt_2h}), 
    grid resolution (Fig.~\ref{fig:av_relerr_v_257x129}), data assimilation methods
    (Figs.~\ref{fig:av_relerr_v_129x65_NDG_ws16} and~\ref{fig:av_relerr_v_257x129_NDG_ws16}),
    and the number of weather stations (Figs.~\ref{fig:av_relerr_v_129x65_NDG_ws32} and~\ref{fig:av_relerr_v_257x129_NDG_ws32});
   \item Analysis of how uncertainty of the stochastic spread is influenced by the data assimilation 
    step (Figs.~\ref{fig:spread_129x65_16ws_4h} and~\ref{fig:spread_129x65_16ws_2h}), 
    grid resolution (Fig.~\ref{fig:spread_257x129_16ws_4h}), data assimilation methods
    (Figs.~\ref{fig:spread_129x65_16ws_4h_nudging} and~\ref{fig:spread_257x129_16ws_4h_nudging}),
    and the number of weather stations (Figs.~\ref{fig:spread_129x65_32ws_4h_nudging} and~\ref{fig:spread_257x129_32ws_4h_nudging});
   \item Forecast reliability rank histrograms for the stochastic QG model with and without the data assimilation procedure (Fig.~\ref{fig:rank_histograms_257x129_32ws_4h}).    
\end{itemize}
Finally, Section~\ref{sec:concl} concludes the present work and discusses the outlook for future research.


\section{Hamiltonian equations of motion for a multi-layer fluid}
\label{sec:ham}

\subsection{The deterministic $N$-layer quasi-geostrophic (NLQG) equations}
\label{sec:D_NLQG}

Consider a stratified fluid of $N$ superimposed layers of constant densities $\rho_1 < \dots<\rho_N$; 
the layers being stacked according to increasing density, such that the density of the upper layer is $\rho_1$. 
The quasi-geostrophic (QG) approximation assumes that the velocity field is constant in the vertical direction and 
that in the horizontal direction the motion obeys a system of coupled incompressible shallow water equations. 
We shall denote by $\mathbf{u}_i = (- \,\partial_y\psi_i, \partial_x\psi_i) = \mathbf{\hat{z}}\times\nabla \psi_i$ the velocity field of the $i^{th}$ layer, where $\psi_i$ is its stream function, and the layers are numbered from the top to the bottom. We define the 
potential vorticity of the $i^{th}$ layer as
\begin{equation}\label{omsubi}
\omega_i = q_i + f_i = \Delta  \psi_i + \alpha_i \sum\limits^N_{j=1} T_{ij}\psi_j + f_i
=: \sum\limits^N_{j=1} E_{ij}\psi_j +  f_i
\,,\qquad
i=1,\dots,N,
\end{equation}
where the potential vorticity  is defined as $\omega_i = q_i + f_i $, the  elliptic operator $E_{ij}$ defines the layer vorticity,
\[q_i = \sum\limits^N_{j=1} E_{ij}\psi_j:= \Delta \psi_i + \alpha_i \sum\limits^N_{j=1} T_{ij}\psi_j\,,\] and the constant parameters $\alpha_i $, $f_i $, $f_0$, $\beta$, $f_N$ are
\begin{align}\label{paramdefs}
\begin{split}
\alpha_i &= (f_0^2/g)\big((\rho_{i+1}-\rho_i)/\rho_0\big)D_i
\,,\qquad i=1,\dots,N,
\\
f_i &= f_0 + \beta y
\,,\qquad  i=1,\dots,N-1,
\\
f_N &= f_0 + \beta y + f_0 d(y)/D_N,
\\
f_0 &= 2\Omega \sin(\phi_0)
\,,\qquad 
\beta = 2\Omega \cos(\phi_0)/R\,,
\end{split}
\end{align}
where $g$ is the gravitational acceleration, $\rho_0 = (1/N)(\rho_1 + \dots + \rho_N)$ is the mean density, $D_i$ is the mean thickness of the $i^{th}$ layer, $R$ is the Earth's radius, $\Omega$ is the Earth's angular velocity, $\phi_0$ is the reference latitude, and $d(y)$ is the shape of the bottom. The $N \times N$ symmetric tri-diagonal matrix $T_{ij}$ represents the second-order difference operator,
\begin{equation}\label{2ndDiffTop}
\sum\limits^N_{j=1} T_{ij}\psi_j  = (\psi_{i-1} - \psi_i) - (\psi_i - \psi_{i+1})\,,
\end{equation}
so that 
\begin{equation}\label{2ndDiffT}
T_{ij} = 
\begin{bmatrix}
-1  & 1     & 0      &  \ldots      & \ldots     & 0  \\
1  & -2 & \ddots & \ddots &        &  \vdots\\
0 & \ddots & \ddots & \ddots & \ddots &  \vdots  \\
\vdots   & \ddots & \ddots & \ddots & \ddots & 0  \\
\vdots  &        & \ddots & \ddots & -2 & 1 \\
0 & \ldots      &   \ldots     & 0     & 1      & -1
\end{bmatrix}
\,,\qquad
i,j=1,\dots,N.
\end{equation}
With these standard notations, the motion of the NLQG fluid is given by
\begin{equation}\label{NlayerVortDyn}
\partial_t q_i  = \Big\{ \omega_i ,\,\psi_i \Big\}_{xy}
=
-\, \mathbf{\hat{z}} \times \nabla \psi_i  \cdot \nabla \omega_i
=
-\,
\mathbf{u}_i \cdot \nabla \omega_i
\,,\qquad
i =1,\dots,N,
\end{equation}
where $\mathbf{\hat{z}}$ is the vertical unit vector, $\mathbf{u}_i = 
\mathbf{\hat{z}} \times \nabla \psi_i  $ is the horizontal flow velocity in the $i^{th}$ layer, and the brackets in 
\begin{equation}\label{canPoissonBrkt}
\{\omega,\psi\}=J(\omega,\psi)=\omega_x\psi_y-\omega_y\psi_x
= \mathbf{\hat{z}}\cdot \nabla \omega \times \nabla \psi 
\end{equation}
denote the usual $xy$ canonical Poisson bracket in $\mathbb{R}^2$. The boundary conditions in a compact domain $D\subset\mathbb{R}^2$ with smooth boundary $\cup_{j}\partial{D}_j$
are $\psi_j|_{(\partial{D}_j)} = constant$, whereas in the entire $\mathbb{R}^2$ they are
$\lim_{(x, y)\to±\infty} \nabla\psi_j=0$.
The space of variables with canonical Poisson bracket in \eqref{canPoissonBrkt} consists of $N$-tuples $(q_1,\dots, q_N)$ of real-valued functions on $D$ (the ``generalized vorticities'') with the above boundary conditions and certain smoothness properties that guarantee that solutions are at least of class $C^1$.
The Hamiltonian for the $N$-layer vorticity dynamics in \eqref{NlayerVortDyn} is the total energy
\begin{equation}\label{NlayerVortDyn-erg}
H(q_1,\dots, q_N) = \frac12\int_D
\Big[\sum\limits^N_{i=1} \frac{1}{\alpha_i} |\nabla\psi_i|^2
+
\sum\limits^{N-1}_{i=1} (\psi_{i+1}-\psi_i)^2\Big]dx\,dy
\,,\qquad
i =1,\dots,N,
\end{equation}
with stream function $\psi_i$ determined from vorticity $\omega_i$ by solving the elliptic equation \eqref{omsubi} for $q_i=\omega_i-f_i$ with 
\begin{equation}\label{elliptic-op}
q_i = \sum\limits^N_{j=1} E_{ij}\psi_j\,,
\end{equation}
for the boundary conditions discussed above. Hence, we find that 
\begin{equation}\label{NlayerVortDyn-erq}
H(q_1,\dots, q_N) 
= -\frac12\int_D\sum\limits^N_{i,j=1} \psi_iE_{ij}\psi_j dx\,dy
= -\frac12\int_D\sum\limits^N_{i,j=1} q_i E^{-1}_{ij}*q_j dx\,dy
= -\frac12\int_D\sum\limits^N_{i=1} q_i \psi_i dx\,dy\,,
\end{equation}
where $E^{-1}_{ij}*q_j = \psi_i$ denotes convolution with the Greens function $E^{-1}_{ij}$ for the symmetric elliptic operator $E_{ij}$. The relation \eqref{NlayerVortDyn-erq} means that $\delta H/\delta q_i = \psi_i$ for the variational derivative of the Hamiltonian functional $H$ with respect to the function $q_j$.

\begin{remark}[Lie--Poisson bracket]
Equations \eqref{NlayerVortDyn} are Hamiltonian with respect to the Lie--Poisson bracket on the dual of $\oplus \sum\limits^N_{i=1}{\cal F}(D)$ given by
\begin{equation}\label{VortLie-PoissonBrkt}
\{F,H\}(q_1,\dots, q_N)
=
\sum\limits^N_{i=1} \int_D (q_i + f_i(x))
\left\{\frac{\delta F}{\delta q_i},\,\frac{\delta H}{\delta q_i}\right\}_{xy}dx\,dy\,,
\end{equation}
for arbitrary functions $F$ and $H$, provided the domain of flow $D$ is simply connected.%
\footnote{If the domain $D$ is not simply connected, then variational derivatives such as $\delta H/\delta q_i$ must be interpreted with care, because in that case the boundary conditions on $\psi_i$ will come into play \cite{McWilliams1977}.}

The motion equations \eqref{NlayerVortDyn} for $q_i$ now follow from the Lie--Poisson bracket \eqref{VortLie-PoissonBrkt} after an integration by parts to write it equivalently as
\begin{equation}\label{VortLie-PoissonBrkt2}
\frac{dF}{dt} = \{F,H\}(q_1,\dots, q_N)
=
-\sum\limits^N_{i=1} \int_D \frac{\delta F}{\delta q_i}
\left\{q_i + f_i(x) ,\,\frac{\delta H}{\delta q_i}\right\}_{xy}dx\,dy
\,,
\end{equation}
and recalling that $\delta H/\delta q_i =-E^{-1}_{ij}*q_j=- \,\psi_i$, $i=1,2,\dots,N$, so that equations \eqref{NlayerVortDyn} follow. 
\end{remark}
\begin{remark}[Constants of motion]
According to equations \eqref{NlayerVortDyn}, the material time derivative of $\omega_i(t, x, y)$ vanishes along the flow lines of the divergence-free horizontal velocity $\mathbf{u}_i = \mathbf{\hat{z}}\times\nabla \psi_i   $. Consequently, for every differentiable function $\Phi_i: \mathbb{R}\to\mathbb{R}$ the functional
\begin{equation}\label{Casimirs}
C_{\Phi_i}(\omega_i)
=
\int_D \Phi_i (\omega_i)\,dx\,dy
\end{equation}
is a conserved quantity for the system \eqref{NlayerVortDyn} for $i =1,\dots,N$, provided the integrals exist. By Kelvin's circulation theorem, the following integrals over an advected domain $S(t)$ in the plane are also conserved,
\begin{equation}\label{KelvinThm}
I_i(t) 
=
\int_{S(t)} \omega_i \,dx\,dy
= 
\int_{\partial S(t)}  \nabla \psi_i \cdot \mathbf{\hat{n}} \,ds
\,,
\end{equation}
where $\mathbf{\hat{n}}$ is the horizontal outward unit normal and $ds$ is the arclength parameter of the closed curve $\partial S(t)$ bounding the domain $S(t)$ moving with the flow. 
\end{remark}

\subsection{Hamiltonian formulation for the stochastic NLQG fluid\label{sec:D_NLQG_ham}}

Having understood the geometric structure (Lie--Poisson bracket, constants of motion and Kelvin circulation theorem) for the deterministic case, we can introduce the stochastic versions of equations \eqref{NlayerVortDyn} by simply making the Hamiltonian stochastic while preserving the previous geometric structure, as done in the previous section. Namely, we choose
\begin{equation}\label{Ham-stoch integral}
h(t) = h(0) + \int_0^t H(\{q\})ds + \int_0^t \int_D \sum\limits^N_{i=1}\sum\limits^K_{k=1} q_i (s,x,y)\zeta^k_i(x,y) \,dx\,dy \circ dW_k(s)
\,,
\end{equation}
so that
\begin{equation}\label{Ham-stoch}
\dd h = H(\{q\})dt + \int_D \sum\limits^N_{i=1}\sum\limits^K_{k=1} q_i (t,x,y)\zeta^k_i(x,y) \,dx\,dy \circ dW_k(t)
\,,
\end{equation}
where the $\zeta^k_i(x,y)$, $k =1,\dots,K$ represent the correlations of the Stratonovich noise we have introduced in~\eqref{Ham-stoch}. 

For this stochastic Hamiltonian, the Lie--Poisson bracket \eqref{VortLie-PoissonBrkt}  leads to the following stochastic process for the transport of the $N$-layer generalised vortices,
\begin{equation}\label{NlayerVortDyn-stoch}
\dd q_i  = \Big\{ \omega_i ,\,\dd \psi \Big\}_{xy}
=
J\big(\omega_i ,\,\dd \psi \big)
=
\nabla (\dd \psi_i ) \times \mathbf{\hat{z}}\cdot \nabla \omega_i
=
-\,
\dd \mathbf{u}_i \cdot \nabla \omega_i
\,,\qquad
i =1,\dots,N,
\end{equation}
where we have defined the stochastic transport velocity in the $i^{th}$ layer
\begin{equation}\label{Nlayer-stoch-vel}
\dd \mathbf{u}_i :=
\mathbf{\hat{z}} \times \nabla (\dd \psi_i )  
\,,\qquad
i =1,\dots,N,
\end{equation}
in terms of its stochastic stream function 
\begin{equation}\label{Nlayer-stoch-dpsi}
\dd \psi_i  := \psi_i \,dt + \sum\limits^K_{k=1} \zeta^k_i(x,y) \circ dW_k(t)
= \dd \frac{\delta h}{\delta q_i}
\,,\qquad
i =1,\dots,N,
\end{equation}
determined from the variational derivative of the stochastic Hamiltonian in \eqref{Ham-stoch} with respect to the generalised vorticity $q_i$ in the $i^{th}$ layer. 

\begin{remark}[Constants of motion] The constants of motion $C_{\Phi_i}$ in \eqref{Casimirs} and the Kelvin circulation theorem for the integrals $I_i$ in \eqref{KelvinThm} persist for the stochastic generalised vorticity equations in \eqref{NlayerVortDyn-stoch}. This is because both of these properties follow from the Lie-Poisson bracket in \eqref{VortLie-PoissonBrkt}. However, the stochastic Hamiltonian in \eqref{Ham-stoch} is not conserved, since it depends explicitly on time, $t$, through its Stratonovich noise term.  
\end{remark}

\section{The two-dimensional multilayer quasi-geostrophic model}
\label{sec:2d_qg}

\subsection{Deterministic case}
\label{sec:determ}

The two-layer deterministic QG equations for the potential vorticity (PV) anomaly $q$ in a 
domain $\Omega$ are given by the PV material conservation law augmented with forcing and dissipation~\cite{Pedlosky1987,Vallis2006}:
\begin{equation}
\begin{aligned}
\frac{\partial q_1}{\partial t}+\mathbf{u}_1\cdot\nabla q_1&=\nu\Delta^2\psi_1-\beta\frac{\partial \psi_1}{\partial x},\\
\frac{\partial q_2}{\partial t}+\mathbf{u}_2\cdot\nabla q_2&=\nu\Delta^2\psi_2-\mu\Delta\psi_2-\beta\frac{\partial \psi_2}{\partial x},
\end{aligned}
\label{eq:pv}
\end{equation}
where $\psi$ is the stream function, $\beta$ is the planetary vorticity gradient, 
$\mu$ is the bottom friction parameter, $\nu$ is the lateral eddy viscosity,
and $\mathbf{u}=(u,v)$ is the velocity vector.
The computational domain $\Omega=[0,L_x]\times[0,L_y]\times[0,H]$ is a horizontally periodic flat-bottom channel of depth
$H=H_1+H_2$ given by two stacked isopycnal fluid layers of depth
$H_1$ and $H_2$. 
A mollified version of the existence and uniqueness theorem for the QG model can be found in~\cite{Farhat_et_al2012}.

Forcing in~\eqref{eq:pv} is introduced via a vertically sheared, baroclinically unstable background flow (e.g.,~\cite{BerloffKamenkovich2013})
\begin{equation}
\psi_i\rightarrow-U_i\,y+\psi_i,\quad i=1,2, 
\label{eq:forcing}
\end{equation}
where the parameters $U_i$ are background-flow zonal velocities.

The PV anomaly and stream function are related through two elliptic equations:
\begin{subequations}
\begin{align}
q_1=\Delta\psi_1+s_1(\psi_2-\psi_1),\\
q_2=\Delta\psi_2+s_2(\psi_1-\psi_2),
\end{align}
\label{eq:q_psi}
\end{subequations}
with stratification parameters $s_1$, $s_2$.

System~(\ref{eq:pv})-(\ref{eq:q_psi}) is augmented by the integral mass conservation constraint~\cite{McWilliams1977}
\begin{equation}
\frac{\partial}{\partial t}\iint\limits_{\Omega}(\psi_1-\psi_2)\ dydx=0,
\label{eq:masscon}
\end{equation}
by the periodic horizontal boundary conditions,
\begin{equation}
\boldsymbol{\psi}\Big|_{\Gamma_2}=\boldsymbol{\psi}\Big|_{\Gamma_4}\,,\quad \boldsymbol{\psi}=(\psi_1,\psi_2)\,,
\label{eq:bc24}
\end{equation}
and no-slip boundary conditions 
\begin{equation}
\boldsymbol{u}\Big|_{\Gamma_1}=\boldsymbol{u}\Big|_{\Gamma_3}=0\,.
\label{eq:bc13}
\end{equation}
set at northern and southern boundaries of the domain.

\subsubsection{Numerical method\label{sec:2d_qg_method_determ}}
The QG model~\eqref{eq:pv}-\eqref{eq:bc13} is solved with the CABARET method, which is based
on a second-order, non-dissipative and low-dispersive, conservative advection scheme~\cite{Karabasov_et_al2009}.
The CABARET scheme can simulate large-Reynolds-number flow regimes at lower
computational costs compared to conventional methods (see, e.g.,~\cite{Arakawa1966,WoodwardColella1984,ShuOsher1988,Hundsdorfer_et_al1995}),
since the scheme is low dispersive and non-oscillatory.

The CABARET method is a predictor-corrector scheme in which the components of the conservative variables
are updated at half time steps. Algorithm~\ref{alg:DCABARET} illustrates the principal steps of the CABARET method adopted from~\cite{Karabasov_et_al2009}.
To make the notation more concise, 
we introduce the forward difference operators in space 
$$\Delta_x[f]=\frac{f_{i+1,j}-f_{ij}}{\Delta x},\quad \Delta_y[f]=\frac{f_{i,j+1}-f_{ij}}{\Delta y},$$
and omit spatial and layer indices wherever possible, unless stated otherwise. The time step of the CABARET scheme is
denoted by $\tau$.

\begin{algorithm}
\caption{CABARET scheme for the deterministic QG system~\eqref{eq:pv}-\eqref{eq:bc13}}
\label{alg:DCABARET}
\begin{algorithmic}
\STATE{\underline{\bf Predictor}}
\STATE{$\displaystyle q^{n+\frac12}_{i+\frac12,j+\frac12}=q^n_{i+\frac12,j+\frac12}+
\frac{\Delta t}{2}\,F\left(q^n,u(q^n),v(q^n)\right)+\Delta t\, F_{\beta}\left(v^n,v^{n-1}\right)
+\Delta t\, F_{\rm visc}\left(\psi\left(q^{n+\frac12}\right)\right)$,}\\

\STATE{\quad $F\left(q^n_{ij},u(q^n),v(q^n)\right)=-\left(\Delta_x\left[(uq)^n_{i,j+\frac12}\right]+
\Delta_y\left[(vq)^n_{i+\frac12,j}\right]\right)$,}\\

\STATE{\quad$\displaystyle F_{\beta}\left(v^n,v^{n-1}\right)=\frac32R^n-\frac12R^{n-1},\quad
R^n=-\frac{\beta}{2}\left(v^n_{i+\frac12,j+1}+v^n_{i+\frac12,j}\right)$.}\\

\STATE{\quad The forcing term 
\vspace*{-0.125cm}
$$F_{\rm visc}\left(\psi\left(q^{n+\frac12}\right)\right)=\nu\left(\Delta^2\psi_l\right)^{n+\frac12}_{i+\frac12,j+\frac12}
-\delta_{2l}\,\mu\left(\Delta\psi_l\right)^{n+\frac12}_{i+\frac12,j+\frac12},\, l=1,2$$ 
\quad is added in the prediction step after the elliptic problem is solved.}\\

\STATE{\quad\bf Solve the elliptic system of equations with respect to $(\psi_1)^{n+\frac12}_{i+\frac12,j+\frac12}$ and $(\psi_2)^{n+\frac12}_{i+\frac12,j+\frac12}$}\\

\STATE{\qquad$\displaystyle (q_1)^{n+\frac12}_{i+\frac12,j+\frac12}=\left(\Delta\psi_1\right)^{n+\frac12}_{i+\frac12,j+\frac12}+s_1\left(\psi_{[21]}\right)^{n+\frac12}_{i+\frac12,j+\frac12},\quad
(q_2)^{n+\frac12}_{i+\frac12,j+\frac12}=\left(\Delta\psi_2\right)^{n+\frac12}_{i+\frac12,j+\frac12}+s_2\left(\psi_{[12]}\right)^{n+\frac12}_{i+\frac12,j+\frac12}\,.$}\\

\STATE{\quad\bf Calculate}

\STATE{\qquad$\displaystyle \psi^{n+\frac12}_{ij}=\frac14\left(\psi^{n+\frac12}_{i+\frac12,j+\frac12}+
\psi^{n+\frac12}_{i+\frac12,j-\frac12}+\psi^{n+\frac12}_{i-\frac12,j+\frac12}+\psi^{n+\frac12}_{i-\frac12,j-\frac12}\right).$}\\

\STATE{\quad\bf Update velocity components at the cell faces}
\STATE{\qquad $\displaystyle u^{n+\frac12}_{i,j+\frac12}=\Delta_y\left[\psi^{n+\frac12}_{ij}\right],\quad 
\left(v_l\right)^{n+\frac12}_{i+\frac12,j}=-\Delta_x\left[\psi^{n+\frac12}_{ij}\right].$}\\

\STATE{\bf\underline{Extrapolator}}
\STATE{\quad$\displaystyle u^{n+1}_{i,j+\frac12}=\frac32u^{n+\frac12}_{i,j+\frac12}-\frac12 u^{n-\frac12}_{i,j+\frac12}\,,\quad
v^{n+1}_{i,j+\frac12}=\frac32v^{n+\frac12}_{i,j+\frac12}-\frac12 v^{n-\frac12}_{i,j+\frac12}$.}\\

\STATE{\quad$\displaystyle q^{n+1}_{i+1,j+\frac12}=2q^{n+\frac12}_{i+\frac12,j+\frac12}-q^n_{i,j+\frac12}$\quad 
if $\displaystyle u^{n+1}_{i+1,j+\frac12}\ge0$; \quad$\displaystyle q^{n+1}_{i,j+\frac12}=2q^{n+\frac12}_{i+\frac12,j+\frac12}-q^n_{i+1,j+\frac12}$\quad 
if $\displaystyle u^{n+1}_{i,j+\frac12}<0$.}\\

\STATE{\quad$\displaystyle q^{n+1}_{i+\frac12,j+1}=2q^{n+\frac12}_{i+\frac12,j+\frac12}-q^n_{i+\frac12,j}$\quad 
if $\displaystyle v^{n+1}_{i+\frac12,j+1}\ge0$; \quad$\displaystyle q^{n+1}_{i+\frac12,j}=2q^{n+\frac12}_{i+\frac12,j+\frac12}-q^n_{i+\frac12,j+1}$\quad 
if $\displaystyle v^{n+1}_{i+\frac12,j}<0$.}\\

\STATE{\quad\bf Correction of the computed cell-face PV anomaly values}\\

\STATE{\qquad If $q^{n+1}_{i,j+\frac12}>M^{n+1}_{i,j+\frac12}\Rightarrow q^{n+1}_{i,j+\frac12}=M^{n+1}_{i,j+\frac12}$;\quad
If $q^{n+1}_{i,j+\frac12}<m^{n+1}_{i,j+\frac12}\Rightarrow q^{n+1}_{i,j+\frac12}=m^{n+1}_{i,j+\frac12}$.}\\

\STATE{\qquad If $q^{n+1}_{i+\frac12,j}>M^{n+1}_{i+\frac12,j}\Rightarrow q^{n+1}_{i+\frac12,j}=M^{n+1}_{i+\frac12,j}$;\quad
If $q^{n+1}_{i+\frac12,j}<m^{n+1}_{i+\frac12,j}\Rightarrow q^{n+1}_{i+\frac12,j}=m^{n+1}_{i+\frac12,j}$.}\\

\STATE{\qquad $\text{If}\,\, u^{n+1}_{i+1,j+\frac12}\ge0\quad
\left\{
\begin{array}{ll}
M^{n+1}_{i+1,j+\frac12}=\max\left(q^n_{i,j+\frac12} ,q^n_{i+\frac12,j+\frac12},q^n_{i+1,j+\frac12}\right)+\tau Q^n_{i+\frac12,j+\frac12},\\
m^{n+1}_{i+1,j+\frac12}=\min\left(q^n_{i,j+\frac12} ,q^n_{i+\frac12,j+\frac12},q^n_{i+1,j+\frac12}\right)+\tau Q^n_{i+\frac12,j+\frac12}.\\
\end{array}
\right.
$}\\

\STATE{}

\STATE{\qquad
$\text{If}\,\, u^{n+1}_{i,j+\frac12}<0\quad
\left\{
\begin{array}{ll}
M^{n+1}_{i,j+\frac12}=\max\left(q^n_{i,j+\frac12} ,q^n_{i+\frac12,j+\frac12},q^n_{i+1,j+\frac12}\right)+\tau Q^n_{i+\frac12,j+\frac12},\\
m^{n+1}_{i,j+\frac12}=\min\left(q^n_{i,j+\frac12} ,q^n_{i+\frac12,j+\frac12},q^n_{i+1,j+\frac12}\right)+\tau Q^n_{i+\frac12,j+\frac12}.\\
\end{array}
\right.
$}

\STATE{\qquad $\displaystyle Q^{n+\frac12}_{i+\frac12,j+\frac12}=\frac{q^{n+\frac12}_{i+\frac12,j+\frac12}-
q^n_{i+\frac12,j+\frac12}}{\Delta t/2}+\frac12\left((u_l)^{n+1}_{i+1,j+\frac12}+(u_l)^{n+1}_{i,j+\frac12}\right)
\Delta_x\left[q^n_{i,j+\frac12}\right]$.}\\

\STATE{\bf\underline{Corrector}}

\STATE{$\displaystyle q^{n+1}_{i+\frac12,j+\frac12}=q^{n+\frac12}_{i+\frac12,j+\frac12}+
\frac{\Delta t}{2}\,F\left(q^{n+1},u(q^{n+1}),v(q^{n+1})\right)$, where
$q^{n+1}$, $u(q^{n+1})$, $v(q^{n+1})$ are computed in the extrapolation step.}
\end{algorithmic}
\end{algorithm}

\subsection{Stochastic case\label{sec:SQG}}
The stochastic version of the QG equations~\eqref{eq:pv} is given by~\cite{holm2015variational}:
\begin{align}
\label{eq:SLTpv}
\begin{split}
\dd q_1+\left(\mathbf{u}_1\,dt +{\color{red} \sum\limits^K_{k=1} \xi^k_1 \circ dW^k_t}\right)\cdot\nabla q_1
&=\left(\nu\Delta^2\psi_1-\beta\frac{\partial \psi_1}{\partial x}\right)\,dt
,\\
\dd q_2+\left(\mathbf{u}_2\,dt +{\color{red} \sum\limits^K_{k=1} \xi^k_2 \circ dW^k_t}\right)\cdot\nabla q_2
&=\left(\nu\Delta^2\psi_2-\mu\Delta\psi_2-\beta\frac{\partial \psi_2}{\partial x}\right)\,dt.
\end{split}
\end{align}

The stochastic terms marked in red color is the only
difference from the deterministic QG model~\eqref{eq:pv}, all other equations remain the same as in the deterministic case.
However, the CABARET scheme in the stochastic case differs from the deterministic version
and therefore its use can only be justified if it is consistent with the stochastic QG model.
In other words, the CABARET scheme should be in the Stratonovich form.

\subsubsection{Numerical method\label{sec:2d_qg_method_stochastic}}
The CABARET scheme for the stochastic QG system~\eqref{eq:SLTpv} 
is given by Algorithm~\ref{alg:SCABARET} (with the stochastic terms highlighted in red).
To the best of our knowledge, the CABARET scheme has not been applied to the stochastic QG equations, and 
is used in this work for the first time.

\begin{algorithm}
\caption{The CABARET scheme for the stochastic QG system}
\label{alg:SCABARET}
\begin{algorithmic}
\STATE{\underline{\bf Predictor}}
\vspace*{-0.75cm}
\STATE{\begin{equation}
\begin{split}
q^{n+\frac12}_{i+\frac12,j+\frac12}=q^n_{i+\frac12,j+\frac12}&+
\frac{\Delta t}{2}\,F\left(q^n,u(q^n),v(q^n)\right)+\Delta t\, F_{\beta}\left(v^n,v^{n-1}\right)
+\Delta t\, F_{\rm visc}\left(\psi\left(q^{n+\frac12}\right)\right)\\
&+{\color{red}\sum\limits^{K}_{k=1}\left(G_k\left(q^n\right)+G_{k,\beta}\right)\frac{\Delta W_k}{2}},
\end{split}\label{eq:s_pred}
\end{equation}}

\STATE{\quad ${\color{red}G_k(q^n)=-\left(\Delta_x\left[(\xi^u_kq^n)_{i,j+\frac12}\right]+
\Delta_y\left[(\xi^v_kq^n)_{i+\frac12,j}\right]\right)}$,\qquad $\displaystyle {\color{red}G_{k,\beta}=3R^n-R^{n-1}},\,
{\color{red}R^n=-\frac{\beta}{2}\left((\xi^u_k)_{i+\frac12,j+1}+(\xi^v_k)_{i+\frac12,j}\right)}$.}\\

\STATE{\quad The forcing term 
\vspace*{-0.125cm}
$$F_{\rm visc}\left(\psi\left(q^{n+\frac12}\right)\right)=\nu\left(\Delta^2\psi_l\right)^{n+\frac12}_{i+\frac12,j+\frac12}-
\delta_{2l}\,\mu\left(\Delta\psi_l\right)^{n+\frac12}_{i+\frac12,j+\frac12},\, l=1,2$$ 
\quad is added in the prediction step after the elliptic problem is solved.}\\

\STATE{\quad\bf Solve the elliptic system of equations with respect to $(\psi_1)^{n+\frac12}_{i+\frac12,j+\frac12}$ and $(\psi_2)^{n+\frac12}_{i+\frac12,j+\frac12}$}\\

\STATE{\qquad$\displaystyle (q_1)^{n+\frac12}_{i+\frac12,j+\frac12}=\left(\Delta\psi_1\right)^{n+\frac12}_{i+\frac12,j+\frac12}+s_1\left(\psi_{[21]}\right)^{n+\frac12}_{i+\frac12,j+\frac12},\quad
(q_2)^{n+\frac12}_{i+\frac12,j+\frac12}=\left(\Delta\psi_2\right)^{n+\frac12}_{i+\frac12,j+\frac12}+s_2\left(\psi_{[12]}\right)^{n+\frac12}_{i+\frac12,j+\frac12}\,.$}\\

\STATE{\quad\bf Calculate}

\STATE{\qquad$\displaystyle \psi^{n+\frac12}_{ij}=\frac14\left(\psi^{n+\frac12}_{i+\frac12,j+\frac12}+
\psi^{n+\frac12}_{i+\frac12,j-\frac12}+\psi^{n+\frac12}_{i-\frac12,j+\frac12}+\psi^{n+\frac12}_{i-\frac12,j-\frac12}\right).$}\\

\STATE{\quad\bf Update velocity components at the cell faces}
\STATE{\qquad $\displaystyle u^{n+\frac12}_{i,j+\frac12}=\Delta_y\left[\psi^{n+\frac12}_{ij}\right],\quad 
\left(v_l\right)^{n+\frac12}_{i+\frac12,j}=-\Delta_x\left[\psi^{n+\frac12}_{ij}\right].$}\\

\STATE{\bf\underline{Extrapolator}}
\STATE{\quad$\displaystyle u^{n+1}_{i,j+\frac12}=\frac32u^{n+\frac12}_{i,j+\frac12}-\frac12 u^{n-\frac12}_{i,j+\frac12}\,,\quad
v^{n+1}_{i,j+\frac12}=\frac32v^{n+\frac12}_{i,j+\frac12}-\frac12 v^{n-\frac12}_{i,j+\frac12}$.}\\

\STATE{\quad$\displaystyle q^{n+1}_{i+1,j+\frac12}=2q^{n+\frac12}_{i+\frac12,j+\frac12}-q^n_{i,j+\frac12}$\quad 
if $\displaystyle u^{n+1}_{i+1,j+\frac12}+{\color{red}\Xi^u_{i+1,j+\frac12}}\ge0$; 
\quad$\displaystyle q^{n+1}_{i,j+\frac12}=2q^{n+\frac12}_{i+\frac12,j+\frac12}-q^n_{i+1,j+\frac12}$\quad 
if $\displaystyle u^{n+1}_{i,j+\frac12}+{\color{red}\Xi^u_{i,j+\frac12}}<0$.}\\

\STATE{\quad$\displaystyle q^{n+1}_{i+\frac12,j+1}=2q^{n+\frac12}_{i+\frac12,j+\frac12}-q^n_{i+\frac12,j}$\quad 
if $\displaystyle v^{n+1}_{i+\frac12,j+1}+{\color{red}\Xi^v_{i+\frac12,j+1}}\ge0$; \quad$\displaystyle q^{n+1}_{i+\frac12,j}=2q^{n+\frac12}_{i+\frac12,j+\frac12}-q^n_{i+\frac12,j+1}$\quad 
if $\displaystyle v^{n+1}_{i+\frac12,j}+{\color{red}\Xi^v_{i+\frac12,j}}<0$.}\\

\STATE{\quad ${\color{red}\Xi^u_{ij}=\sum\limits^{K}_{k=1}\left(\xi^u_k\right)_{ij}\Delta W_k},\quad 
{\color{red}\Xi^v_{ij}=\sum\limits^{K}_{k=1}\left(\xi^v_k\right)_{ij}\Delta W_k}$.}

\STATE{\quad\bf Correction of the computed cell-face PV anomaly values}
\STATE{\qquad If $q^{n+1}_{i,j+\frac12}>M^{n+1}_{i,j+\frac12}\Rightarrow q^{n+1}_{i,j+\frac12}=M^{n+1}_{i,j+\frac12}$;\quad
If $q^{n+1}_{i,j+\frac12}<m^{n+1}_{i,j+\frac12}\Rightarrow q^{n+1}_{i,j+\frac12}=m^{n+1}_{i,j+\frac12}$.}\\

\STATE{\qquad If $q^{n+1}_{i+\frac12,j}>M^{n+1}_{i+\frac12,j}\Rightarrow q^{n+1}_{i+\frac12,j}=M^{n+1}_{i+\frac12,j}$;\quad
If $q^{n+1}_{i+\frac12,j}<m^{n+1}_{i+\frac12,j}\Rightarrow q^{n+1}_{i+\frac12,j}=m^{n+1}_{i+\frac12,j}$.}\\

\STATE{\qquad $\text{If}\,\, u^{n+1}_{i+1,j+\frac12}+{\color{red}\Xi^u_{i+1,j+\frac12}}\ge0\quad
\left\{
\begin{array}{ll}
M^{n+1}_{i+1,j+\frac12}=\max\left(q^n_{i,j+\frac12} ,q^n_{i+\frac12,j+\frac12},q^n_{i+1,j+\frac12}\right)+\tau Q^n_{i+\frac12,j+\frac12},\\
m^{n+1}_{i+1,j+\frac12}=\min\left(q^n_{i,j+\frac12} ,q^n_{i+\frac12,j+\frac12},q^n_{i+1,j+\frac12}\right)+\tau Q^n_{i+\frac12,j+\frac12}.\\
\end{array}
\right.
$}\\

\STATE{}

\STATE{\qquad
$\text{If}\,\, u^{n+1}_{i,j+\frac12}+{\color{red}\Xi^u_{i,j+\frac12}}<0\quad
\left\{
\begin{array}{ll}
M^{n+1}_{i,j+\frac12}=\max\left(q^n_{i,j+\frac12} ,q^n_{i+\frac12,j+\frac12},q^n_{i+1,j+\frac12}\right)+\tau Q^n_{i+\frac12,j+\frac12},\\
m^{n+1}_{i,j+\frac12}=\min\left(q^n_{i,j+\frac12} ,q^n_{i+\frac12,j+\frac12},q^n_{i+1,j+\frac12}\right)+\tau Q^n_{i+\frac12,j+\frac12}.\\
\end{array}
\right.
$}

\STATE{\qquad $\displaystyle Q^{n+\frac12}_{i+\frac12,j+\frac12}=\frac{q^{n+\frac12}_{i+\frac12,j+\frac12}-
q^n_{i+\frac12,j+\frac12}}{\Delta t/2}+\frac12\left(\left(u^{n+1}_{i+1,j+\frac12}+{\color{red}\Xi^u_{i+1,j+\frac12}}\right)+\left(u^{n+1}_{i,j+\frac12}+{\color{red}\Xi^u_{i,j+\frac12}}\right)\right)
\Delta_x\left[q^n_{i,j+\frac12}\right]$.}\\

\STATE{\bf\underline{Corrector}}
\vspace*{-0.5cm}
\STATE{\begin{equation} q^{n+1}_{i+\frac12,j+\frac12}=q^{n+\frac12}_{i+\frac12,j+\frac12}+
\frac{\Delta t}{2}\,F\left(q^{n+1},u(q^{n+1}),v(q^{n+1})\right)+
{\color{red}\sum\limits^{K}_{k=1}\left(G_k\left(q^{n+1}\right)+G_{k,\beta}\right)\frac{\Delta W_k}{2}}\,,\label{eq:s_corr}\end{equation}
where $q^{n+1}$, $u(q^{n+1})$, $v(q^{n+1})$ are computed in the extrapolation step.}

\end{algorithmic}
\end{algorithm}

In order to show that the CABARET scheme is consistent with the stochastic QG model, 
we rewrite the scheme as the improved Euler method (also known as Heun's method)~\cite{KloedenPlaten1999},
\begin{equation*}
\begin{split}
x^*= & x^n+\Delta t f(x^n) + {\color{red}\Delta W g(x^n)},\\
x^{n+1}= & x^n+\frac{\Delta t}{2} (f(x^n)+f(x^*)) + {\color{red}\frac{\Delta W}{2} (g(x^n)+g(x^*))},
\end{split}
\label{eq:Heun}
\end{equation*}
which solves stochastic differential equations (SDEs) in the form of Stratonovich.

In doing so, we omit the space indices for the potential vorticity anomaly $q$ to emphasize the functional dependence on $q$, and introduce an extra variable 
$$q^*=2q^{n+\frac12}-q^n,$$ 
which allows to recast~\eqref{eq:s_pred} and~\eqref{eq:s_corr} (see Algorithm 2) in the form
\begin{subequations}
\begin{equation}
q^*=q^n+
\Delta t\,F\left(q^n,u(q^n),v(q^n)\right)+2\Delta t\, F_{\beta}\left(v^n,v^{n-1}\right)
+2\Delta t\, F_{\rm visc}\left(\psi\left(q^{n+\frac12}\right)\right)\\
+{\color{red}\sum\limits^{K}_{k=1}\left(G_k(q^n)+G_{k,\beta}\right){\Delta W_k}},\\
\label{eq:s_pred2}
\end{equation}
\begin{equation}
q^{n+1}=\frac{q^*+q^n}{2}+
\frac{\Delta t}{2}\,F\left(q^*,u(q^*),v(q^*)\right)+{\color{red}\sum\limits^{K}_{k=1}\left(G_k\left(q^*\right)+G_{k,\beta}\right)\frac{\Delta W_k}{2}}.
\label{eq:s_corr2}
\end{equation}
\end{subequations}

Substitution of~\eqref{eq:s_pred2} into~\eqref{eq:s_corr2} and~\eqref{eq:s_pred} into the forcing term $F_{\rm visc}\left(\psi\left(q^{n+\frac12}\right)\right)$ leads to
\begin{equation}
\begin{split}
q^{n+1}=q^n
&+\frac{\Delta t}{2}\,\left[F(q^n,u(q^n),v(q^n))+F(q^n+{\color{red}O_1(\Delta W_k)},u(q^n+{\color{red}O_1(\Delta W_k)}),v(q^n+{\color{red}O_1(\Delta W_k)}))\right]\\
&+\Delta t\,\left[F_{\beta}\left(v^n,v^{n-1}\right)+F_{\rm visc}\left(\psi\left(q^n+{\color{red}O_2(\Delta W_k)}\right)\right)\right]\\
&+{\color{red}\sum\limits^{K}_{k=1}\left(G_k\left(q^n\right)+G_{k,\beta}\right)\frac{\Delta W_k}{2}}
+{\color{red}\sum\limits^{K}_{k=1}\left(G_k\left(q^n+O_1(\Delta W_k)\right)+G_{k,\beta}\right)\frac{\Delta W_k}{2}},
\end{split}
\label{eq:s_pred_into_corr}
\end{equation}
where
\begin{equation}
\begin{split}
O_1(\Delta W_k):=&\Delta t\,F\left(q^n,u(q^n),v(q^n)\right)+2\Delta t\, F_{\beta}\left(v^n,v^{n-1}\right)\\
+&2\Delta t\,F_{\rm visc}\left(\psi\left(q^n+{\color{red}O_2(\Delta W_k)}\right)\right)
{\color{red}+\sum\limits^{K}_{k=1}\left(G_k\left(q^n\right)+G_{k,\beta}\right)\Delta W_k}\,,
\end{split}
\nonumber
\end{equation}
and
\begin{equation}
O_2(\Delta W_k):=\frac{\Delta t}{2}\,F\left(q^n,u(q^n),v(q^n)\right)
+\Delta t\, F_{\beta}\left(v^n,v^{n-1}\right)+{\color{red}\sum\limits^{K}_{k=1}\left(G_k\left(q^n\right)+G_{k,\beta}\right)\frac{\Delta W_k}{2}}\,.
\nonumber
\end{equation}

Retaining the terms up to order $\Delta t$ in~\eqref{eq:s_pred_into_corr} we get
\begin{equation}
\begin{split}
q^{n+1}
=q^n
&+\Delta t\,\left[F(q^n,u(q^n),v(q^n))+F_{\beta}\left(v^n,v^{n-1}\right)+F_{\rm visc}\left(\psi(q^n)\right)\right]\\
&+{\color{red}\sum\limits^{K}_{k=1}\left(G_k\left(q^n\right)+G_{k,\beta}\right)\Delta W_k+
{\color{red}\sum\limits^{K}_{k_1=1}\sum\limits^{K}_{k_2=1}G_{k_1}\left(G_{k_2}(q^n)+G_{k_2,\beta}\right)\frac{\Delta W_{k_1}\Delta W_{k_2}}{2}}}+H.O.T.\,,\\
\end{split}
\label{eq:s_pred_into_corr2}
\end{equation}
where $G_{\beta}$ does not depend on $q^n$, and $H.O.T.$ denotes higher order terms. 
Thus we have shown that the CABARET scheme is in Stratonovich form up to order $(\Delta t)^{3/2}$.

\section{Consistency in time of the stochastic CABARET scheme\label{sec:2d_qg_consistency}}
In this section we prove that the stochastic CABARET scheme~\eqref{eq:s_pred_into_corr2} is consistent with the stochastic QG equation~\eqref{eq:SLTpv}
in the mean square sense in time, since its consistency in space is guaranteed by its second order approximation~\cite{Karabasov_et_al2009}.
We consider a Stratonovich process $q=q(t,\mathbf{x})$, $\mathbf{x}=(x,y)$ satisfying the SPDE
\begin{equation*}
\dd q=a_t\diff t+{\color{red}\sum\limits^{K}_{i=1}b_{i,t}\circ\diff W_{i,t}},\quad a_t:=F(q^n,u(q^n),v(q^n))+F_{\beta}+F_{\rm visc}\left(\psi(q^n)\right),\quad b_{i,t}:=G_i(q^n)+G_{i,\beta},
\label{eq:Stratonovich1} 
\end{equation*}
and rewrite it in the It\^{o} form
\begin{equation*}
\dd q=a_t\diff t+{\color{red}\sum\limits^{K}_{i=1}b_{i,t}\diff W_{i,t}}+\frac12\sum\limits^{K}_{i=1}b_{i,t}(b_{i,t})\diff t\,,
\label{eq:Ito1_} 
\end{equation*}
or alternatively
\begin{equation}
\dd q=q_d\diff t+{\color{red}\sum\limits^{K}_{i=1}q^i_{s,t}\diff W_{i,t}}
\label{eq:Ito1} 
\end{equation}
with the stochastic and deterministic parts defined as $\displaystyle q_d:=a_t+\frac12\sum\limits^{K}_{i=1}b_{i,t}(b_{i,t})$ and $\displaystyle q^i_{s,t}:=b_{i,t}$, respectively.

We define consistency for SPDE~\eqref{eq:Ito1} as follows
\begin{definition}
We say that a discrete time-space approximation $q^n=q^n_d+q^n_s$ of $q=q_d+q_s$ with 
the time step $\Delta t$ and space steps $\Delta\mathbf{x}=(\Delta x_1,\Delta x_2,\ldots,\Delta x_d)$ 
is consistent in mean square of order $\alpha>1$ and $\beta>1$ in time and space with respect to~\eqref{eq:Ito1} if there exists a nonnegative function 
$c=c((\Delta t)^\alpha,(\Delta\mathbf{x})^\beta)$ with 
$\lim\limits_{\substack{\Delta t\rightarrow0\\ \Delta\mathbf{x}\rightarrow0}}c((\Delta t)^\alpha,(\Delta\mathbf{x})^\beta)=0$ such that
\begin{equation*}
\mathbb{E}\left[\left\|q_s- q^n_s \right\|^2_{L^2(\Omega)}\right]\le c((\Delta t)^\alpha,(\Delta\mathbf{x})^\beta)\,,\qquad 
\mathbb{E}\left[\left\|q_d- q^n_d \right\|^2_{L^2(\Omega)}\right]\le c((\Delta t)^\alpha,(\Delta\mathbf{x})^\beta)
\label{eq:consitency_spde1} 
\end{equation*}
for all fixed values $q^n$, time $n=0,1,2,\ldots$ and space indices.
\end{definition}

Since our focus in this section is on consistency in time, we have to prove that the following estimation holds:
\begin{equation}
\mathbb{E}\left[\left\|q_s- q^n_s \right\|^2_{L^2(\Omega)}\right]\le c((\Delta t)^\alpha)\,.
\label{eq:consitency_spde2} 
\end{equation}

\begin{theorem}
Assuming that there exists a constant $\widetilde{C}>0$ such that the following assumptions hold
\begin{enumerate}[label={\bf A\arabic*.}]
 \item $\mathbb{E}\left[\left\|a_r-a_s\right\|_{L^2(\Omega)}\right]\le \widetilde{C}\sqrt{r-s}$,
 \item $\mathbb{E}\left[\left\|\sum\limits^{K}_{i=1}(b_{i,r}-b_{i,s})\right\|_{L^2(\Omega)}\right]\le \widetilde{C}\sqrt{r-s}$,
 \item $\mathbb{E}\left[\left\|\sum\limits^{K}_{i=1}\sum\limits^{K}_{j=1}b_{i,s}(b_{j,s})\right\|_{L^2(\Omega)}\right]\le \widetilde{C}$, for $i,j=1,2,\ldots,m$,
 \item $\mathbb{E}\left[\left\|\sum\limits^{K}_{i=1}(b_{i,r}(b_{i,r})-b_{i,s}(b_{i,s}))\right\|_{L^2(\Omega)}\right]\le \widetilde{C}\sqrt{r-s}$, 
 \item $\mathbb{E}\left[\left\|{H.O.T.}\right\|\right]\le \widetilde{C}(r-s)^{3/2}$,
\end{enumerate}
with $\left|r-s\right|\le \Delta t$, the stochastic CABARET scheme~\eqref{eq:s_pred_into_corr2} is consistent in mean square with $c(\Delta t)=(\Delta t)^2$.
\end{theorem}
\begin{proof}
Integration of~\eqref{eq:Ito1}  with respect to time over the interval $[s,t]$ gives
\begin{equation}
q_t=q_s+\int\limits^t_s a_r\,dr+{\color{red}\int\limits^t_s \sum\limits^{K}_{i=1}b_{i,r}\, dW_{i,r}}+\frac12\int\limits^t_s 
\sum\limits^{K}_{i=1}b_{i,r}(b_{i,r}) dr\,.
\label{eq:int_Ito1}
\end{equation}

\noindent
Substitution of~\eqref{eq:s_pred_into_corr2} and~\eqref{eq:int_Ito1} into~\eqref{eq:consitency_spde2} leads to
\begin{equation}
\begin{split}
&\mathbb{E}\left[\left\|
\int\limits^t_s a_r\,dr+{\color{red}\int\limits^t_s \sum\limits^{K}_{i=1}b_{i,r}\, dW_{i,r}}+\frac12\int\limits^t_s \sum\limits^{K}_{i=1} b_{i,r}(b_{i,r}) dr\right.\right. \\
&-\left.\left.
\left(a_t\Delta t
+{\color{red}\sum\limits^{K}_{i=1} b_{i,s}\, \Delta W_{i,s}}
+{\color{red}\frac12\sum\limits^{K}_{i,j=1} b_{i,s}(b_{j,s}) \Delta W_{i,s}\Delta W_{j,s}}\right)+{H.O.T.}\right\|^2_{L^2(\Omega)}
\right]\le c(\Delta t).
\end{split}
\label{eq:con_spde3}
\end{equation}
By combining the terms in~\eqref{eq:con_spde3}, we get 
\begin{equation}
\mathbb{E}\left[\left\|A+B+C\right\|^2_{L^2(\Omega)}\right]\le c(\Delta t),
\label{eq:con_spde4}
\end{equation}
where
\begin{equation*}
A:=\int\limits^t_s (a_r-a_s),dr\quad B:={\color{red}\int\limits^t_s \sum\limits^{K}_{i=1}(b_{i,r}-b_{i,s})\, dW_{i,r}},\quad 
C:=C_1-C_2-C_3,
\end{equation*}
with 
\begin{equation*}
C_1:=\frac12\int\limits^t_s \sum\limits^{K}_{i=1}(b_{i,r}(b_{i,r})-b_{i,s}(b_{i,s}))\, dr,\quad
C_2:=\frac12\sum\limits^{K}_{i=1}b_{i,s}(b_{i,s})({\color{red}(\Delta W_i)^2}-\Delta t),\quad
C_3:={\color{red}\frac12\sum\limits^{K}_{i\ne j}b_{i,s}(b_{j,s})\Delta W_i\Delta W_j}\,.
\end{equation*}

\noindent
Applying the triangle and Young's inequalities to~\eqref{eq:con_spde4} we arrive at 
\begin{equation*}
\mathbb{E}\left[\left\|A+B+C\right\|^2_{L^2(\Omega)}\right]\le 3\mathbb{E}\left[\left\|A\right\|^2_{L^2(\Omega)}+
\left\|B\right\|^2_{L^2(\Omega)}+\left\|C\right\|^2_{L^2(\Omega)}+\widetilde{C}^2(\Delta t)^3\right].
\label{eq:con_spde5}
\end{equation*}

\noindent
Using {\bf A2}, the Cauchy--Schwarz inequality and {\bf A1}, we estimate the first term as
\begin{equation*}
\mathbb{E}\left[\left\|A\right\|^2_{L^2(\Omega)}\right] \le\Delta t\, \mathbb{E}\left[ \int\limits^t_s \left\|a_r-a_s\right\|^2_{L^2(\Omega)}\, dr \right] \le \frac{\widetilde{C}^2}{2}(\Delta t)^3.
\label{eq:A_estimate1} 
\end{equation*}

\noindent
Estimation of the second term is given by
\begin{equation*}
\begin{aligned}
 \mathbb{E}\left[\left\|B\right\|^2_{L^2(\Omega)}\right]=& \int\limits_{\Omega}\mathbb{E}\left [{\color{red}\left(\int\limits^t_s\sum\limits^{K}_{i=1}(b_{i,r}-b_{i,s})\, dW_{i,r}\right)^2} \right]\, d\Omega  &&\text{(using the It\^{o} isometry)}\\
=& \mathbb{E}\left [\int\limits_{\Omega}\int\limits^t_s\left(\sum\limits^{K}_{i=1}(b_{i,r}-b_{i,s})\right)^2\, dr\, d\Omega \right]  && \text{(the Cauchy--Schwarz inequality leads to)}\\
\le& \Delta t\, \mathbb{E}\left [\int\limits^t_s\left\|\sum\limits^{K}_{i=1}(b_{i,r}-b_{i,s})\right\|^2_{L^2(\Omega)}\, dr \right]  \le \frac{\widetilde{C}^2}{2}(\Delta t)^3 &&  \text{(using {\bf A2}).}\\
\end{aligned}
\label{eq:B_estimate1} 
\end{equation*}

\noindent
To estimate the term $C$ in \eqref{eq:con_spde4}, we use the triangle inequality to get
\begin{equation*}
\mathbb{E}\left[\left\|C\right\|^2_{L^2(\Omega)}\right] \le \mathbb{E}\left [ \left\|C_1\right\|^2_{L^2(\Omega)}+\left\|C_2\right\|^2_{L^2(\Omega)}+\left\|C_3\right\|^2_{L^2(\Omega)} \right]\,,
\label{eq:C_estimate1} 
\end{equation*}
and then separately estimate each term on the right hand side. 

Applying the Cauchy--Schwarz inequality and {\bf A4} to $C_1$, we get the following estimation
\begin{equation*}
\mathbb{E}\left[ \left\|C_1\right\|^2_{L^2(\Omega)} \right] \le\frac{\Delta t}{2}\mathbb{E}\left[\int\limits_{\Omega} 
\left\| \sum\limits^{K}_{i=1}(b_{i,r}(b_{i,r})-b_{i,s}(b_{i,s})) \right\|^2_{L^2(\Omega)}\, d\Omega \right] \le \frac{\widetilde{C}^2}{8}(\Delta t)^3.
\label{eq:C1_estimate1} 
\end{equation*}

\noindent
The term $C_2$ is estimated as
\begin{equation*}
\begin{aligned}
\mathbb{E}\left[ \left\|C_2\right\|^2_{L^2(\Omega)} \right]=& \int\limits_{\Omega} \mathbb{E}\left[\left(\frac12 \sum\limits^{K}_{i=1}(b_{i,s}(b_{i,s}))\left({\color{red}(\Delta W_i)^2}-\Delta t\right)\right)^2 \right]\, d\Omega && \\
= & \frac14 \int\limits_{\Omega} \sum\limits^{K}_{i=1}(b_{i,s}(b_{i,s}))^2\, \mathbb{E}\left[ {\color{red}(\Delta W_i)^4}-2{\color{red}(\Delta W_i)^2}\Delta t+(\Delta t)^2 \right]\,d\Omega && \\
= & \frac{(\Delta t)^2}{2} \left\| \sum\limits^{K}_{i=1}(b_{i,s}(b_{i,s}))^2\right\|^2_{L^2(\Omega)} \le \frac{\widetilde{C}^2}{2}(\Delta t)^2  &&  \text{(using {\bf A3}).}\\
\end{aligned}
\label{eq:C2_estimate1} 
\end{equation*}

\noindent
Using {\bf A3} for $C_3$ leads to
\begin{equation*}
\mathbb{E}\left[ \left\|C_3\right\|^2_{L^2(\Omega)} \right]= \frac14 \int\limits_{\Omega} \sum\limits^{K}_{i\ne j}(b_{i,s}(b_{i,s}))^2\, \mathbb{E}\left[ {\color{red}(\Delta W_i)^2}\right] \mathbb{E}\left[ {\color{red}(\Delta W_j)^2}\right]\, d\Omega 
= \frac{(\Delta t)^2}{4} \left\|\sum\limits^{K}_{i\ne j}(b_{i,s}(b_{i,s}))\right\|^2_{L^2(\Omega)} \le  \frac{\widetilde{C}^2}{4}(\Delta t)^2. 
\label{eq:C3_estimate1} 
\end{equation*}

\noindent
Finally, we arrive at the following estimation
\begin{equation*}
\mathbb{E}\left[\left\|A+B+C\right\|^2_{L^2(\Omega)}\right]\le C^*\left((\Delta t)^2+(\Delta t)^3\right)\le C^*(\Delta t)^2,\quad C^*>0,
\label{eq:con_spde4_2}
\end{equation*}
which proves the theorem.
\end{proof}
\begin{remark}
Conditions {\bf A1-A5} are satisfied and SPDE~\eqref{eq:Ito1} is well-posed for sufficiently large p for all $T>0$ 
if the stochastic QG equation~\eqref{eq:SLTpv} has a solution in $W^{2p,2}$ 
such that 
$\mathbb{E}\left[\sup\limits_{t\in [0,T]}||q_i||^2_{W^{2p,2}}\right]<\infty$, $i=1,2$.   
\end{remark}

\section{Data assimilation methods}
\label{sec:da_methods}

We find it useful to describe the framework and the data assimilation methodology through the language of nonlinear filtering. For this purpose, let us consider a probability space $(\Omega,\mathcal{F},\mathbb{P})$ on which we define a pair of processes $Z$ and $Y$. The process $Z$ is normally called the signal process, and the process $Y$ models the observational data and is called the observation process. In the context of this work the signal process, also called the true state, is given by the solution of the deterministic QG equation~\eqref{eq:pv} computed on a fine grid $G_f=2049\times 1025$ 
and projected onto a coarse grid, denoted by $G_s$ (details below), by spatially averaging the high-resolution stream 
function $\psi^f$ over the corresponding coarse grid cells.

The  filtering problem consists in computing/approximating the posterior distribution of the signal $Z_t$, denoted by $\pi_t$ given the observations $Y_s$, $s\in [0,t]$. In our context, the observations consist of noisy measurements of the true state recorded at discrete times (every 2 or 4 hours) and
are taken at locations (called weather stations) on a data grid $G_d$ defined below. The data assimilation is performed at these times, which we call the assimilation times.

The most basic particle filter, called the bootstrap particle filter (see section  \ref{sec:bootstrap} for details) uses ensembles of, say, $N$ particles that evolve according to the law of the signal between assimilation times. At the data assimilation times, each particle is weighed according to the likelihood of its position given the new data. A new set of particles is then obtained by sampling $N$ times (with replacement)  from the set of weighted particles. The end result is that particles with high likelihoods (close to the true trajectory) are kept and possibly multiplied, and particles with small likelihoods (that are away from the true trajectory) are eliminated. As a result, the ensemble of particles should stay closer to the true trajectory when compared to the particles that evolve according to the signal distribution.  
The bootstrap particle filter described below uses multiple copies of the signal which, in our case, would require the resolution of the deterministic QG equation~\eqref{eq:pv} on a fine grid $G_f=2049\times1025$. Each run of the particle filter at such a high
resolution is very expensive computationally (one data assimilation step takes approximately 15 minutes). Taking into account that we 
assimilate data over thousands steps, the computational resources needed are too large.
For this reason we replace the true state with a proxy. We use a process $X$ defined on the same probability space whose sample paths are a lot cheaper to simulate. In our case, the process $X$ will be    
the solution of the stochastic QG equation~\eqref{eq:SLTpv} computed
on the (signal) grid $G_s$ (each run of the process $X$ requires around 20 seconds). This replacement is a form of model reduction: we reduce the dimension of the underlying state from $G_f=2049\times1025$ to $G_s=129\times65$.
This model reduction is critical to successful implementation of the data assimilation procedure. It is also rigorously justified as we explain now.

The posterior distribution $\pi_t$ depends continuously on two constituents: the (prior) distribution of the signal and the observation data. That means that if we replace the original signal distribution  with a proxy distribution, we will still obtain a good approximation of $\pi_t$ provided the original and the proxy distributions are close to each other in some suitably chosen topology on the space of distribution. 

For the current work, the way in which we ensure that two distributions (original and proxy) remains close to each other is by adding the right type of stochasticity to the model and ``in the right directions". This is done through 
the Stochastic Advection by Lie Transport (SALT) approach~\cite{holm2015variational}. The stochasticity is calibrated to match the fluctuations of $X$ as  explained in see~\cite{CCHWS2019_1,CCHWS2019_2}. We emphasize that one does not seek a pathwise approximation of the true state, but only an approximation of its distribution.

In our case, the true state is deterministic, and the process $X$ is random. As we saw in~\cite{CCHWS2019_1,CCHWS2019_2}, one can visualise the distribution of $X$ through ensembles of particles with trajectories that are solutions of the stochastic QG equation~\eqref{eq:SLTpv} computed
on the grid $G_s$ and driven by independent families of Brownian motions. In the language of uncertainty quantification, the difference between the two distributions is  
interpreted as the ``uncertainty of the model". Typically, the ensemble of particles is a spread "around" the true trajectory, the size of the spread measuring the (model) uncertainty. To visualise this one can look at projections of the true trajectory and the ensembles of particles at various grid points. Of course, the more refined the grid $G_s$ is, the closer the two distributions are, and the smaller the spread. However, refining the grid $G_s$ leads to an increase in the computational effort of generating the particle trajectories. One of the roles of data assimilation is to reduce the spread (the uncertainty) without refining the grid.  

The average position of the ensemble of particles obtained through the data assimilation, denoted by $\widehat{Z}_t$, is a pointwise estimate of the true state $Z_t$, whilst the spread of the ensemble is a measure of the approximation error $Z_t-\widehat{Z}_t$. As explained in the introduction, the data assimilation methodology presented here is asymptotically consistent: the empirical distribution of the particles converges, as $N\mapsto \infty$ to the posterior distribution $\pi_t$~\cite{Crisan2002ASO}. As a consequence, $\widehat{Z}_t$ converges to the conditional expectation of the true state $Z_t$ given the observations, and the empirical covariance matrix converges to the conditional expectation of  $(Z_t-\widehat{Z}_t)(Z_t-\widehat{Z}_t)^{T}$ given the observations.\footnote{$(Z_t-\widehat{Z}_t)^{T}$ is the transpose of $(Z_t-\widehat{Z}_t)$}. This is true if the particles evolve according to the original distribution of the signal. In our case we use a proxy distribution, so the limit will be an approximation of $\pi_t$, the difference between this approximation and $\pi_t$ being controlled by the choice of the signal grid $G_s$.       

The data assimilation methodology described below consists in combination between the bootstrap particle filter and three additional procedures: nudging, tempering and jittering. The bootstrap particle filter cannot be used on its own to solve the data assimilation problem. The reason is that the particle likelihoods vary wildly from each other. That is because the particle themselves stray away from the true state rapidly and in different directions. This is reflected through observation data. One particle likelihood or a small number of such likelihoods will become much higher than the rest, and only the corresponding particle(s) will be selected and multiplied. This will not offer a good representation of the posterior distribution. As we will explain below the additional procedures will eliminate this effect ensuring a reasonably spread set of particles.

In the next subsections we study how each of these individual procedures  
influences the accuracy of the estimation $\widehat{Z}_t$ and the quality of the 
forecast given by the stochastic QG model. In order to study how the dimension of the observation process $Y$ (the number of weather stations) affects data assimilation,
we consider two different data grids $G_d=\{4\times4,8\times4\}$. We also study stochastic solutions on 
two different signal grids $G_s=\{129\times65,257\times129\}$ in order to highlight the effect of more accurate proxy distributions on the results. 

As stated above, we will use ensembles $\mathcal{S}$ of solutions of the
stochastic QG equation~\eqref{eq:SLTpv} driven by independent realizations of the Brownian noise $W$. For the purpose of this paper, the size of the ensemble is taken to be $N=100$ and the number of Brownian motion (independent sources of stochasticity) is taken to be $K=32$; as already stated, the elements of the ensembles will be called particles. It was numerically shown in~\cite{CCHWS2019_2} that $N=100$ and $K=32$ is enough to reasonably approximate the fluctuations of the original distribution. Through  numerical experiments, we showed that the spread of the ensemble will not increase substantially by
taking more particles and/or sources of noise (Brownian motions).

The observations data $Y_t$ is, in our case, an $M$-dimensional process that consists of noisy measurements of the velocity field $\mathbf{u}$ taken at a point belonging to the data grid $G_d$: 
$$
Y_t:={\rm P}^s_d(Z_t)+\eta,
$$
where $\mathcal{P}^s_d:G_s\rightarrow G_d$ is a projection operator
from the signal grid $G_s$ to the data grid $G_d$,
$\eta=\mathcal{N}(\mathbf{0},I_\sigma)$ is a normally distributed random vector, with mean vector  $\mathbf{0}=(0,\ldots,0)$ and diagonal covariance matrix $I_\sigma=diag(\sigma_1^2,\ldots,\sigma_M^2)$. Rather than choosing an arbitrary  $\sigma=(\sigma_1,\ldots,\sigma_M)$ for the standard deviation of the noise, we use the standard deviation of the velocity field computed
over the coarse grid cell of the signal grid.

We introduce the likelihood-weight function 
\begin{equation}
 \mathcal{W}(\mathbf{X},\mathbf{Y})=
 \exp\left(-\frac12\sum\limits^M_{i=1}\left\|\frac{{\rm P}^s_d (X_i)-Y_i}{\sigma_i}\right\|^2_2\right),
 \label{eq:prob_function}
\end{equation}
with $M$ being the number of grid points (weather stations). 
In order to measure the variability of the weights~\eqref{eq:prob_function} of particles we use the effective sample size:
\begin{equation}
 {\rm ESS}(\overline{\mathbf{w}})=\left(\sum\limits^N_{i=1}\left(\overline{w}_i\right)^2\right)^{-1},\quad
 \overline{\mathbf{w}}:=\mathbf{w}\left(\sum\limits^N_{i=1}w_i\right)^{-1},
 \label{eq:ess}
\end{equation}
which is close to the ensemble size $N$ if the particles have weights that are close to each other, 
and decays to one, as the ensemble degenerates 
(i.e. there are fewer and fewer particles with large weights and the rest have small weights). One should resample for the weighted ensemble if the ESS drops below a given threshold, $N^*$, 
$$
{\rm ESS}<N^*.
$$
We chose $N^*=80$ to be our threshhold.

\subsection{Bootstrap particle filter}
\label{sec:bootstrap}
In this section we consider the most basic particle filter, called the bootstrap particle filter
or Sampling Importance Resampling filter~\cite{DFG2001}. This method works as follows.

Given an initial distribution of particles, each particle is propagated forward according to the stochastic
QG equation. Then, based on partial observations, $\mathbf{Y}_{t_{j+1}}$ of the true state, the weights of 
new particles are computed, and if the effective sample size drops below the critical value $N^*$, the particles 
are resampled to remove particles with small weights. 

\begin{algorithm}
\caption{Bootstrap particle filter}
\begin{algorithmic}
\FOR{$j=0,1,2,\ldots$}
\STATE{Solve\, $\dd q^{(n)}_i+\left(\mathbf{u}^{(n)}_i\,dt +{\color{red} \sum\limits^K_{k=1} \xi^k_i \circ dW^{(n)}_k}\right)\cdot\nabla q^{(n)}_i=F^{(n)}_i\,dt,\quad 
i=1,2;\quad t\in[t_j,t_{j+1}],\quad n\in[1..N].$}
\STATE{Obtain observation $\mathbf{Y}_{t_{j+1}}$ from weather stations}
\STATE{Compute $\overline{\mathbf{w}}:=\mathcal{W}(\mathbf{q}_{t_{j+1}},\mathbf{Y}_{t_{j+1}})$}
\IF{${\rm ESS(\overline{\mathbf{w}})}<N^*$}
\STATE{$\mathbf{q}_{t_{j+1}}:={\bf Resample}(\overline{\mathbf{w}})$}
\ENDIF
\ENDFOR
\end{algorithmic}
\end{algorithm}

For a high dimensional problem such as the one studied in this paper, the effective sample size drops very quickly to 1 as the sample degenerates rapidly. The reason for this is that particles travel very quickly away from the true state, and this is picked up by the observation data (unless the measurement noise is large which is not in our case - the observations are accurate). To counteract this, the resampling procedure would need to be performed unreasonably frequently or a large number of particles would need to be used. 


To maintain the diversity of the ensemble we use instead three additional procedures: \textit{the tempering technique}
and \textit{jittering} based on the Metropolis-Hastings Markov chain Monte Carlo (MCMC) method. We explain each of these procedures in the following sub-sections.

\subsection{Tempering and Jittering \label{sec:tempering}}

We will explain briefly the usage of these two procedures, see \cite{CCHWS2019_3} for further details. The idea behind \textit{tempering} is to artificially flatten the weights  thorough rescaling 
the log likelihoods by a factor $\phi\in(0,1]$, which is called temperature. Once this is done resampling can be applied. This gives a much more diverse ensemble as the ESS will have more reasonable values (the temperature is specifically chosen to ensure this). 
However, some of resulting particles will still have duplicates. To eliminate these, one uses \textit{jittering}. 

\textit{Jittering} is another technique
which improves the diversity of the ensemble by computing new particles which have been duplicated during resampling. There are different ways how to diversify the ensemble. For example, one
can jitter the particles by simply adding some random perturbations to them. However, in this case, 
the perturbed particles are not the solutions of the stochastic QG equation that, in turn, can lead to nonphysical behaviour of 
the model. Instead, we compute new particles by solving the stochastic QG equation driven by the Brownian motion
$\rho W+\sqrt{1-\rho^2}\,d\widetilde{W}$, where $W$ is the original Brownian motion $W$ and $\widetilde{W}$ is a new Brownian motion independent of $W$. The perturbation parameter $\rho$ is chosen so that particles are not placed too far from the original position, yet far enough to ensure the diversity of the sample. In our experiments, we use $\rho=0.9999$. Each new proposal for the position of the particle is then accepted/rejected thorough a standard Metropolis-Hastings method, in which $M_1$ stands for the number of iterations; we choose $M_1=20$. This ensures that the perturbations do not change the sought distribution.

Of course, after the first tempering-jittering cycle has finished,  the particles in the resulting ensemble are samples of the altered distribution which is not what we desire, therefore the procedure is repeated by finding the next temperature value in the range $(\varphi,1]$ that offers a reasonable ESS. This is repeated until the temperature scaling is 1 so that the original distribution is recovered.
The tempering-jittering methodology is given by Algorithm 4 below.

\begin{algorithm}
\caption{Particle Filter with Tempering and MCMC}
\begin{algorithmic}
\FOR{$j=0,1,2,\ldots$}
\STATE{Solve\, $\dd q^{(n)}_i+\left(\mathbf{u}^{(n)}_i\,dt +{\color{red} \sum\limits^K_{k=1} \xi^k_i \circ dW^{(n)}_k}\right)\cdot\nabla q^{(n)}_i=F^{(n)}_i\,dt,\quad 
i=1,2;\quad t\in[t_j,t_{j+1}],\quad n\in[1..N].$}
\STATE{Obtain observation $\mathbf{Y}_{t_{j+1}}$ from weather stations}
\STATE{Compute $\overline{\mathbf{w}}:=\mathcal{W}(\mathbf{q}_{t_{j+1}},\mathbf{Y}_{t_{j+1}})$}
\IF{${\rm ESS(\overline{\mathbf{w}})}<N^*$}
\STATE{Find $p$ such that $ESS(\overline{\mathbf{w}})\ge N^*$, where
$\overline{\mathbf{w}}:=\mathcal{W}^{1/p}(\mathbf{q}_{t_{j+1}},\mathbf{Y}_{t_{j+1}})$}
\FOR{$k=1,2,\ldots,p$}
\STATE{Compute $\overline{\mathbf{w}}:=\mathcal{W}^{\phi_k}(\mathbf{q}_{t_{j+1}},\mathbf{Y}_{t_{j+1}}),\ \phi_k:=\frac{k}{p}$}
\STATE{$\mathbf{q}_{t_{j+1}}:={\bf Resample}(\overline{\mathbf{w}})$}
\FOR{$m=1,2,\ldots,M_1$}
\STATE{${\color{red} \Xi^{(n)}_i:=\sum\limits^K_{k=1} \xi_{k,i} \circ \left(\rho\, dW^{(n)}_{k}+\sqrt{1-\rho^2}\,d\widetilde{W}^{(n)}_{k}\right)},\quad n\in[1..N]$}
\STATE{Solve\, $\dd \widetilde{q}^{(n)}_i+
\left(\widetilde{\mathbf{u}}^{(n)}_i\,dt +{\color{red} \Xi^{(n)}_i}\right)\cdot\nabla \widetilde{q}^{(n)}_i=\widetilde{F}^{(n)}_i\,dt,\quad 
i=1,2;\quad t\in[t_j,t_{j+1}],\quad n\in[1..N]$}
\FOR{n=1,2,\ldots,N}
\STATE{$\alpha:=\left(\mathcal{W}(\widetilde{q}^{(n)}_{t_{j+1}},Y^{(n)}_{t_{j+1}})/\mathcal{W}(q^{(n)}_{t_{j+1}},Y^{(n)}_{t_{j+1}})\right)^{\phi_k}$}
\IF{$1\le\alpha$}
\STATE{$q^{(n)}_{t_{j+1}}:=\widetilde{q}^{(n)}_{t_{j+1}}$}
\ELSIF{$\mathcal{U}[0,1]<\alpha$}
\STATE{$q^{(n)}_{t_{j+1}}:=\widetilde{q}^{(n)}_{t_{j+1}}$}
\ENDIF
\ENDFOR
\ENDFOR
\ENDFOR
\ENDIF
\ENDFOR
\end{algorithmic}
\end{algorithm}

\subsection{Nudging}
\label{sec:nudging}
Tempering combined with jittering is a powerful technique which can correctly narrow the stochastic spread in the presence of informative data, while also 
maintaining the diversity of the ensemble over a long time period. Their combined success depends crucially on the quality of the original sample proposals. This quality is produced by evolving the particles using the SPDE (the proxy distribution) and not the true distribution. To reduce the discrepancy introduced in this way, one can use \textit{nudging}. The idea of nudging is to correct the solution of SPDE~\eqref{eq:SLTpv} 
so as to keep the particles closer to the true state.  To do so, we add a `nudging term' (marked in blue) to SPDE~\eqref{eq:SLTpv},
\begin{align}
 \label{eq:SLTpv_nudge}
\dd q_i(\lambda)+\left(\mathbf{u}_i(\lambda)\,dt +{\color{red} \sum\limits^K_{k=1} \xi^k_i \circ dW^k_t}
+{\color{blue}\sum\limits_{k=1}^K\boldsymbol{\xi}^k_i\lambda_k\,dt}\right)\cdot\nabla q_i(\lambda)
&=F_i\,dt,\quad i=1,2.
\end{align}
Note that $q$ depends on the parameter $\lambda$.
The trajectories of the particles will be solutions of this perturbed SPDE~\eqref{eq:SLTpv_nudge}. 
To account for the perturbation, the particles will have new weights according to Girsanov's theorem, given by 
\begin{equation}
 \mathcal{W}({\mathbf{q}}(\lambda),\mathbf{Y},\lambda)=
 \exp\left(-\left(\displaystyle\left[\frac12\sum\limits_{i=1}^M\Big\|\frac{P^s_d({\mathbf{q}}_{t_{j+1}}(\lambda))-\mathbf{Y}_{t_{j+1}}}{\sigma_i}\Big\|^2_2+
{\color{blue}\int_{t_{j}}^{t_{j+1}}\left(\lambda^2_k\frac{ dt}{2}-\lambda_k dW_k\right)} \right]\right)\right).
 \label{eq:girsanov}
\end{equation}
As explained above, these weights measure the likelihood of the position of the particles given the observation, 
and the last term accounts for the change of probability distribution from $q$ to $q(\lambda)$. It therefore makes sense to choose weights that maximize these likelihoods. In other words, we could look to solve the equivalent minimization problem
\begin{equation}
 \label{eq:min_lambda_nudge}
\min\limits_{\color{blue}\lambda_k,\ k\in[1..K]} 
\displaystyle\left[\frac12\sum\limits_{i=1}^M\Big\|\frac{P^s_d({\mathbf{q}}_{t_{j+1}}(\lambda))-\mathbf{Y}_{t_{j+1}}}{\sigma_i}\Big\|^2_2+
{\color{blue}\int_{t_j}^{t_{j+1}}\left(\lambda^2_k\frac{dt}{2}-\lambda_k dW_k\right)} \right]
\end{equation}
together with~\eqref{eq:SLTpv_nudge}. 
In general this is a challenging nonlinear optimisation problem, especially if one allows the $\lambda_k$'s to vary in time.

To simplify the problem, we perturb only
the corrector stage of the final timestep before $t_{j+1}$. Then the (discrete version of the) minimization problem~\eqref{eq:min_lambda_nudge} 
becomes
\begin{equation}
 \label{eq:min_lambda_nudge discrete}
\min\limits_{\color{blue}\lambda_k,\ k\in[1..K]} 
\displaystyle\left[\frac12\sum\limits_{i=1}^M\Big\|\frac{P^s_d({\mathbf{q}}_{t_{j+1}}(\lambda))-\mathbf{Y}_{t_{j+1}}}{\sigma_i}\Big\|^2_2+
{\color{blue}\sum\limits_{k=1}^K\left(\lambda^2_k\frac{\delta t}{2}-\lambda_k\Delta W_k\right)}\right],
\end{equation}
where $\delta t$ is the time step.
Let us re-write 
\[
{\mathbf{q}}_{t_{j+1}}(\lambda)=A(q_{t_{j+1/2}})+\sum
\limits_{k=1}^{K}B_{k}(\tilde{q}_{t_{j+1}})(\Delta W_{k}+\lambda _{k}\delta t)
\],
where $q_{t_{j+1/2}}$ and $\tilde{q}_{t_{j+1}}$ are computed in the prediction and the extrapolation steps, respectively (see Algorithm 2 for detail).

We can then re-write the minimisation problem \eqref{eq:min_lambda_nudge discrete} as
\begin{equation}
 \label{eq:min_lambda_nudge discrete'}
\min\limits_{\color{blue}\lambda_k,\ k\in[1..K]}
\mathcal{V}({\mathbf{q}}(\lambda ),\mathbf{Y},\lambda ),
\end{equation}
where
\[
\mathcal{V}({\mathbf{q}}(\lambda ),\mathbf{Y},\lambda )=
Q+Q_{1}(\lambda )+
Q_{2}(\lambda ,\Delta W_{1},...,\Delta W_{K}),
\]
with
\[
Q:=\frac{1}{2}\sum\limits_{i=1}^{M}\left\vert \left\vert \frac{P_{d}^{s}(A(q_{t_{j+1/2}}))-\mathbf{Y}_{t_{j+1}}}{\sigma _{i}}\right\vert \right\vert
_{2}^{2},
\]%
and 
\begin{eqnarray*}
Q_{1}(\lambda ) &:=&\frac{1}{2}\sum\limits_{i=1}^{M}\left\vert \left\vert
\sum\limits_{k=1}^{K}\frac{P_{d}^{s}\left( B_{k}(\tilde{q}_{t_{j+1}})\right)\lambda _{k}\delta t}{\sigma _{i}}\right\vert \right\vert _{2}^{2}\\
&+& 
\sum\limits_{i=1}^{M}\frac{1}{\sigma _{i}^{2}}\left\langle \left(
P_{d}^{s}(A(q_{t_{j+1/2}}))-\mathbf{Y}_{t_{j+1}}
,\sum\limits_{k=1}^{K}P_{d}^{s}\left( B_{k}(\tilde{q}_{t_{j+1}})\right) \lambda_{k}\delta t\right) \right\rangle+
{\sum\limits_{k=1}^{K}\lambda _{k}^{2}\frac{\delta t}{2}},\\
Q_{2}(\lambda ,\Delta W_{1},...,\Delta W_{K}) &:=&
\sum\limits_{i=1}^{M}\frac{1}{\sigma _{i}^{2}}\left\langle \left(
P_{d}^{s}(A(q_{t_{j+1/2}}))-\mathbf{Y}_{t_{j+1}}+\sum%
\limits_{k=1}^{K}P_{d}^{s}\left( B_{k}(\tilde{q}_{t_{j+1}})\right) \lambda_{k}\delta t\right) ,\sum\limits_{k=1}^{K}P_{d}^{s}\left( B_{k}(\tilde{q}_{t_{j+1}})\right) \Delta W_{k}\right\rangle  \\
&+&\frac{1}{2}\sum\limits_{i=1}^{M}\left\vert \left\vert \frac{%
\sum\limits_{k=1}^{K}P_{d}^{s}\left( B_{k}(\tilde{q}_{t_{j+1}})\right)
\Delta W_{k}}{\sigma _{i}}\right\vert \right\vert
_{2}^{2}-\sum\limits_{k=1}^{K}{\lambda _{k}\Delta W_{k}}.
\end{eqnarray*}%
This is a quadratic minimization problem with the optimal value $\lambda $ depending (linearly) on the increments $\Delta W_{1},...,\Delta W_{K}$. This optimal choice is not allowed as the parameter $\lambda$ can only be a function of all the approximation $\tilde{q}_{t_{j+1}}$, $q_{t_{j+1/2}}$ and $Y_{t_{j+1}}$ (since it needs to be adapted
to the forward filtration of the set of Brownian motions $\{W_{k}\}$).
To ensure that this constraint is satisfied, we minimise \emph{the conditional expectation} of \ $%
\mathcal{V}({\mathbf{q}}(\lambda ),\mathbf{Y},\lambda )$ given the $\tilde{q}_{t_{j+1}}$, $q_{t_{j+1/2}}$ and $\mathbf{Y}_{t_{j+1}}$, that is 
$
\displaystyle
\min\limits_{\lambda _{k},\ k\in \lbrack 1..K]}\mathbb{E}\left[ \mathcal{V}({%
\mathbf{q}}(\lambda ),\mathbf{Y},\lambda )|\tilde{q}_{t_{j+1}}, q_{t_{j+1/2}},\mathbf{Y}%
_{t_{j+1}}\right].
$
We note that $Q$ is independent of $\lambda$ and does not play any role in the minimization operation. 

Also 
$\displaystyle\mathbb{E}\left[ Q_{2}(\lambda ,\Delta W_{1},...,\Delta W_{K})|\tilde{q}_{t_{j+1}}, q_{t_{j+1/2}},\mathbf{Y}%
_{t_{j+1}}\right]$ is independent of $\lambda$. Finally 
$Q_{1}(\lambda ,\Delta W_{1},...,\Delta W_{K})$ is measurable wrt $\tilde{q}_{t_{j+1}}, q_{t_{j+1/2}}\mathbf{Y}$, that is
$
\displaystyle
\mathbb{E}\left[ Q_{1}(\lambda )|\tilde{q}_{t_{j+1}}, q_{t_{j+1/2}},\mathbf{Y}_{t_{j+1}}\right]=Q_{1}(\lambda).
$
Consequently, we only need to minimize $Q_{1}(\lambda)$. 
This functional is quadratic in $\lambda$, and hence the optimization can be done by solving a linear system. This is the approach that we 
use in the present work. We note that this  approximation remains asymptotically consistent.

%

\begin{algorithm}
\caption{Particle Filter with Tempering, MCMC, and Nudging}
\begin{algorithmic}
\FOR{$j=0,1,2,\ldots$}
\STATE{Solve\, $\dd q^{(n)}_i+
\left(\mathbf{u}^{(n)}_i\,dt +{\color{red} \sum\limits^K_{k=1} \xi^k_i \circ dW^{(n)}_k}
+{\color{blue}\sum\limits_{k=1}^K\boldsymbol{\xi}^k_i\lambda_k\,dt}\right)\cdot\nabla q^{(n)}_i
=F^{(n)}_i\,dt,\quad 
i=1,2;\quad t\in[t_j,t_{j+1}],\quad n\in[1..N].$}
\STATE{Obtain observation $\mathbf{Y}_{t_{j+1}}$ from weather stations}
\STATE{\color{blue}Minimize $\displaystyle\left[\frac12\sum\limits_{i=1}^M\Big\|\frac{P^s_d(\mathbf{q}_{t_{j+1}})-\mathbf{Y}_{t_{j+1}}}{\sigma_i}\Big\|^2_2+\sum\limits_{k=1}^K\left(\lambda^2_k\frac{\delta t}{2}-\lambda_k\Delta W^{(n)}_k\right) \right]$ with respect to $\lambda_k,\ k\in[1..K].$}
\STATE{Compute $\overline{\mathbf{w}}:=\mathcal{W}(\mathbf{q}_{t_{j+1}},\mathbf{Y}_{t_{j+1}})$, 
with $\mathcal{W}(\mathbf{q}_{t_{j+1}},\mathbf{Y}_{t_{j+1}})=e^{-\Lambda},\ 
{\displaystyle\Lambda:=\frac12\sum\limits_{i=1}^M\Big\|\frac{P^s_d(\mathbf{q}_{t_{j+1}})-\mathbf{Y}_{t_{j+1}}}{\sigma_i}\Big\|^2_2
{\color{blue}+\sum\limits_{k=1}^K\left(\lambda^2_k\frac{\delta t}{2}-\lambda_k\Delta W^{(n)}_k\right)}}$\,.}
\IF{${\rm ESS(\overline{\mathbf{w}})}<N^*$}
\STATE{Find $p$ such that $ESS(\overline{\mathbf{w}})\ge N^*$, where
$\overline{\mathbf{w}}:=\mathcal{W}^{1/p}(\mathbf{q}_{t_{j+1}},\mathbf{Y}_{t_{j+1}})$}
\FOR{$k=1,2,\ldots,p$}
\STATE{Compute $\overline{\mathbf{w}}:=\mathcal{W}^{\phi_k}(\mathbf{q}_{t_{j+1}},\mathbf{Y}_{t_{j+1}}),\ \phi_k:=\frac{k}{p}$}
\STATE{$\mathbf{q}_{t_{j+1}}:={\bf Resample}(\overline{\mathbf{w}}).$}
\FOR{$m=1,2,\ldots,M_1$}
\STATE{${\color{red} \Xi^{(n)}_i:=\sum\limits^K_{k=1} \xi_{k,i} \circ \left(\rho\, dW^{(n)}_{k}+\sqrt{1-\rho^2}\,d\widetilde{W}^{(n)}_{k}\right)},\quad n\in[1..N]$}
\STATE{Solve\, $\dd \widetilde{q}^{(n)}_i+
\left(\widetilde{\mathbf{u}}^{(n)}_i\,dt +{\color{red} \Xi^{(n)}_i}\right)\cdot\nabla \widetilde{q}^{(n)}_i
=\widetilde{F}^{(n)}_i\,dt,\quad 
i=1,2;\quad t\in[t_j,t_{j+1}],\quad n\in[1..N]$}
\FOR{n=1,2,\ldots,N}
\STATE{$\alpha:=\left(\mathcal{W}(\widetilde{q}^{(n)}_{t_{j+1}},Y^{(n)}_{t_{j+1}})/\mathcal{W}(q^{(n)}_{t_{j+1}},Y^{(n)}_{t_{j+1}})\right)^{\phi_k}$}
\IF{$1\le\alpha$}
\STATE{$q^{(n)}_{t_{j+1}}:=\widetilde{q}^{(n)}_{t_{j+1}}$}
\ELSIF{$\mathcal{U}[0,1]<\alpha$}
\STATE{$q^{(n)}_{t_{j+1}}:=\widetilde{q}^{(n)}_{t_{j+1}}$}
\ENDIF
\ENDFOR
\ENDFOR
\ENDFOR
\ENDIF
\ENDFOR
\end{algorithmic}
\end{algorithm}

\section{Numerical results}
\label{sec:num_res}
We consider a horizontally periodic flat-bottom channel $\Omega=[0,L_x]\times[0,L_y]\times[0,H]$ 
with $L_x=3840\, \rm km$, $L_y=L_x/2\, \rm km$, and total depth 
$H=H_1+H_2$, with $H_1=1.0\, \rm km$, $H_2=3.0\, \rm km$.
The governing parameters of the QG model are typical to a mid-latitude setting, i.e.
the planetary vorticity gradient $\beta=2\times10^{-11}\, {\rm m^{-1}\, s^{-1}}$, lateral eddy viscosity $\nu=3.125\,\rm m^2 s^{-1}$, and
the bottom friction parameters $\mu=4\times10^{-8}$. 
The background-flow zonal velocities in~\eqref{eq:forcing} are given by $U=[6.0,0.0]\,\rm cm\, s^{-1}$, and
the stratification parameters in~\eqref{eq:q_psi} are $s_1=4.22\cdot10^{-3}\,\rm km^{-2}$, $s_2=1.41\cdot10^{-3}\,\rm km^{-2}$; chosen so that the first Rossby deformation radius is $Rd_1=25\, {\rm km}$.
In order to ensure that the numerical solutions are statistically equilibrated, the model is initially spun up from the state of rest to $t=0$ over the time
interval $T_{spin}=[-50,0]\, {\rm years}$. The numerical solutions of the deterministic QG model~\eqref{eq:pv}
at different resolutions are presented in Fig.~\ref{fig:qf_all_grids}.

\begin{figure}[H]
\centering
\begin{tabular}{cccc}
\multicolumn{2}{c}{\hspace*{0.4cm} \bf FIRST LAYER} & \multicolumn{2}{c}{\hspace*{1cm} \bf SECOND LAYER}\\
& & &\\[-0.25cm]
\multicolumn{4}{c}{\hspace*{0.4cm} \textbf{(a)} $G=2049\times1025$}\\
& \hspace*{-0.325cm}\begin{minipage}{0.225\textwidth}\includegraphics[width=4cm]{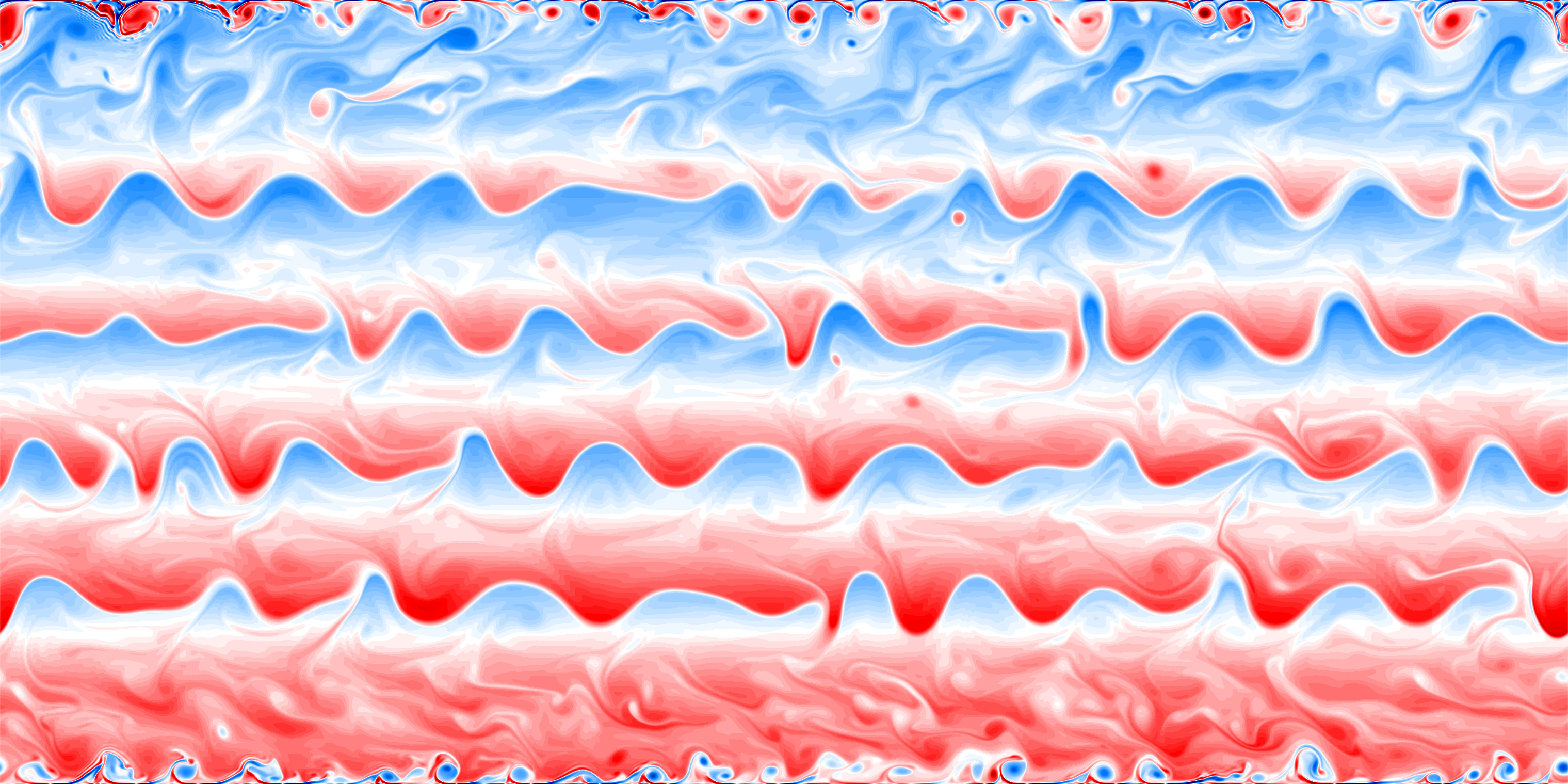}\end{minipage}
& \quad \begin{minipage}{0.02\textwidth}$\color{white}q^f_2$\end{minipage} & \hspace*{-0.325cm}\begin{minipage}{0.225\textwidth}\includegraphics[width=4cm]{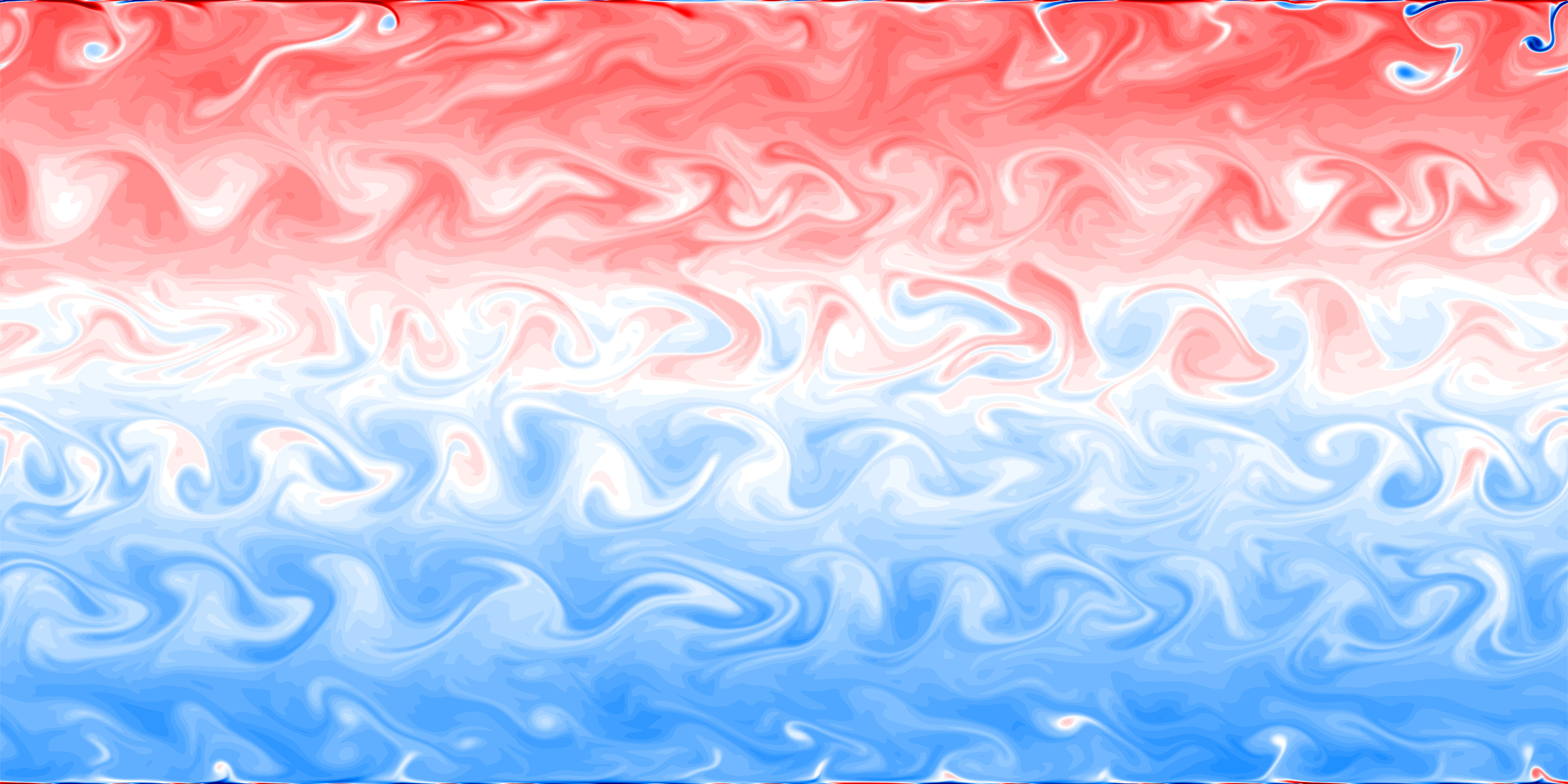}\end{minipage}\\
\\[-0.125cm]
\multicolumn{4}{c}{\hspace*{0.4cm} \textbf{(b)} $G=1025\times513$}\\
& \hspace*{-0.325cm}\begin{minipage}{0.225\textwidth}\includegraphics[width=4cm]{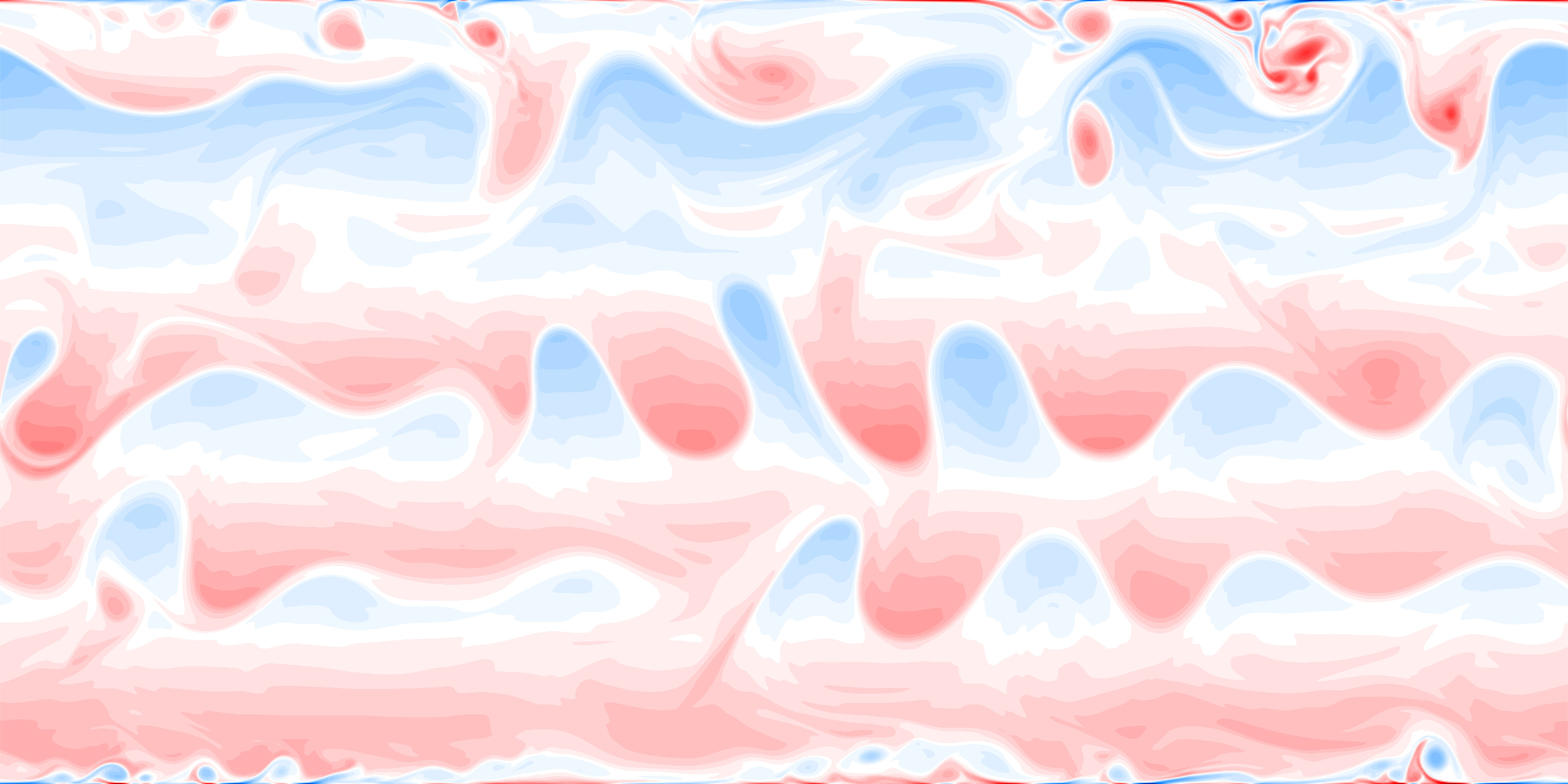}\end{minipage}
& \quad \begin{minipage}{0.02\textwidth}$\color{white}q^f_2$\end{minipage} & \hspace*{-0.325cm}\begin{minipage}{0.225\textwidth}\includegraphics[width=4cm]{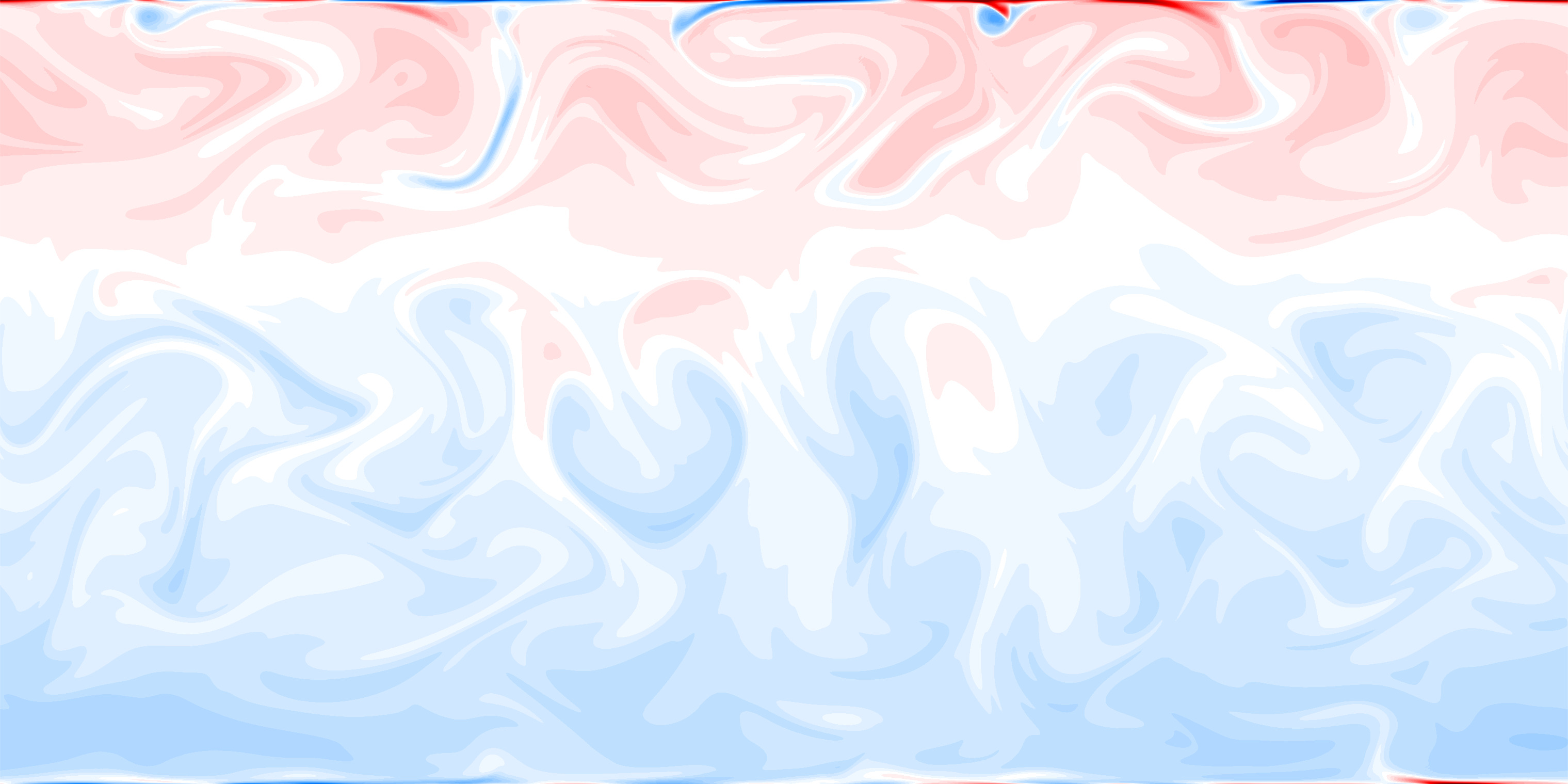}\end{minipage}\\
\\[-0.125cm]
\multicolumn{4}{c}{\hspace*{0.4cm} \textbf{(c)} $G=513\times257$}\\
& \hspace*{-0.325cm}\begin{minipage}{0.225\textwidth}\includegraphics[width=4cm]{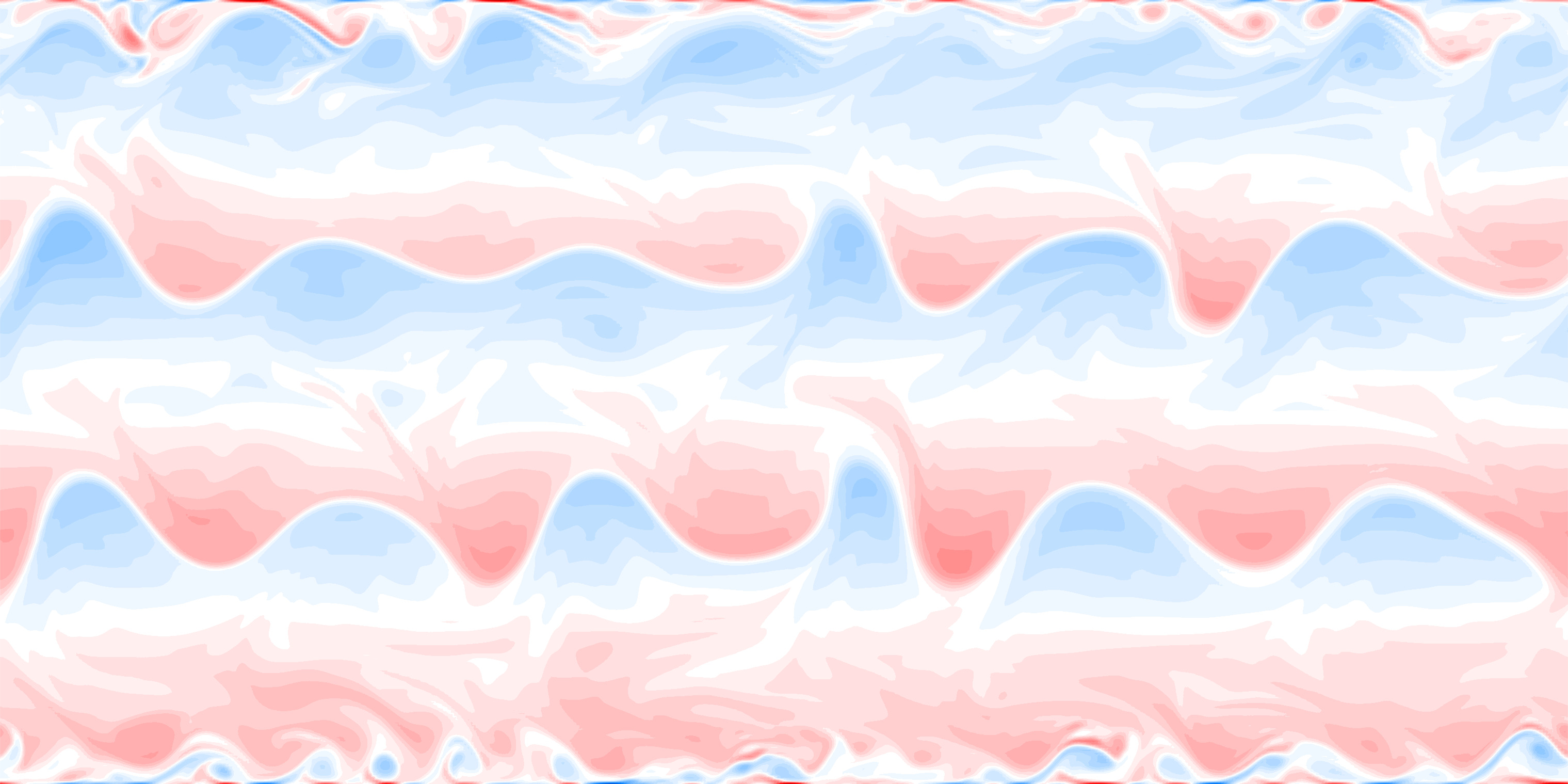}\end{minipage}
& \quad \begin{minipage}{0.02\textwidth}$\color{white}q^f_2$\end{minipage} & \hspace*{-0.325cm}\begin{minipage}{0.225\textwidth}\includegraphics[width=4cm]{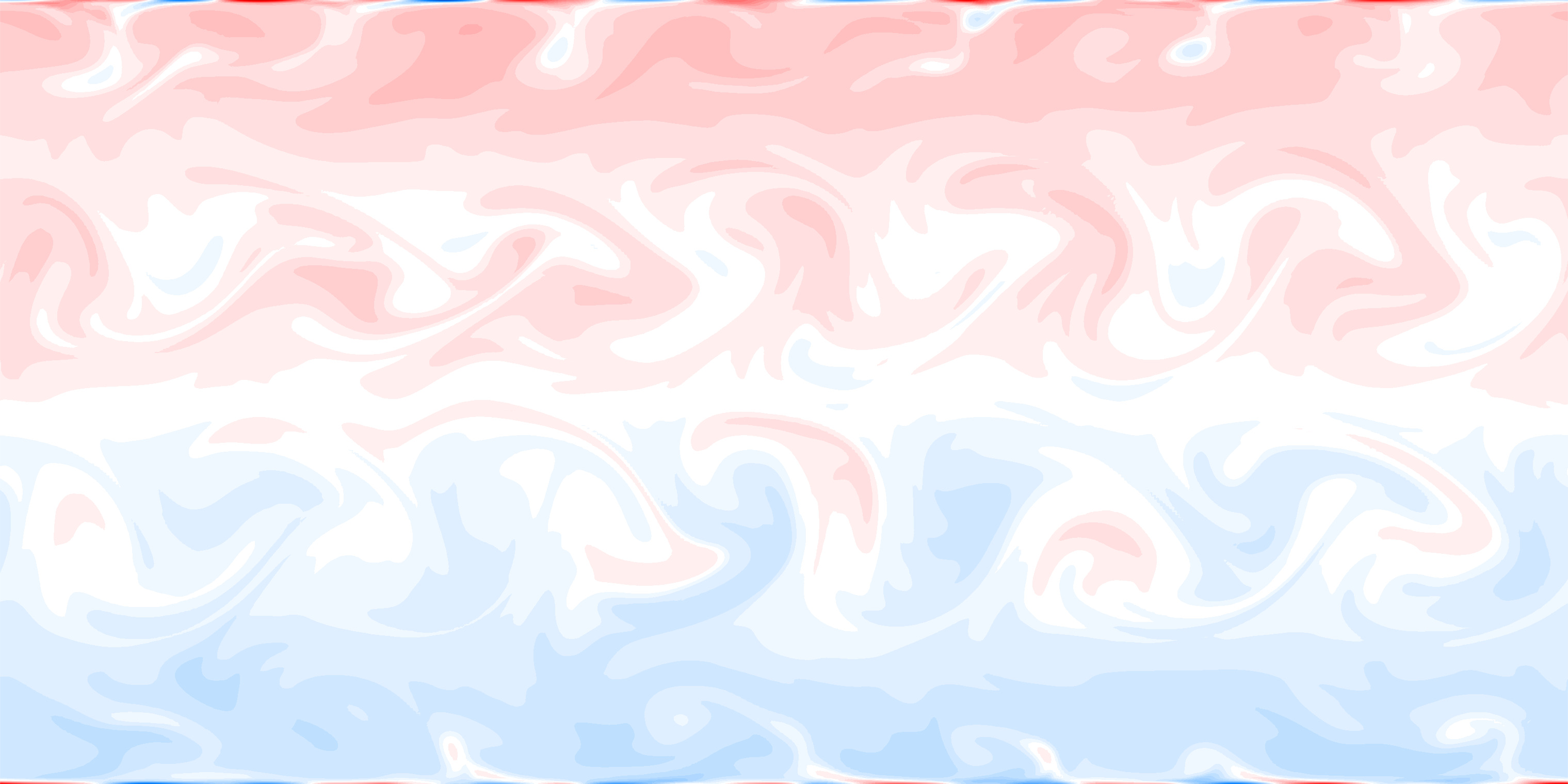}\end{minipage}\\
\\[-0.125cm]
\multicolumn{4}{c}{\hspace*{0.4cm} \textbf{(d)} $G=257\times129$}\\
& \hspace*{-0.325cm}\begin{minipage}{0.225\textwidth}\includegraphics[width=4cm]{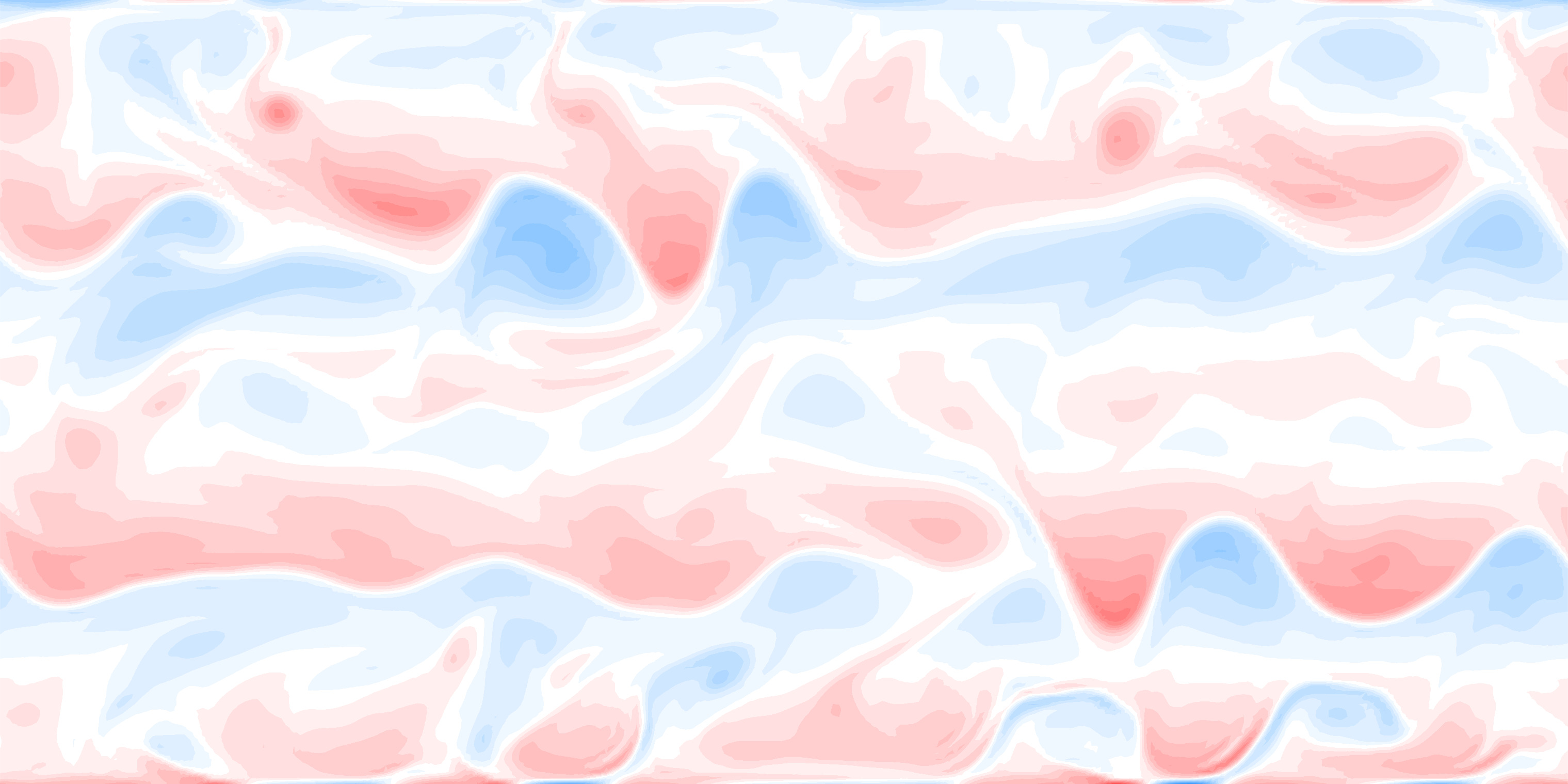}\end{minipage}
& \quad \begin{minipage}{0.02\textwidth}$\color{white}q^f_2$\end{minipage} & \hspace*{-0.325cm}\begin{minipage}{0.225\textwidth}\includegraphics[width=4cm]{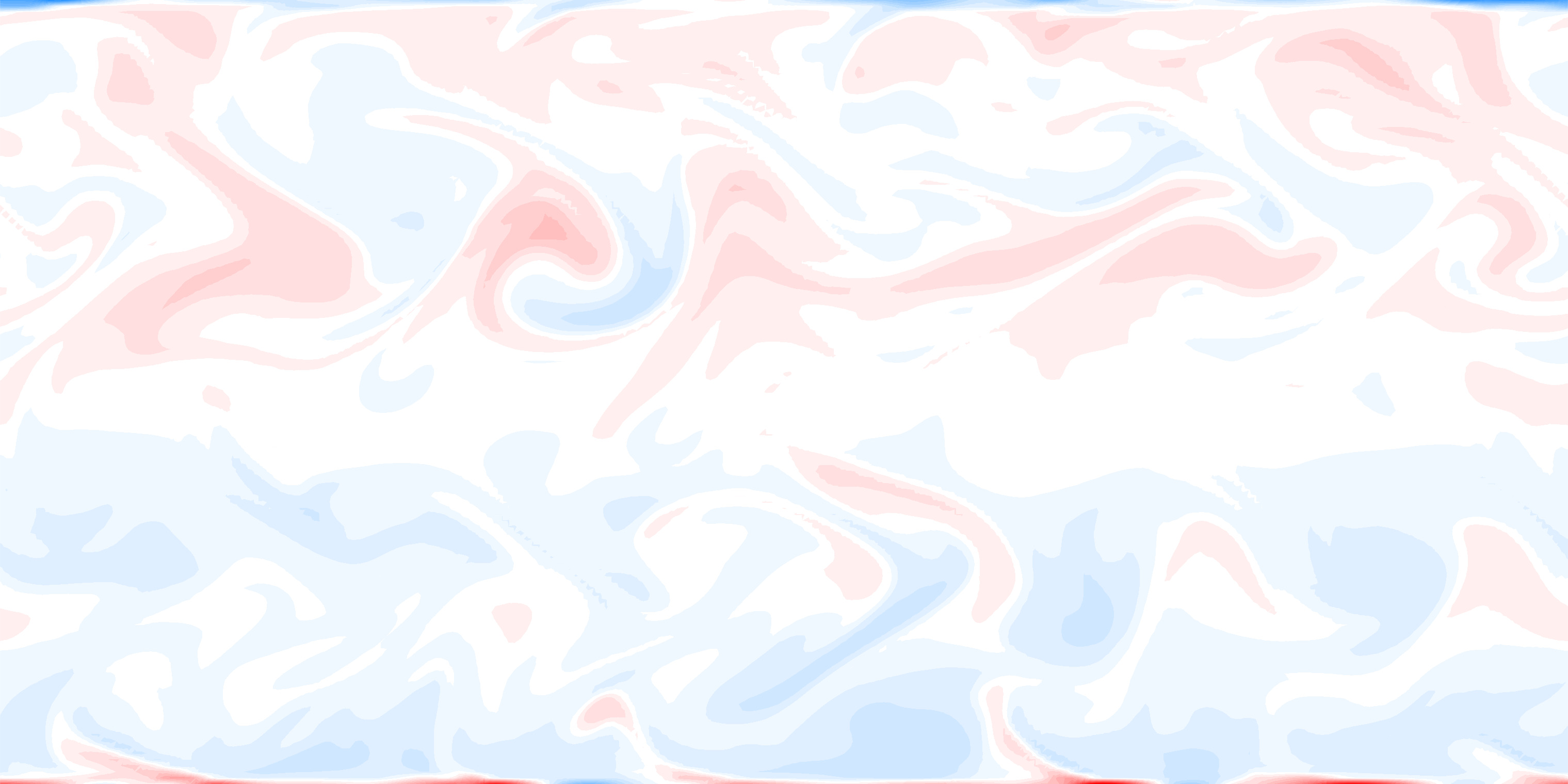}\end{minipage}\\
\\[-0.125cm]
\multicolumn{4}{c}{\hspace*{0.4cm} \textbf{(e)} $G=129\times65$}\\
& \hspace*{-0.325cm}\begin{minipage}{0.225\textwidth}\includegraphics[width=4cm]{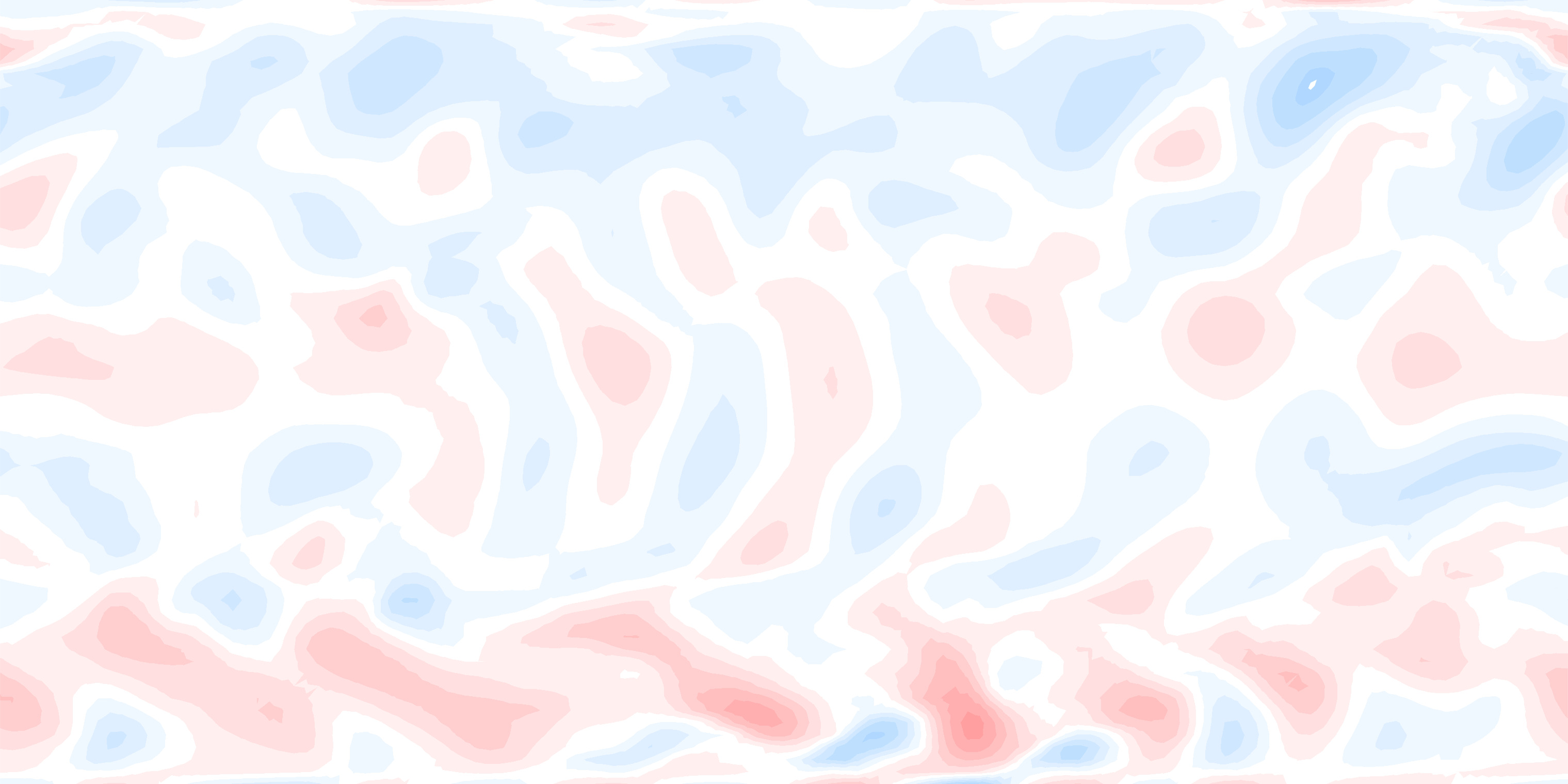}\end{minipage}
& \quad \begin{minipage}{0.02\textwidth}$\color{white}q^f_2$\end{minipage} & \hspace*{-0.325cm}\begin{minipage}{0.225\textwidth}\includegraphics[width=4cm]{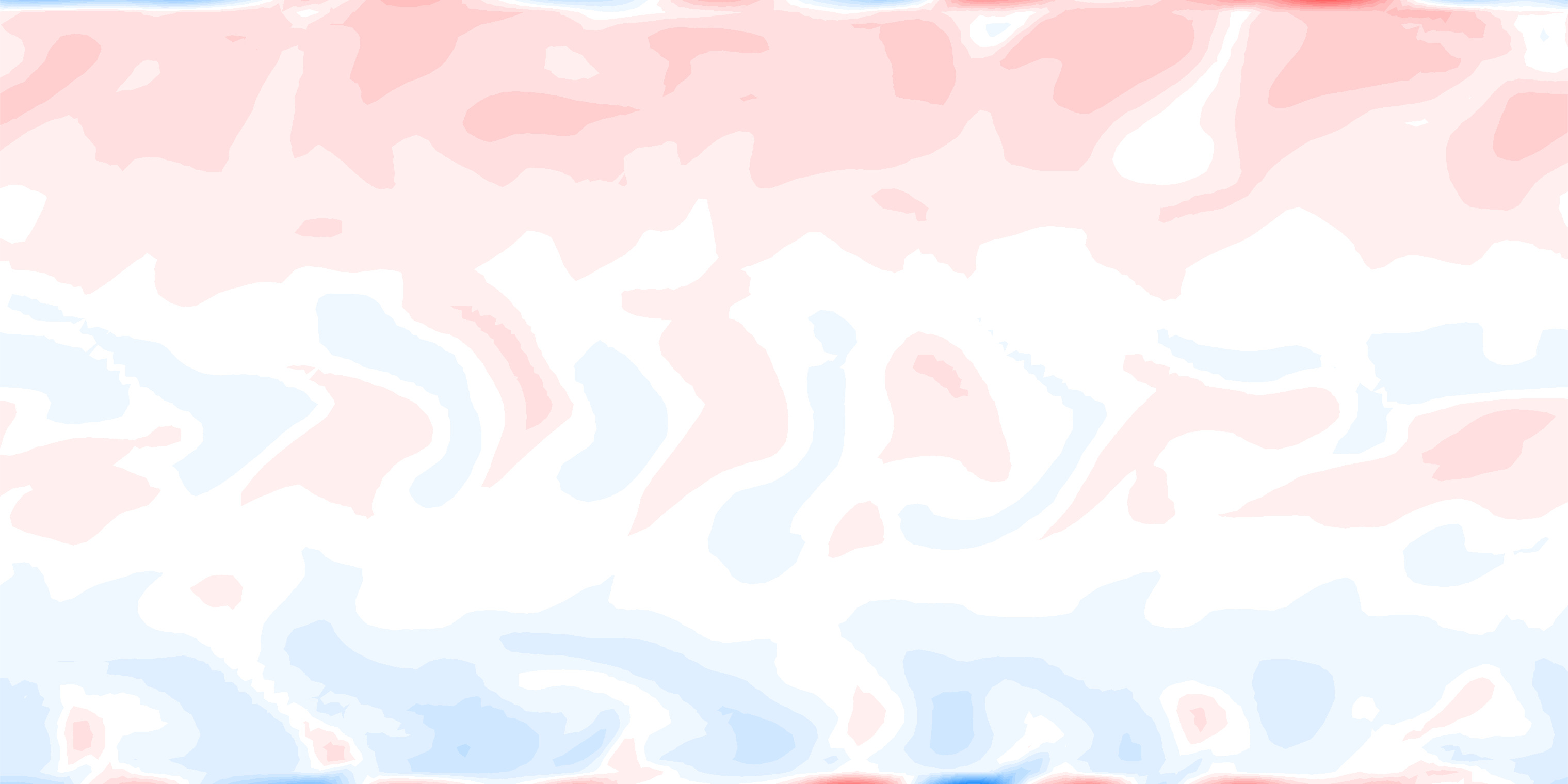}\end{minipage}\\
\\[-0.125cm]
\multicolumn{4}{c}{\hspace*{0.75cm}\includegraphics[width=6cm,height=0.5cm]{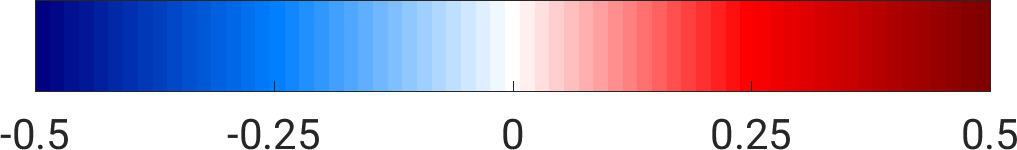}}\\
\end{tabular}
\caption{The series of snapshots of PV anomaly $q$ shows the dependence of the solution on the resolution.
All the fields are given in units of $[s^{-1}f^{-1}_0]$,  where $f_0=0.83\times10^{-4}\, {\rm s^{-1}}$ is the Coriolis parameter.
In order to visualize all the solutions on the same color scale we have multiplied the ones in the second layer by a factor of 5.}
\label{fig:qf_all_grids}
\end{figure}

We would like to draw the reader's attention to the fact that the solution significantly depends on the resolution (see Fig.~\ref{fig:qf_all_grids}). Namely, the number of 
jets for the highest resolution $2049\times1025$ (Fig.~\ref{fig:qf_all_grids}a) is four (four red striations in the interior of the computational domain; the boundary layers on the top and bottom boundaries are not counted as such).
However, there are only two jets for the lower-resolution flows: $G=1025\times513$ (Fig.~\ref{fig:qf_all_grids}b), 
$G=513\times257$ (Fig.~\ref{fig:qf_all_grids}c), $G=257\times129$ (Fig.~\ref{fig:qf_all_grids}d).
Moreover, the 
lowest resolution flow (computed on the grid $G=129\times65$, Fig.~\ref{fig:qf_all_grids}e) shows no jets at all, and this is the flow that we paramaterise and then apply the data assimilation methods
described above. We also use a finer grid $G=257\times129$ to study the effect of the resolutions on the data assimilation methods.
It is important to note that there is no smooth transition between solutions at different resolutions like, for instance,
in the double-gyre problem (e.g.~\cite{SB2015}), and this makes lower-resolution solutions much harder to parameterise, since the 
parameterisation should compensate not only for the information lost because of coarse-graining, but also for the missing physical
effects. For example, in the channel flow, the backscatter mechanism (e.g,~\cite{SB2016}) at low resolutions is very weak, and 
thus it is not capable of cascading energy up to larger scales to maintain the jets. 

In the following, we compare the dependence of the performance of the various procedures discussed above on the following parameters: the resolution of the signal grid, the number of observations (also referred to as weather stations), and the size of the data assimilation step. The methods will be applied to the parameterised QG model~\eqref{eq:SLTpv} which has been studied at length in~\cite{CCHWS2019_2}.

Before going into further details, we remind the reader how we compute the true solution, which is denoted 
as $q^a$, and also referred to as \textit{the truth} or the \textit{the true state}. For the purpose of this paper,
we have computed two versions of the true solution $q^a$. 
The first one is computed on a signal grid $G_s=257\times129$ ($dx\approx dy\approx 15\, {\rm km}$),
and the other one is computed on a signal grid $G_s=129\times65$ ($dx\approx dy\approx 299\, {\rm km}$) (Fig.~\ref{fig:qf_qa_qc_qam_qcm_mu1D-8_257x129}).
Each true solution is computed as the solution of the elliptic equation~\eqref{eq:q_psi} with the stream function $\psi^a$, 
where $\psi^a$ is computed by spatially averaging the high-resolution stream function $\psi^f$ (computed on the fine grid 
$G_f=2049\times1025,\, dx\approx dy\approx 1.9\, {\rm km}$) over the coarse grid cell $G^s$.
From now on, we focus only on the first layer, since the flow in the first layer 
is more energetic and exhibits an interesting dynamics including small-scale vortices and large-scale zonally elongated jets.

\begin{figure}[H]
\centering
\begin{tabular}{ccc}
 & \begin{minipage}{0.225\textwidth}\hspace*{0.5cm}$G_s=257\times129$\end{minipage} & \begin{minipage}{0.225\textwidth}\hspace*{1.5cm}$G_s=129\times65$\end{minipage}\\
\begin{minipage}{0.02\textwidth}$q^a_1$\end{minipage} & \hspace*{-2.325cm}\begin{minipage}{0.225\textwidth}\includegraphics[width=5cm]{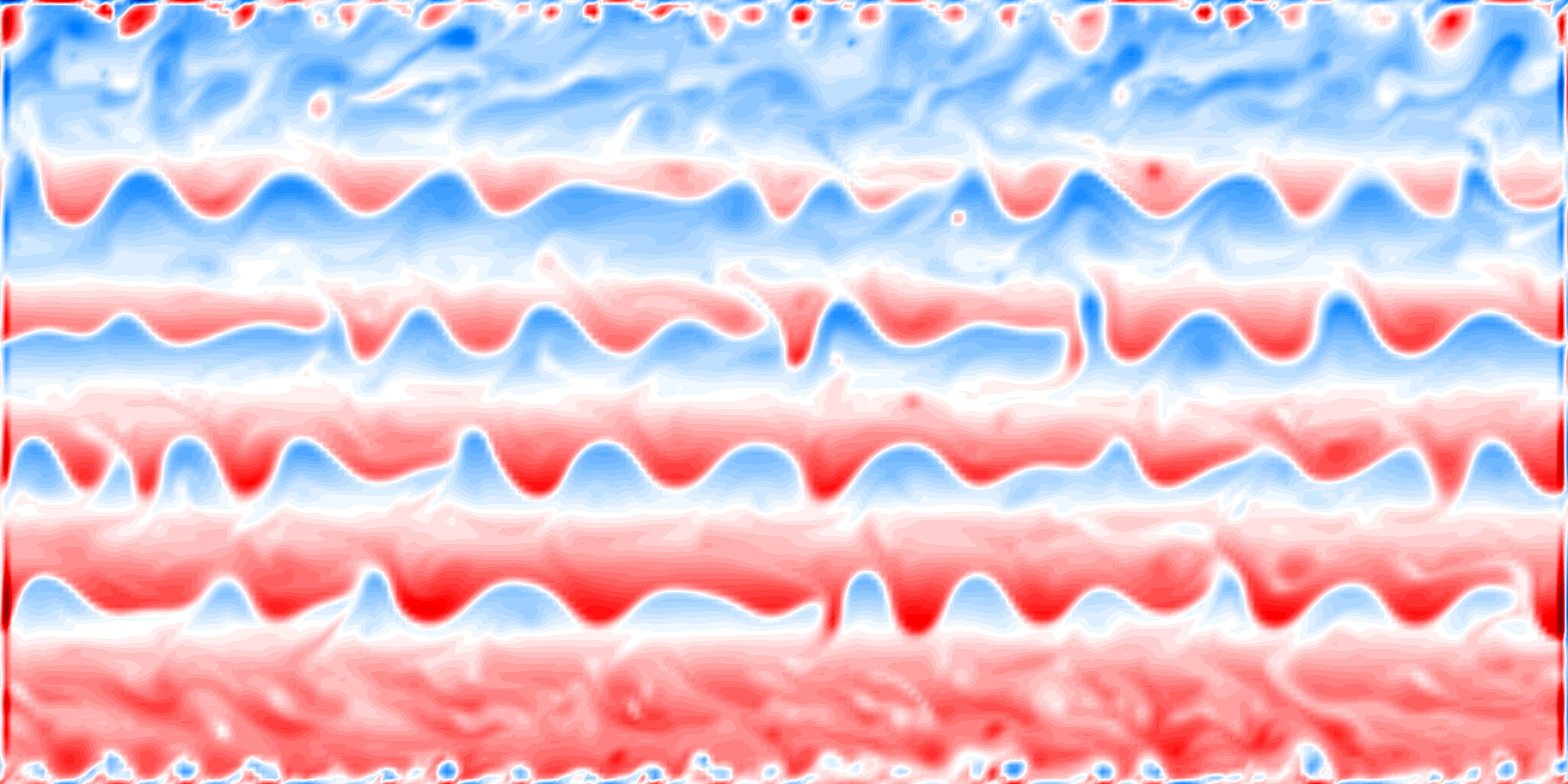}\end{minipage}
& \begin{minipage}{0.225\textwidth}\includegraphics[width=5cm]{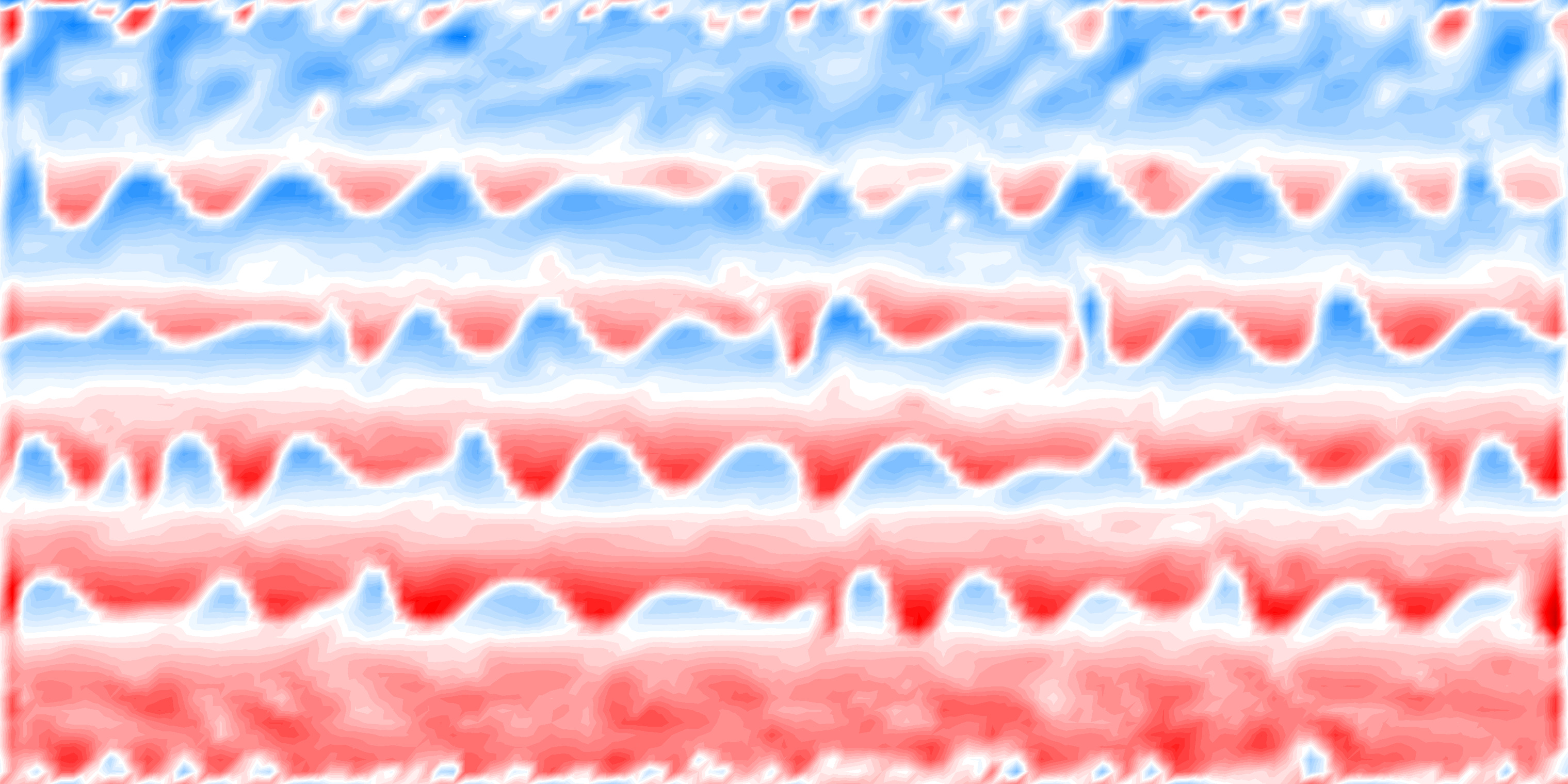}\end{minipage}\\
\\[-0.25cm]
\multicolumn{3}{c}{\hspace*{0.75cm}\includegraphics[width=6cm,height=0.5cm]{figs/colorbar_bwr_mD-5_pD-5.jpg}}\\
\end{tabular}
\caption{Shown is the true solution $q^a$ in the first layer computed on the grids 
$G_s=257\times129$ (left) and $G_s=129\times65$ (right).
All the fields are given in units of $[s^{-1}f^{-1}_0]$,  where $f_0=0.83\times10^{-4}\, {\rm s^{-1}}$ is the Coriolis parameter.}
\label{fig:qf_qa_qc_qam_qcm_mu1D-8_257x129}
\end{figure}

In order to study how the number of weather stations influences the accuracy of data assimilation methods
we will consider two different setups including $M=16$ and $M=32$ weather stations. Clearly, the location of weather stations can be optimized in such a way so as to give the most accurate data assimilation
results. For the purpose of this work,  we locate the weather stations at the nodes of 
equidistant Eulerian grids $G_d=\{4\times4,8\times4\}$ shown in Fig.~\ref{fig:ws_location}.
In all numerical experiments we use $N=100$ particles (ensemble members) and $K=32$ (the number of $\xi$'s);
the choice of these parameters has been justified in~\cite{CCHWS2019_2}.
It is worth noting that the initial conditions for the stochastic model 
have been computed over the spin up period $T_{\rm spin}=[-8,0]$ hours. The method of computing physically consistent
initial conditions for the stochastic model is given in~\cite{CCHWS2019_2}.

\begin{figure}[H]
\centering
\begin{tabular}{c}
\hspace*{-2cm}\begin{minipage}{0.26\textwidth}\includegraphics[width=7cm]{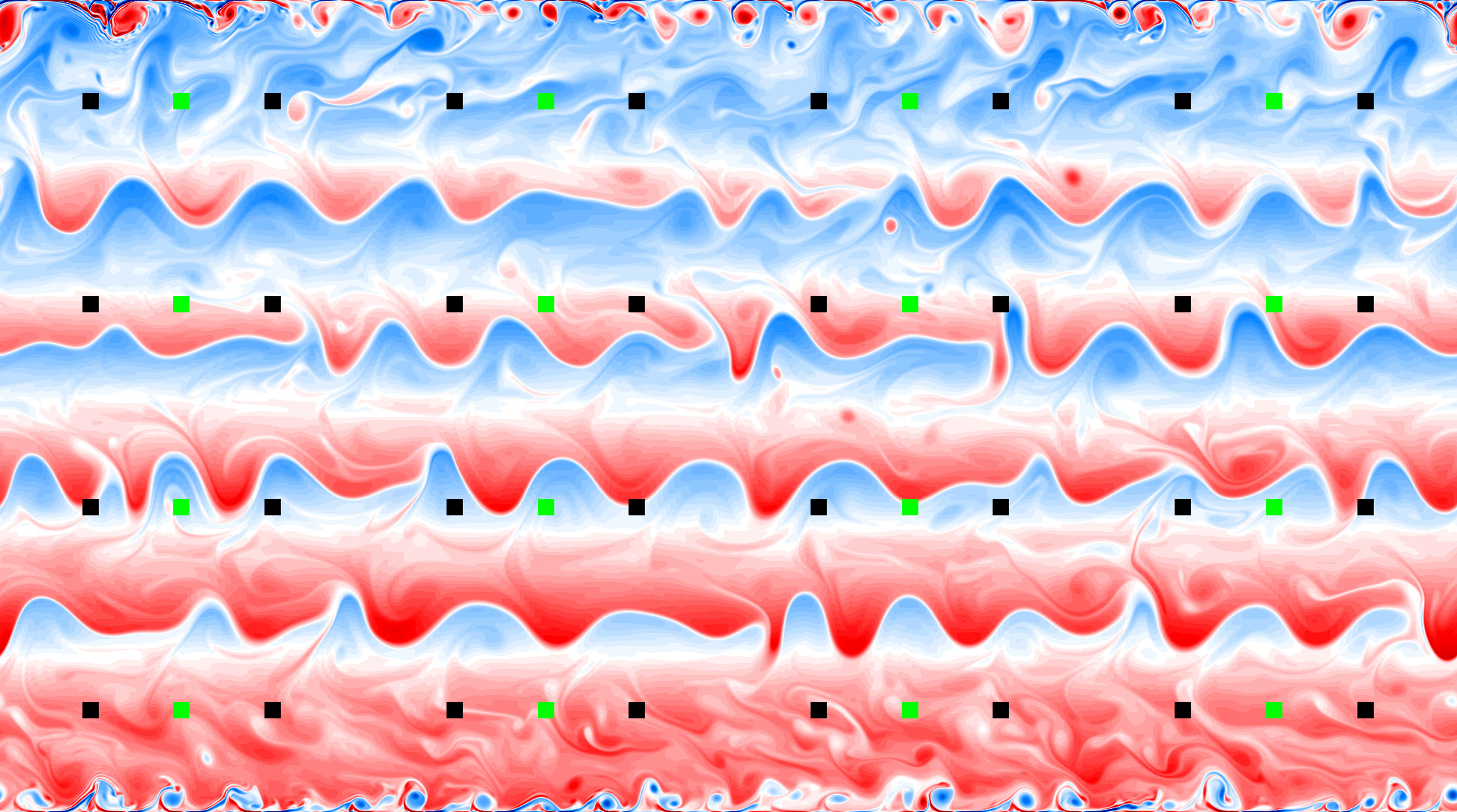}\end{minipage}\\
\\[-0.25cm]
\multicolumn{1}{c}{\hspace*{0.75cm}\includegraphics[width=6cm,height=0.5cm]{figs/colorbar_bwr_mD-5_pD-5.jpg}}\\
\end{tabular}
\caption{Snapshot of PV anomaly $q^f_1$ and location of weather stations for the data grids $G_d=\{4\times4\}$ (green squares)
and $G_d=\{8\times4\}$ (black squares);
the PV anomaly field is given in units of $[s^{-1}f^{-1}_0]$,  where $f_0=0.83\times10^{-4}\, {\rm s^{-1}}$ is the Coriolis parameter.}
\label{fig:ws_location}
\end{figure}
\subsection{Tempering and jittering}
\label{sec:temp_jitt}
We start with the study to the data assimilation algorithm that uses tempering and jittering, but not nudging (Algorithm 4). 
In the first experiment we run the stochastic QG model at the coarsest resolution ($G_s=129\times65$) and use $M=16$ weather stations. We compare the results of the data assimilation methodology with the forward run of the stochastic model. For this experiment
we take the data assimilation step to be $\Delta t=4$ hours (the data assimilation step is time between to consecutive analysis cycles.

The results are presented in Fig.~\ref{fig:av_relerr_v_129x65} in terms of the relative bias (RB) and the ensemble mean $l_2$-norm relative error (EME) given by
\begin{eqnarray}
RB(\mathbf{u}^a,\mathbf{u}^p)&:=&\frac{\|\mathbf{u}^a-\bar{\mathbf{u}}^p\|_2}{\|\mathbf{u}^a\|_2}, \\
EME(\mathbf{u}^a,\mathbf{u}^p)&:=&\frac{1}{N}\sum\limits^N_{n=1}\frac{\|\mathbf{u}^a-\mathbf{u}^p_n\|_2}{\|\mathbf{u}^a\|_2}, 
\label{eq:l2_re_av_ensemble}
\end{eqnarray}
with $\mathbf{u}^p_n$ being the $n$-th member of the stochastic ensemble, $\mathbf{u}^a$ is the true solution,
and $\displaystyle\bar{\mathbf{u}}^p:=\frac{1}{N}\sum\limits^N_{n=1}\mathbf{u}^p_n$.

\begin{figure}
\centering
\begin{tabular}{ccc}
& \begin{minipage}{0.45\textwidth}\hspace*{1.4cm}\bf (a): EME for all weather stations\end{minipage} & \begin{minipage}{0.45\textwidth}\hspace*{1.4cm}\bf (b): EME for the whole domain\end{minipage}\\
\begin{minipage}{0.02\textwidth}\rotatebox{90}{EME}\end{minipage} & 
\hspace*{-0.25cm}\begin{minipage}{0.45\textwidth}\includegraphics[width=7.5cm]{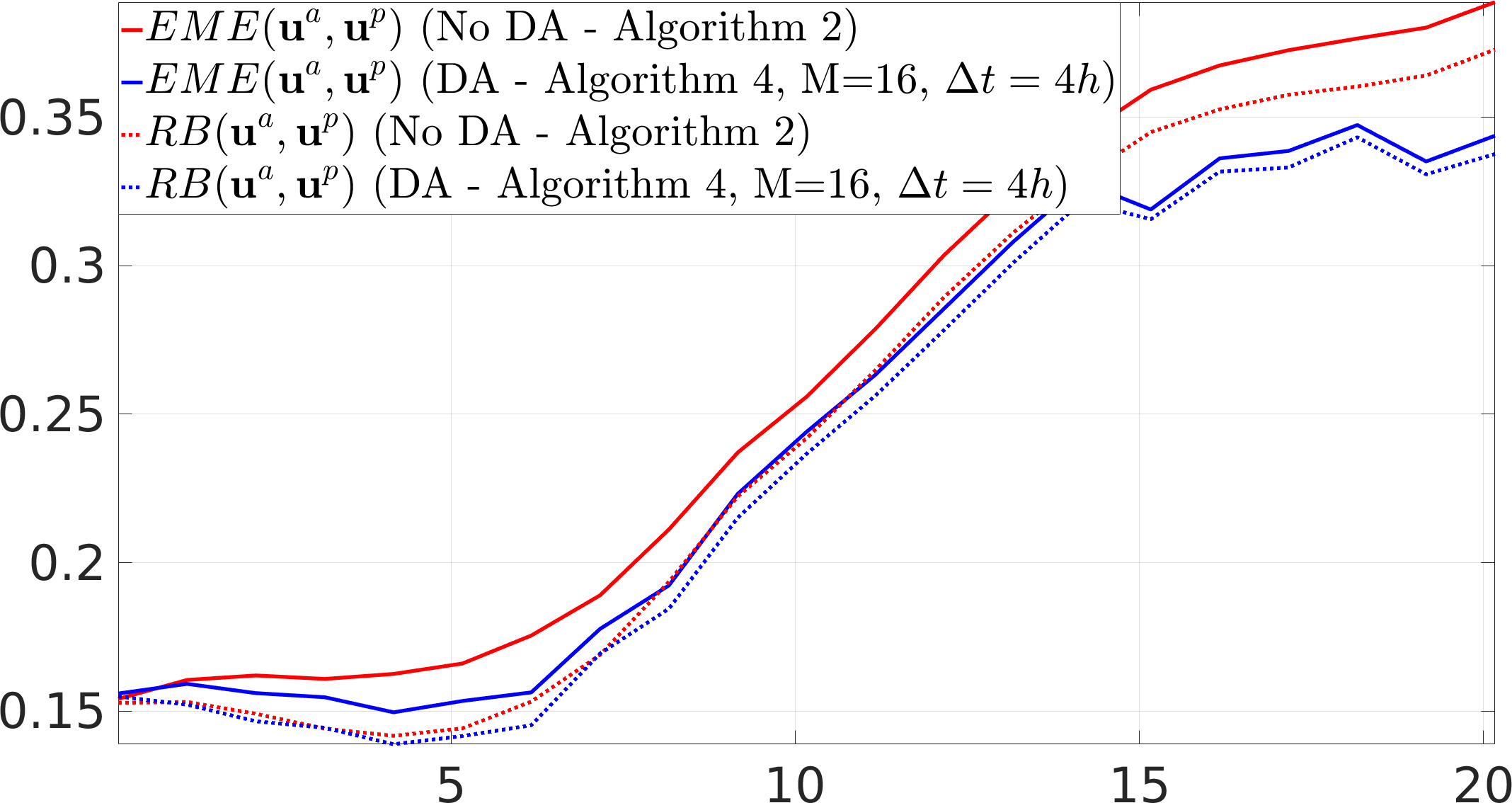}\end{minipage} &
\hspace*{-0.25cm}\begin{minipage}{0.45\textwidth}\includegraphics[width=7.5cm]{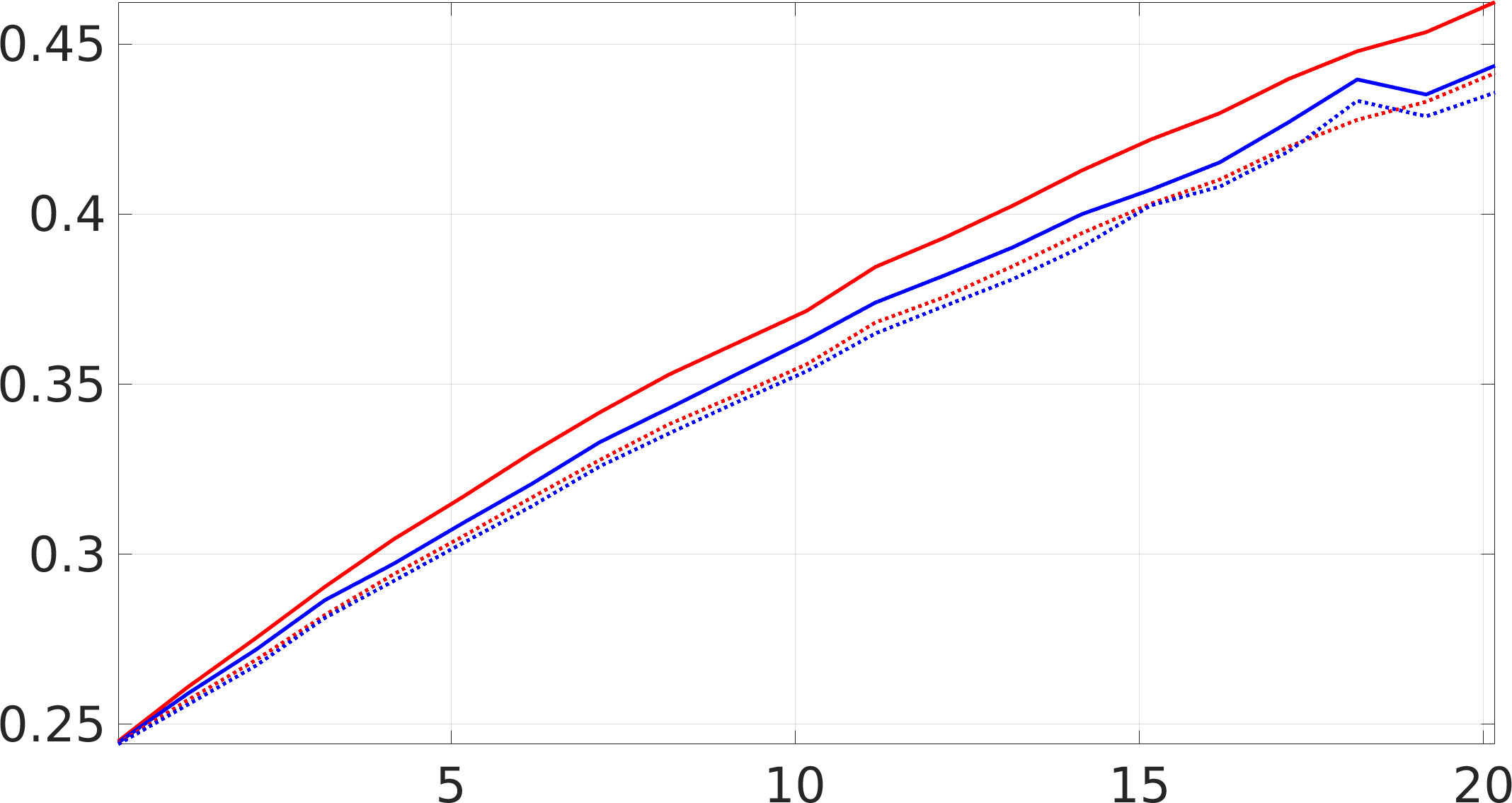}\end{minipage}\\
& &\\[-0.25cm]
&\hspace*{0.5cm}\begin{minipage}{0.1\textwidth}t {\rm [days]}\end{minipage} &\hspace*{0.5cm}\begin{minipage}{0.1\textwidth}t {\rm [days]}\end{minipage}
\end{tabular}
\caption{Evolution of the ensemble mean relative $l_2$-norm error (EME) and relative bias (RB) for \textbf{(a)} all weather stations and \textbf{(b)} the whole domain; $\mathbf{u}^a$ is the true solution, $\mathbf{u}^p$ is the stochastic solution.
In order to assimilate data we use \textit{tempering and jittering (Algorithm 4)}; 
the data is assimilated from $M=16$ weather stations every 4 hours; the grid size is $G_s=129\times65$.}
\label{fig:av_relerr_v_129x65}
\end{figure}

As Fig.~\ref{fig:av_relerr_v_129x65} shows, the data assimilation method presented by Algorithm 4 (blue line)
offers little improvement over the SPDE run without data assimilation methodology (red line) both in terms of the relative bias and in terms of the EME 
at the weather stations (Fig.~\ref{fig:av_relerr_v_129x65}a)
and in the whole domain (Fig.~\ref{fig:av_relerr_v_129x65}b) throughout the time period of 20 days. 
As we will see later, the situation will improve as we decrease the data assimilation step and/or increase the resolution of the signal grid.
But before, let us first look at the  uncertainty quantification results for this particular setting. As expected, the spread for the stochastic QG model~\eqref{eq:SLTpv} decreases (to a certain extend) after the application of the data assimilation methodology. We illustrate the shrinkage of the stochastic spreads in Fig.~\ref{fig:spread_129x65_16ws_4h} for the velocity computed
at weather stations located in the slow flow region (blue stripes)
and fast flow region (red jets) in Fig.~\ref{fig:ws_location}.

\begin{figure}[H]
\centering
\hspace*{1.0cm}
\begin{tabular}{cc}
\hspace*{0cm}\begin{minipage}{0.1\textwidth} \large$u^p_1$ \end{minipage} &\hspace*{0cm}\begin{minipage}{0.1\textwidth} \large$v^p_1$ \end{minipage}\\
&\\[-0.35cm]
\hspace*{-0.25cm}\begin{minipage}{0.45\textwidth}\includegraphics[scale=0.25]{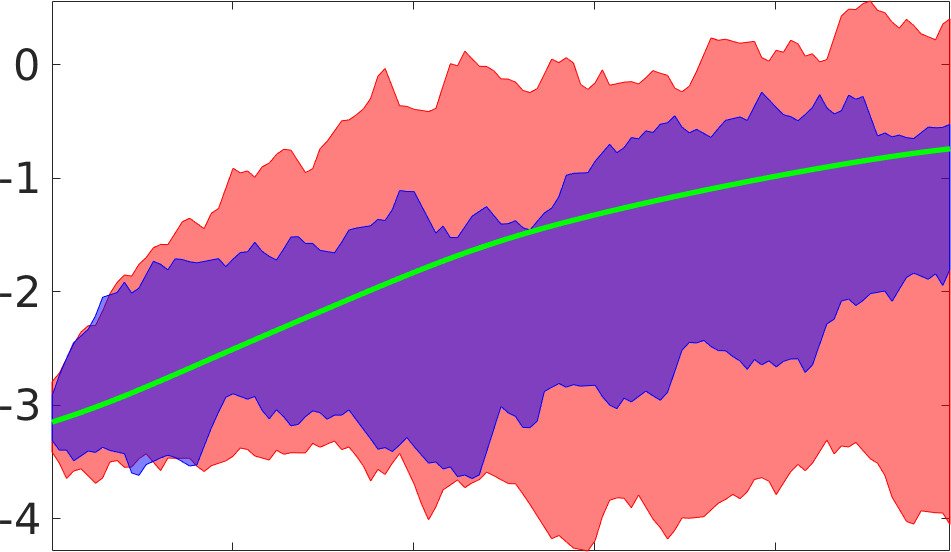}\end{minipage} &
\hspace*{-0.25cm}\begin{minipage}{0.45\textwidth}\includegraphics[scale=0.25]{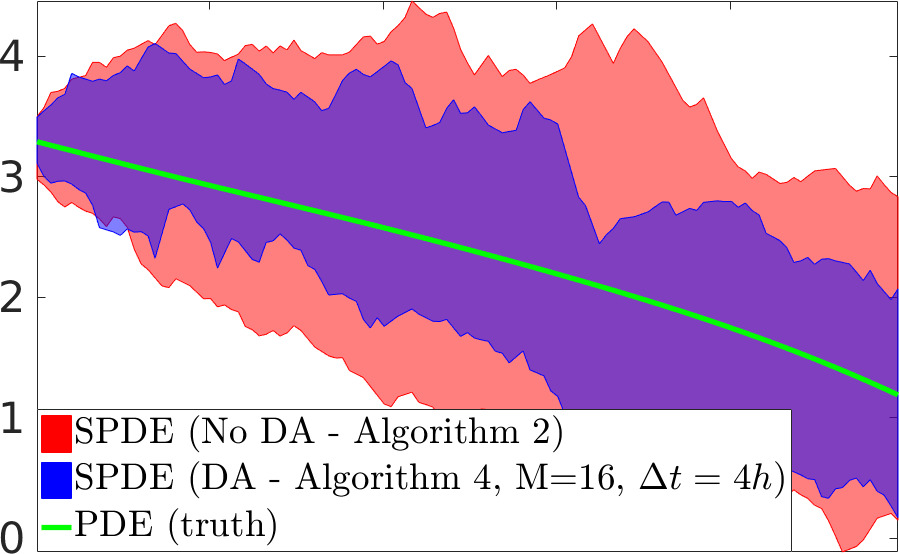}\end{minipage}\\
&\\[-0.25cm]
\hspace*{-0.75cm}\begin{minipage}{0.45\textwidth}\includegraphics[scale=0.25]{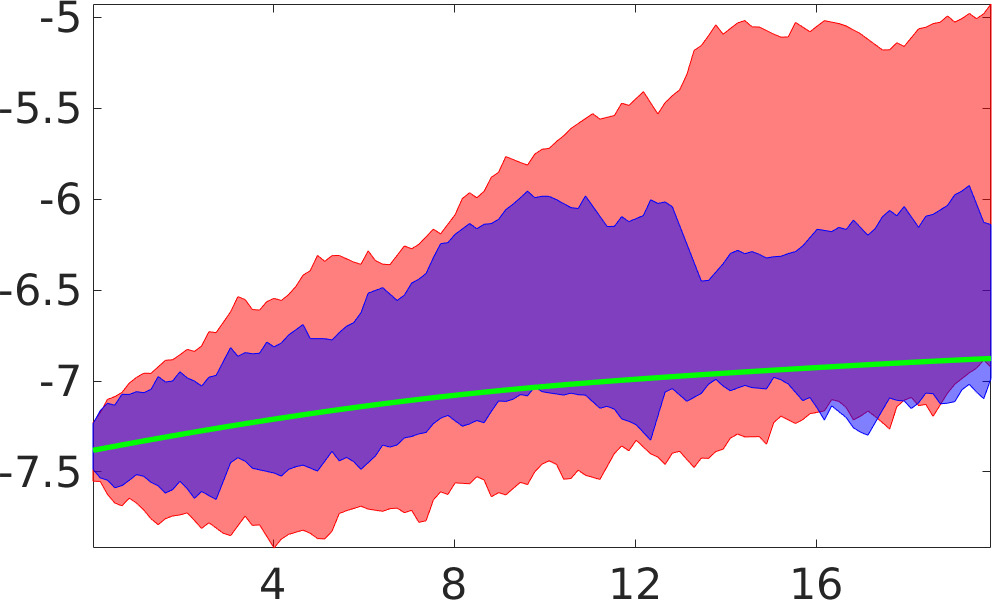}\end{minipage} &
\hspace*{-1cm}\begin{minipage}{0.45\textwidth}\includegraphics[scale=0.25]{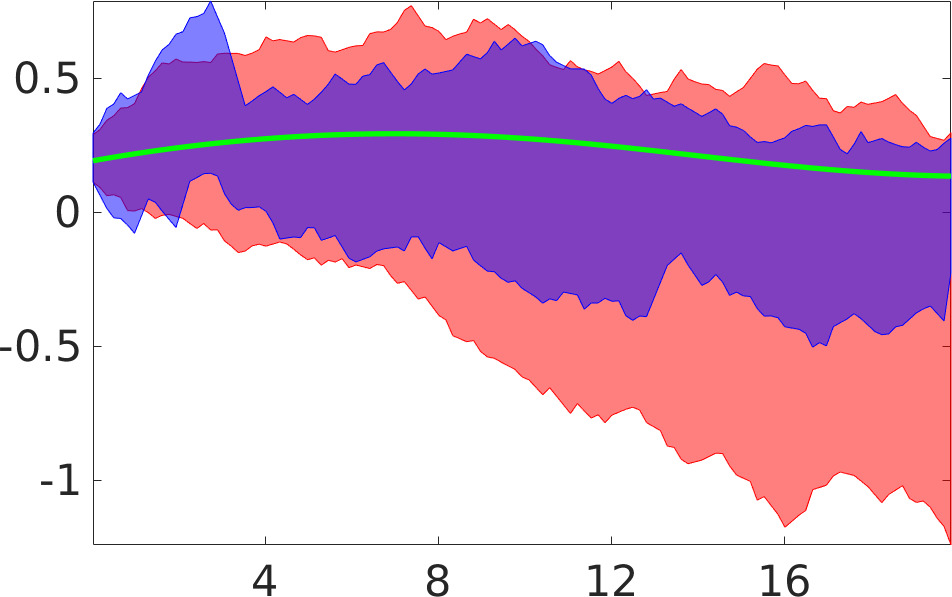}\end{minipage}\\
&\\[-0.25cm]
\hspace*{0cm}\begin{minipage}{0.1\textwidth}t {\rm [days]}\end{minipage} &\hspace*{-0.5cm}\begin{minipage}{0.1\textwidth}t {\rm [days]}\end{minipage}
\end{tabular}
\caption{Shown are typical stochastic spreads for velocity $\mathbf{u}^p_1=(u^p_1,v^p_1)$ at weather stations
located in the slow flow region (upper row) and fast flow region (lower row); the grid size is $G_s=129\times65$.
}
\label{fig:spread_129x65_16ws_4h}
\end{figure}

Fig.~\ref{fig:spread_129x65_16ws_4h} shows that the truth (green line) is contained within the stochastic spread 
computed with and without the data assimilation method. Moreover, the spread computed with the data assimilation method (blue spread)
is narrower than that computed only with the SPDE (red spread). To reduce it further, one can vary the data assimilation step $\Delta t$. In particular, halving $\Delta t$ brings further reduction in both RB and EME  (Fig.~\ref{fig:av_relerr_v_129x65_dt_2h})
and also reduces the uncertainty of the stochastic solution (Fig.~\ref{fig:spread_129x65_16ws_2h}) (the spread is narrower).

\begin{figure}[H]
\centering
\begin{tabular}{ccc}
& \begin{minipage}{0.45\textwidth}\hspace*{1.4cm}\bf (a): EME for all weather stations\end{minipage} & \begin{minipage}{0.45\textwidth}\hspace*{1.4cm}\bf (b): EME for the whole domain\end{minipage}\\
\begin{minipage}{0.02\textwidth}\rotatebox{90}{EME}\end{minipage} & 
\hspace*{-0.25cm}\begin{minipage}{0.45\textwidth}\includegraphics[width=7.5cm]{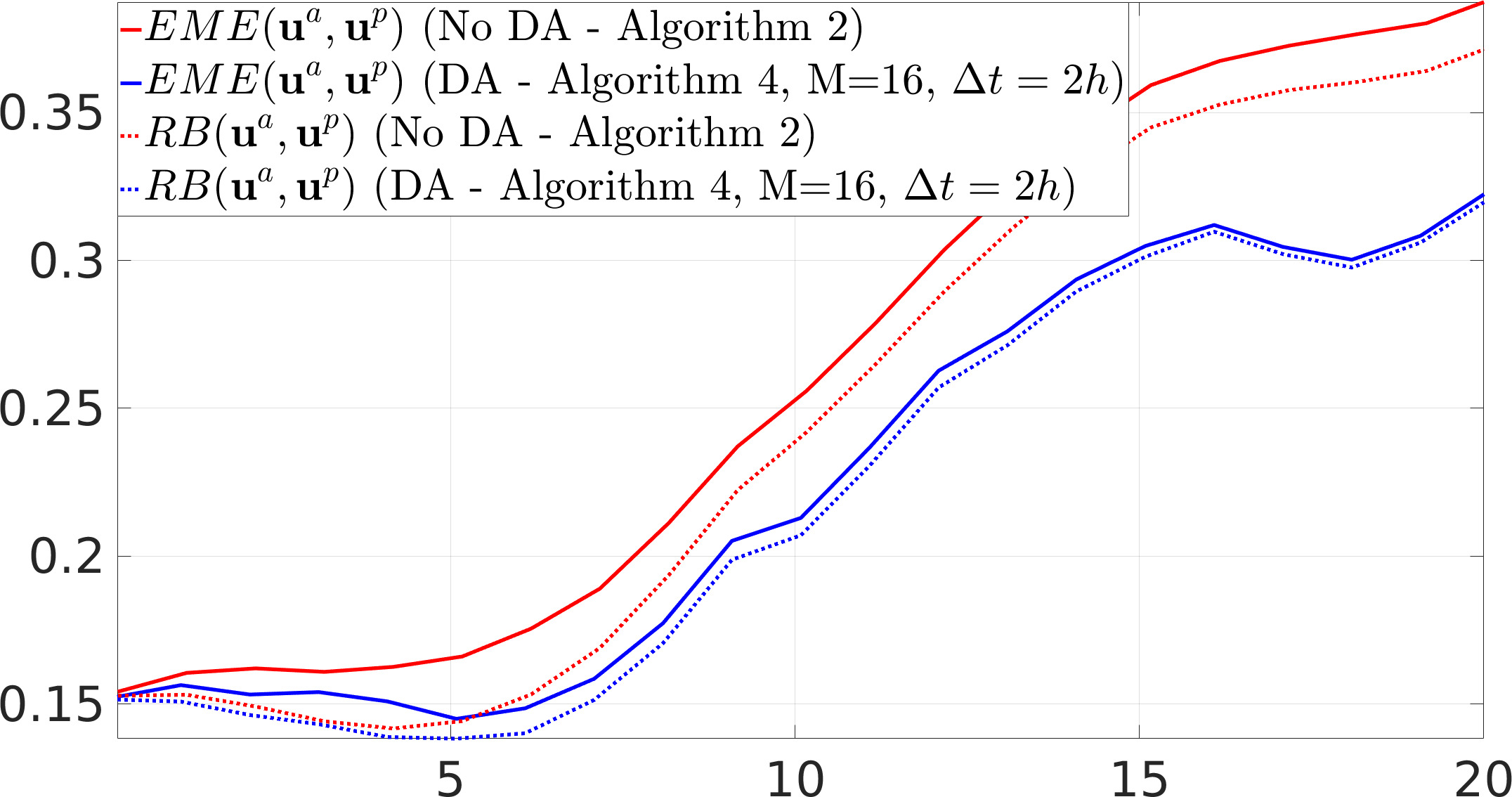}\end{minipage} &
\hspace*{-0.25cm}\begin{minipage}{0.45\textwidth}\includegraphics[width=7.5cm]{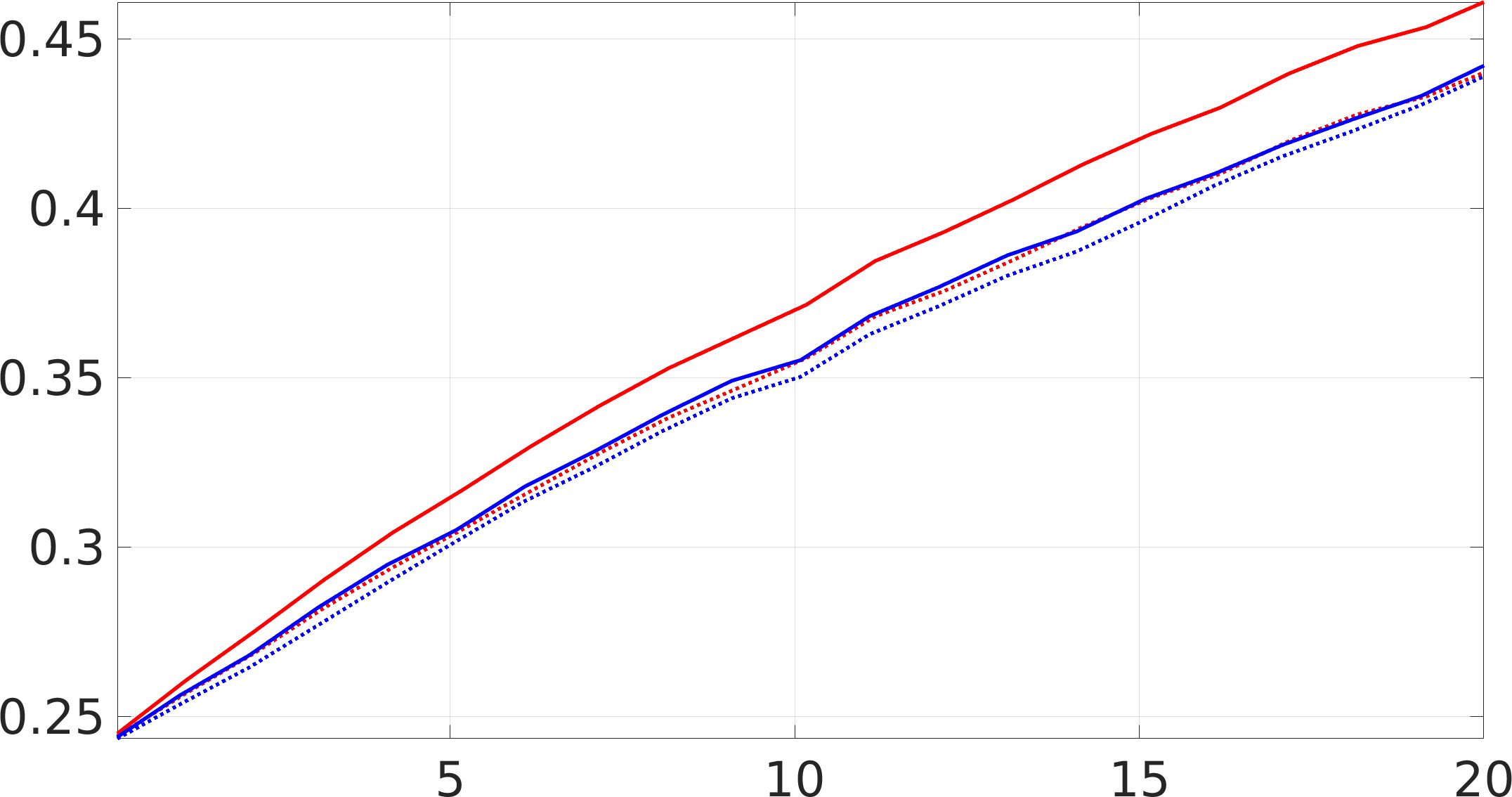}\end{minipage}\\
& &\\[-0.25cm]
&\hspace*{0.5cm}\begin{minipage}{0.1\textwidth}t {\rm [days]}\end{minipage} &\hspace*{0.5cm}\begin{minipage}{0.1\textwidth}t {\rm [days]}\end{minipage}
\end{tabular}
\caption{The same as in Fig.~\ref{fig:av_relerr_v_129x65}, but for the data assimilation step $\Delta t=2h$.}
\label{fig:av_relerr_v_129x65_dt_2h}
\end{figure}

\begin{figure}[H]
\centering
\hspace*{1.0cm}
\begin{tabular}{cc}
\hspace*{0cm}\begin{minipage}{0.1\textwidth} \large$u^p_1$ \end{minipage} &\hspace*{0cm}\begin{minipage}{0.1\textwidth} \large$v^p_1$ \end{minipage}\\
&\\[-0.35cm]
\hspace*{-0.25cm}\begin{minipage}{0.45\textwidth}\includegraphics[scale=0.25]{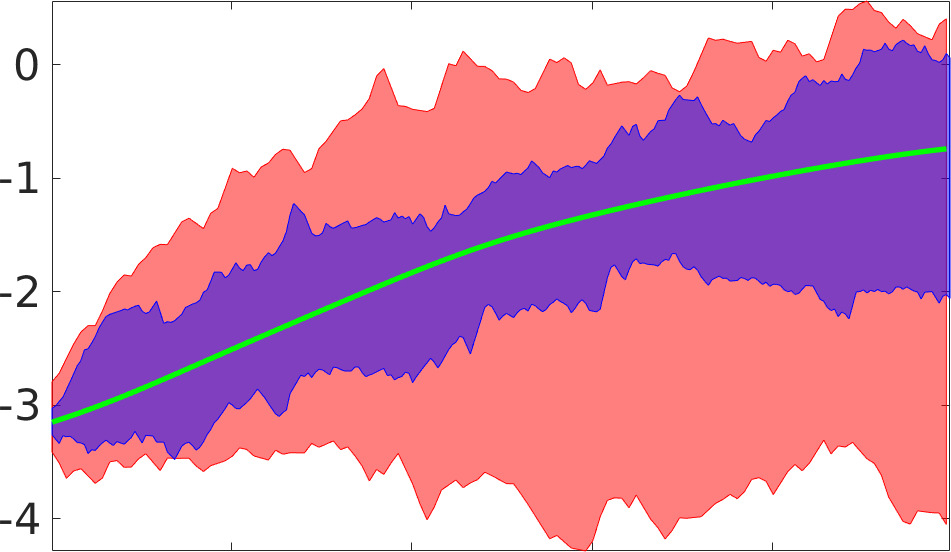}\end{minipage} &
\hspace*{-0.25cm}\begin{minipage}{0.45\textwidth}\includegraphics[scale=0.25]{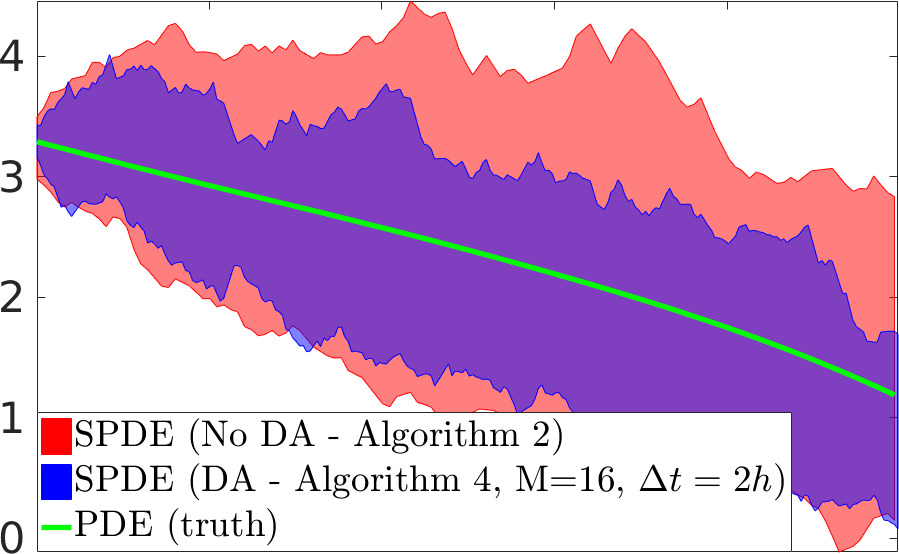}\end{minipage}\\
&\\[-0.25cm]
\hspace*{-0.75cm}\begin{minipage}{0.45\textwidth}\includegraphics[scale=0.25]{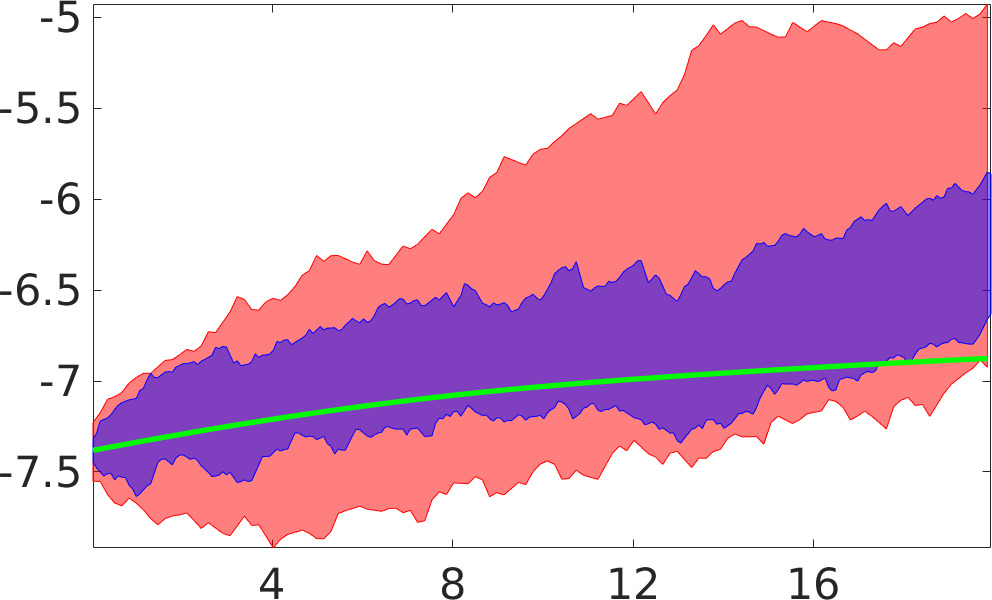}\end{minipage} &
\hspace*{-1cm}\begin{minipage}{0.45\textwidth}\includegraphics[scale=0.25]{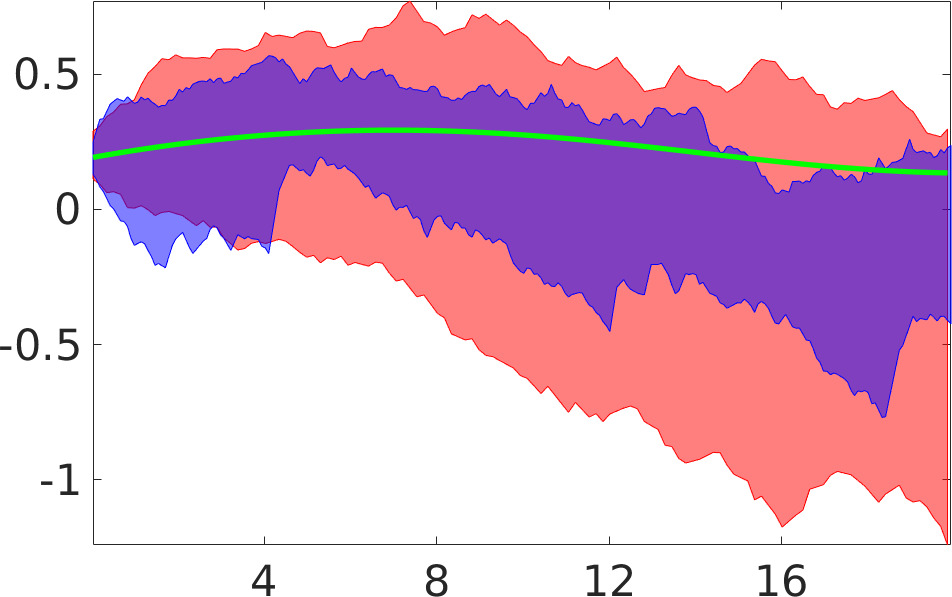}\end{minipage}\\
&\\[-0.25cm]
\hspace*{0cm}\begin{minipage}{0.1\textwidth}t {\rm [days]}\end{minipage} &\hspace*{-0.5cm}\begin{minipage}{0.1\textwidth}t {\rm [days]}\end{minipage}
\end{tabular}
\caption{The same as in Fig.~\ref{fig:spread_129x65_16ws_4h}, but for the data assimilation step $\Delta t=2h$.}
\label{fig:spread_129x65_16ws_2h}
\end{figure}

Further substantial improvements in the performance of the data assimilation methodology are obtained when the resolution of the signal grid gets 
higher ($G_s=257\times129$). In particular, the results are much more accurate both at the observation points (weather stations) (Fig.~\ref{fig:av_relerr_v_257x129}a) and over the whole domain (Fig.~\ref{fig:av_relerr_v_257x129}b). 
Moreover, the spread of the sample reduces dramatically as shown in figure  Fig.~\ref{fig:spread_257x129_16ws_4h}.

\begin{figure}[H]
\centering
\begin{tabular}{ccc}
& \begin{minipage}{0.45\textwidth}\hspace*{1.4cm}\bf (a): EME for all weather stations\end{minipage} & \begin{minipage}{0.45\textwidth}\hspace*{1.4cm}\bf (b): EME for the whole domain\end{minipage}\\
\begin{minipage}{0.02\textwidth}\rotatebox{90}{EME}\end{minipage} & 
\hspace*{-0.25cm}\begin{minipage}{0.45\textwidth}\includegraphics[width=7.5cm]{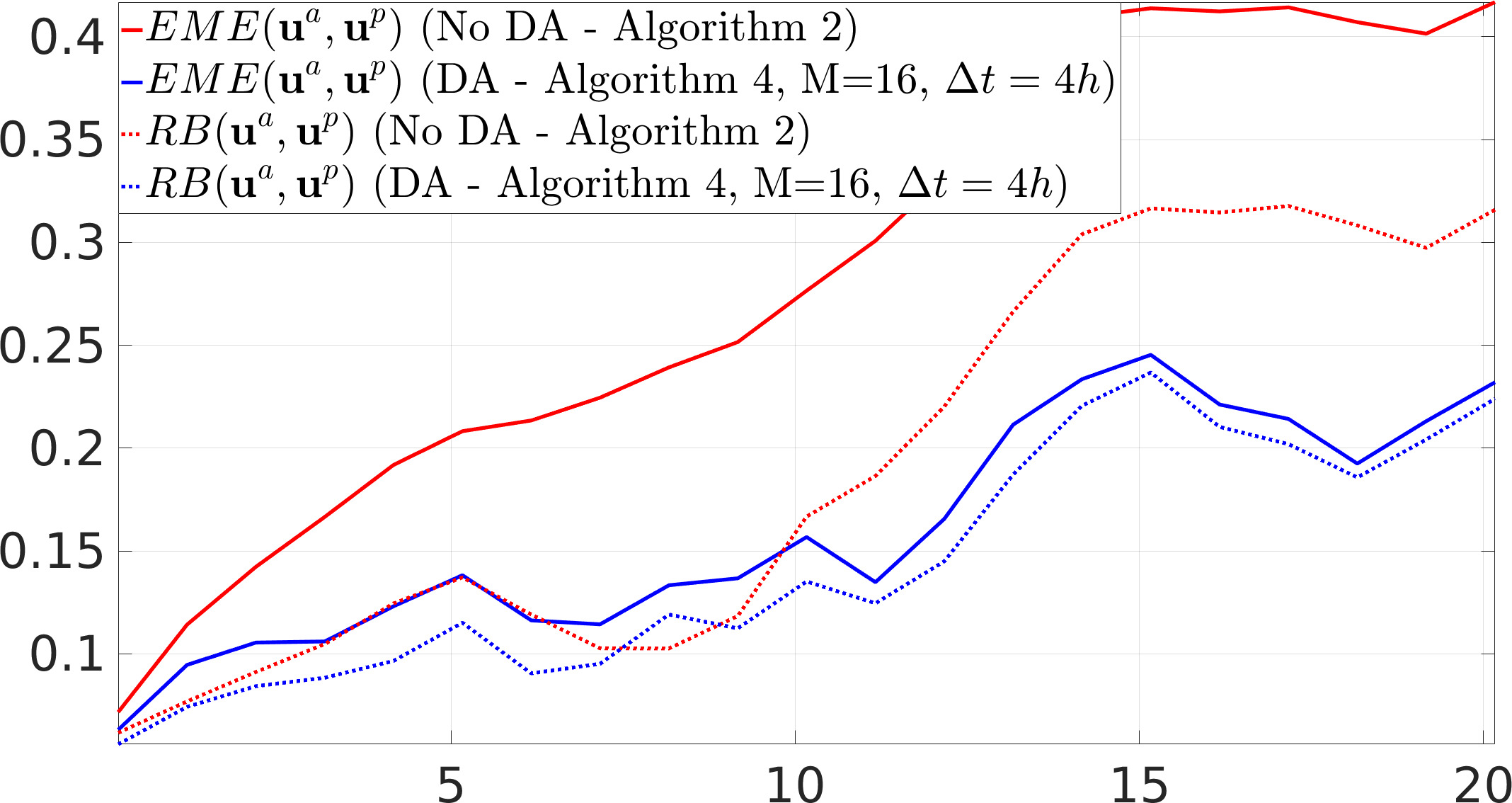}\end{minipage} &
\hspace*{-0.25cm}\begin{minipage}{0.45\textwidth}\includegraphics[width=7.5cm]{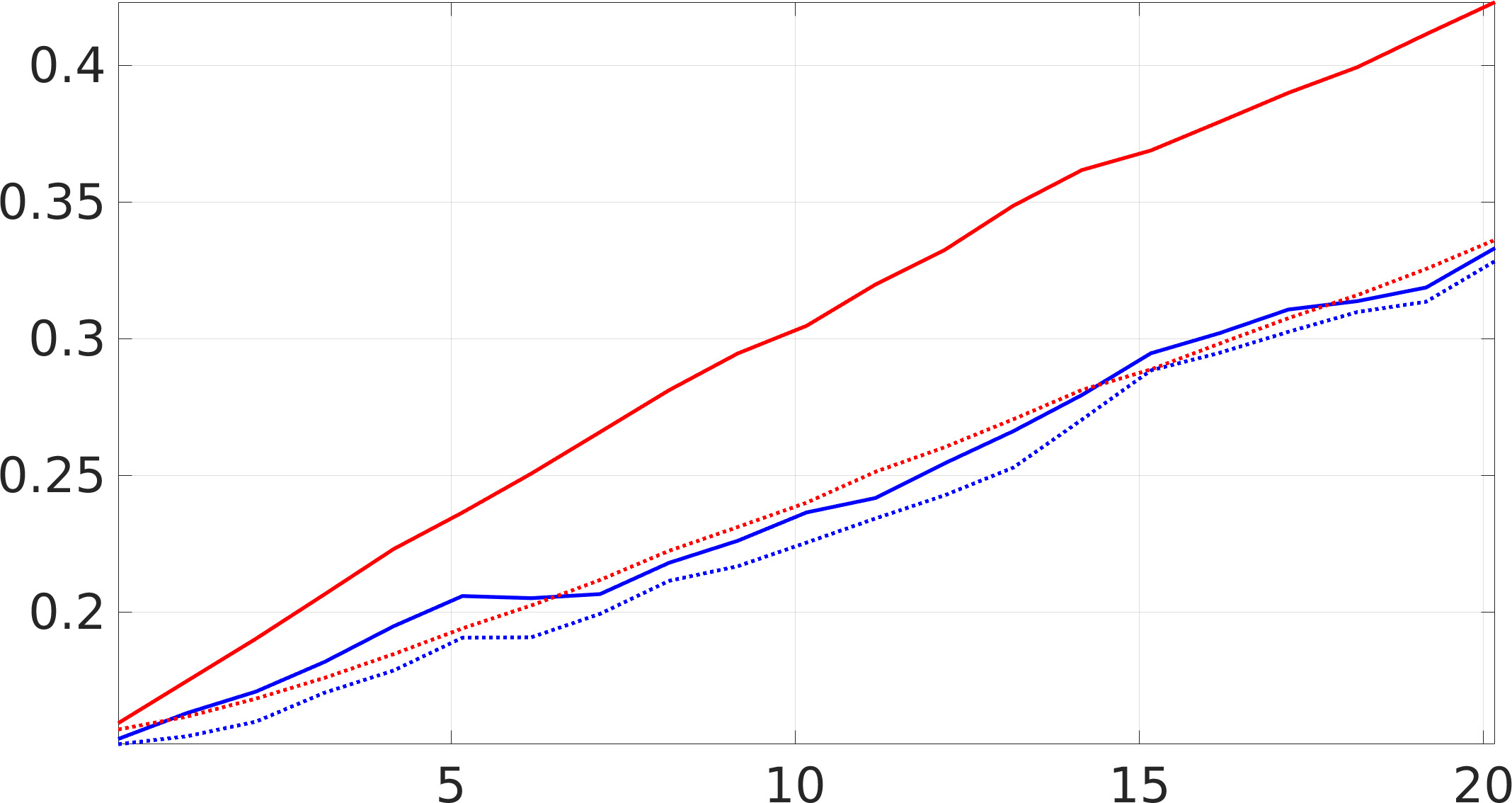}\end{minipage}\\
& &\\[-0.25cm]
&\hspace*{0.5cm}\begin{minipage}{0.1\textwidth}t {\rm [days]}\end{minipage} &\hspace*{0.5cm}\begin{minipage}{0.1\textwidth}t {\rm [days]}\end{minipage}
\end{tabular}
\caption{The same as in Fig.~\ref{fig:av_relerr_v_129x65}, but for the signal grid $G_s=257\times129$.}
\label{fig:av_relerr_v_257x129}
\end{figure}

\begin{figure}[H]
\centering
\hspace*{1.0cm}
\begin{tabular}{cc}
\hspace*{0cm}\begin{minipage}{0.1\textwidth} \large$u^p_1$ \end{minipage} &\hspace*{0cm}\begin{minipage}{0.1\textwidth} \large$v^p_1$ \end{minipage}\\
&\\[-0.35cm]
\hspace*{-0.25cm}\begin{minipage}{0.45\textwidth}\includegraphics[scale=0.25]{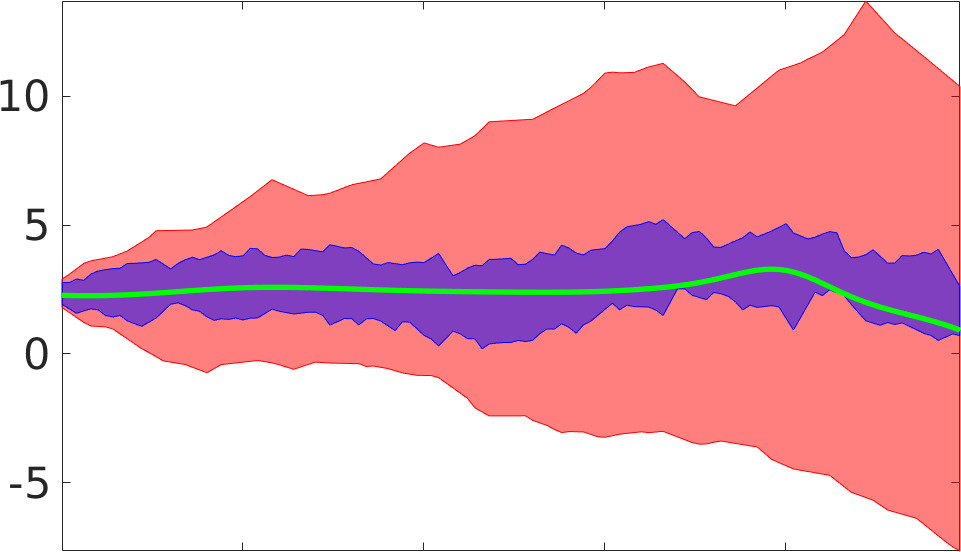}\end{minipage} &
\hspace*{-0.25cm}\begin{minipage}{0.45\textwidth}\includegraphics[scale=0.25]{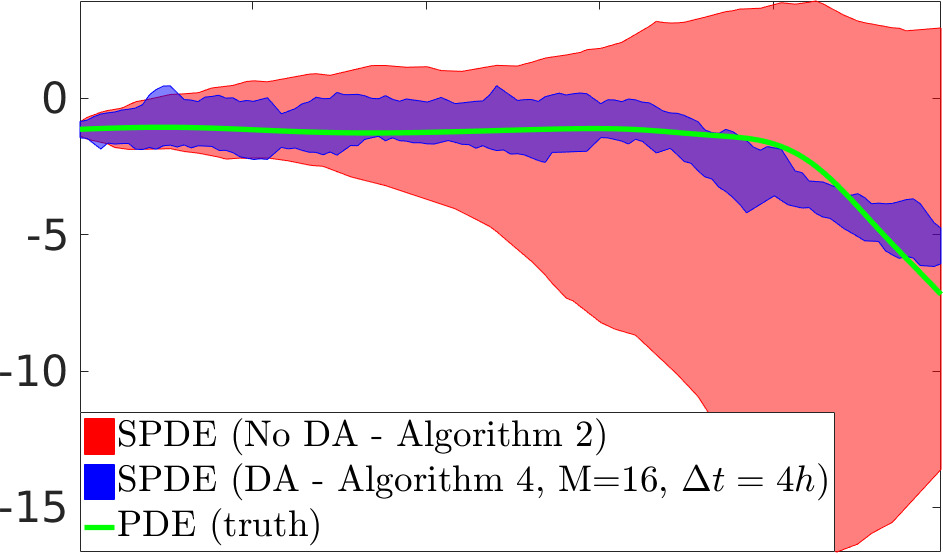}\end{minipage}\\
&\\[-0.25cm]
\hspace*{-0.25cm}\begin{minipage}{0.45\textwidth}\includegraphics[scale=0.25]{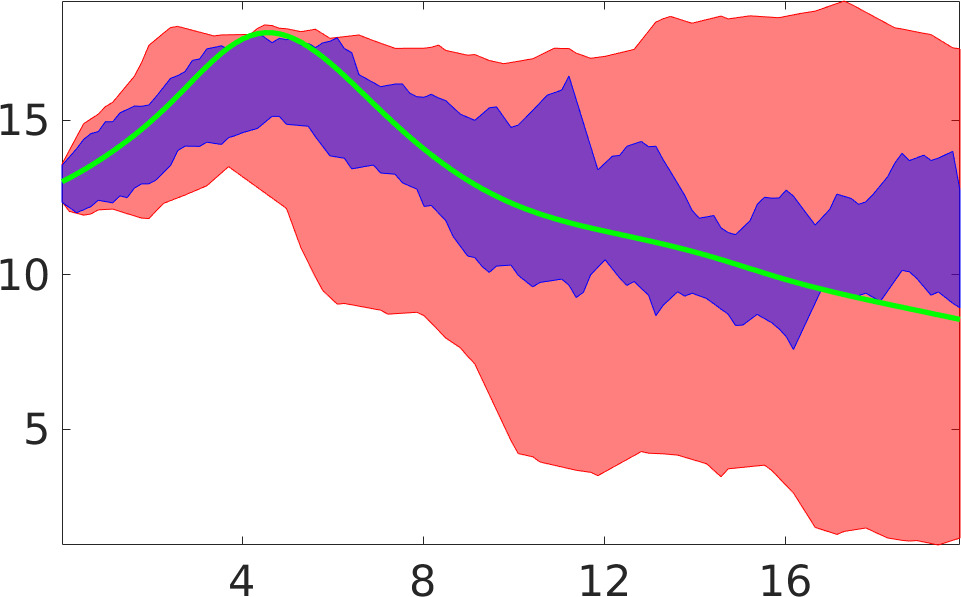}\end{minipage} &
\hspace*{-0.25cm}\begin{minipage}{0.45\textwidth}\includegraphics[scale=0.25]{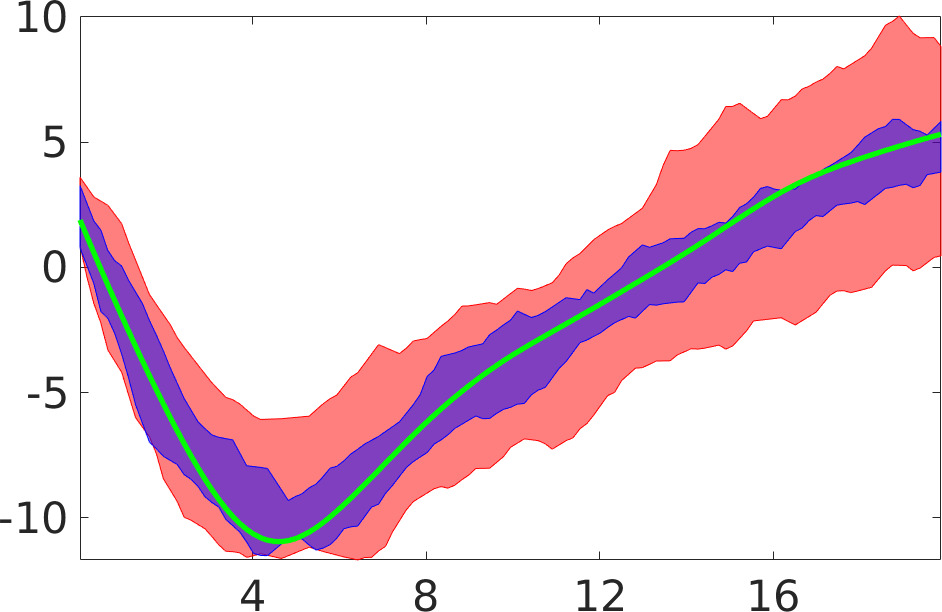}\end{minipage}\\
&\\[-0.25cm]
\hspace*{0cm}\begin{minipage}{0.1\textwidth}t {\rm [days]}\end{minipage} &\hspace*{-0.5cm}\begin{minipage}{0.1\textwidth}t {\rm [days]}\end{minipage}
\end{tabular}
\caption{The same as in Fig.~\ref{fig:spread_129x65_16ws_4h}, but for the signal grid $G_s=257\times129$.}
\label{fig:spread_257x129_16ws_4h}
\end{figure}

Based on the above results, we conclude that both the reduction of the data assimilation step and the increase of the resolution of the signal grid can enhance the accuracy
of the stochastic solution and reduce the spread of the stochastic ensemble. The effect of the increase of the resolution on both the 
accuracy and uncertainty appears to be much more pronounced compared to that of the reduction of the data assimilation step.

\subsection{Nudging}
\label{sec:nudging}
We will now look at the performance of the data assimilation methodology that includes nudging (Algorithm 5) compared with the one that does not (Algorithm 4). 
As above we start with the stochastic QG model at the coarsest resolution ($G_s=129\times65$) and use $M=16$ weather stations. We present the results in Figs.~\ref{fig:av_relerr_v_129x65_NDG_ws16} and~\ref{fig:spread_129x65_16ws_4h_nudging}. The improvements are obvious straightaway. 
The average of the stochastic spread becomes closer to the true solution compared with the same case but without using
the nudging procedure (Fig.~\ref{fig:av_relerr_v_129x65_NDG_ws16}).
However, in some cases, the true solution leaves the spread of the ensemble (Fig.~\ref{fig:spread_129x65_16ws_4h_nudging}).
We do not have a clear explanation for this behaviour. 

\begin{figure}[H]
\centering
\begin{tabular}{ccc}
& \begin{minipage}{0.45\textwidth}\hspace*{1.4cm}\bf (a): EME for all weather stations\end{minipage} & \begin{minipage}{0.45\textwidth}\hspace*{1.4cm}\bf (b): EME for the whole domain\end{minipage}\\
\begin{minipage}{0.02\textwidth}\rotatebox{90}{EME}\end{minipage} & 
\hspace*{-0.25cm}\begin{minipage}{0.45\textwidth}\includegraphics[width=7.5cm]{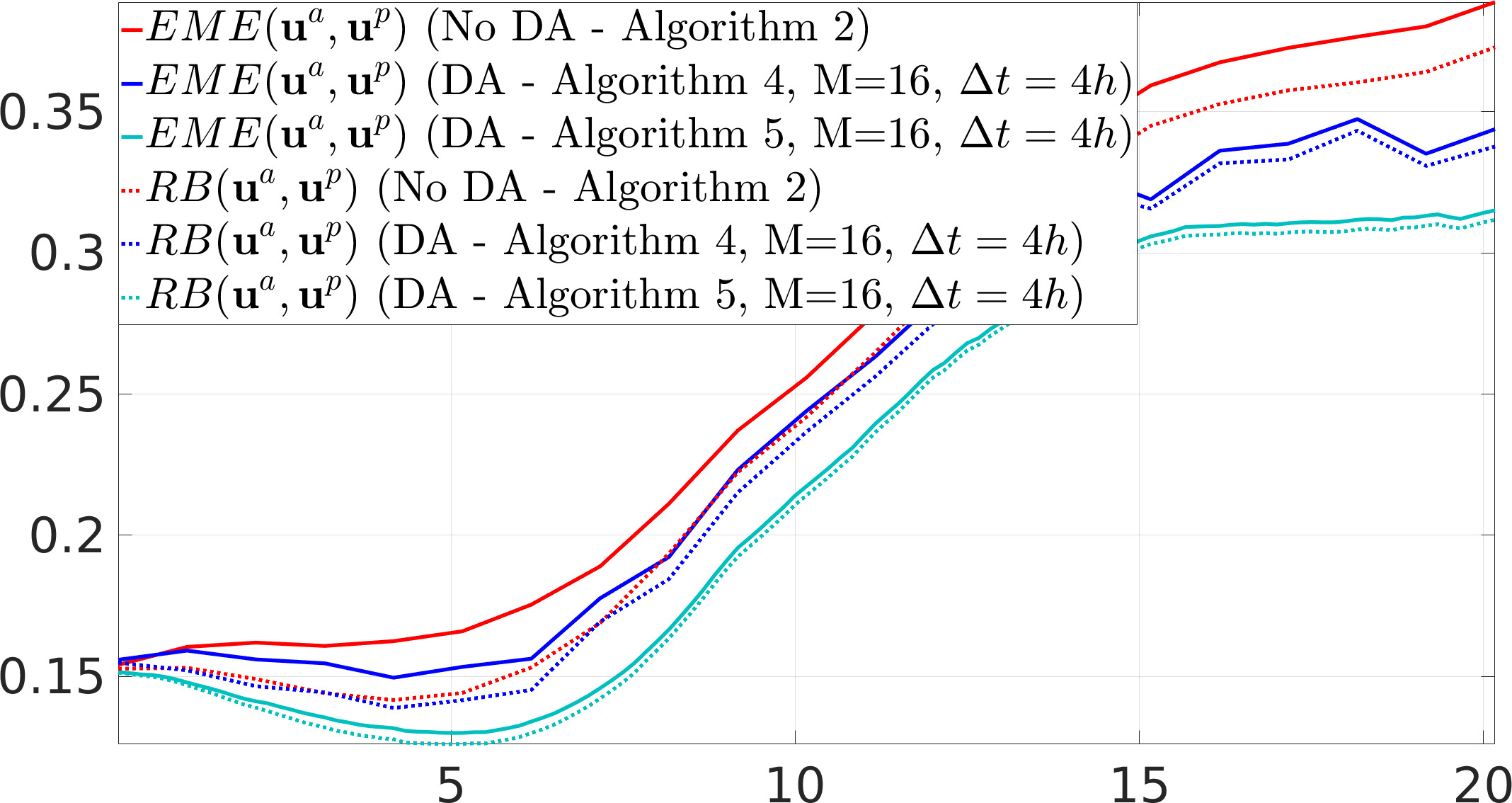}\end{minipage} &
\hspace*{-0.25cm}\begin{minipage}{0.45\textwidth}\includegraphics[width=7.5cm]{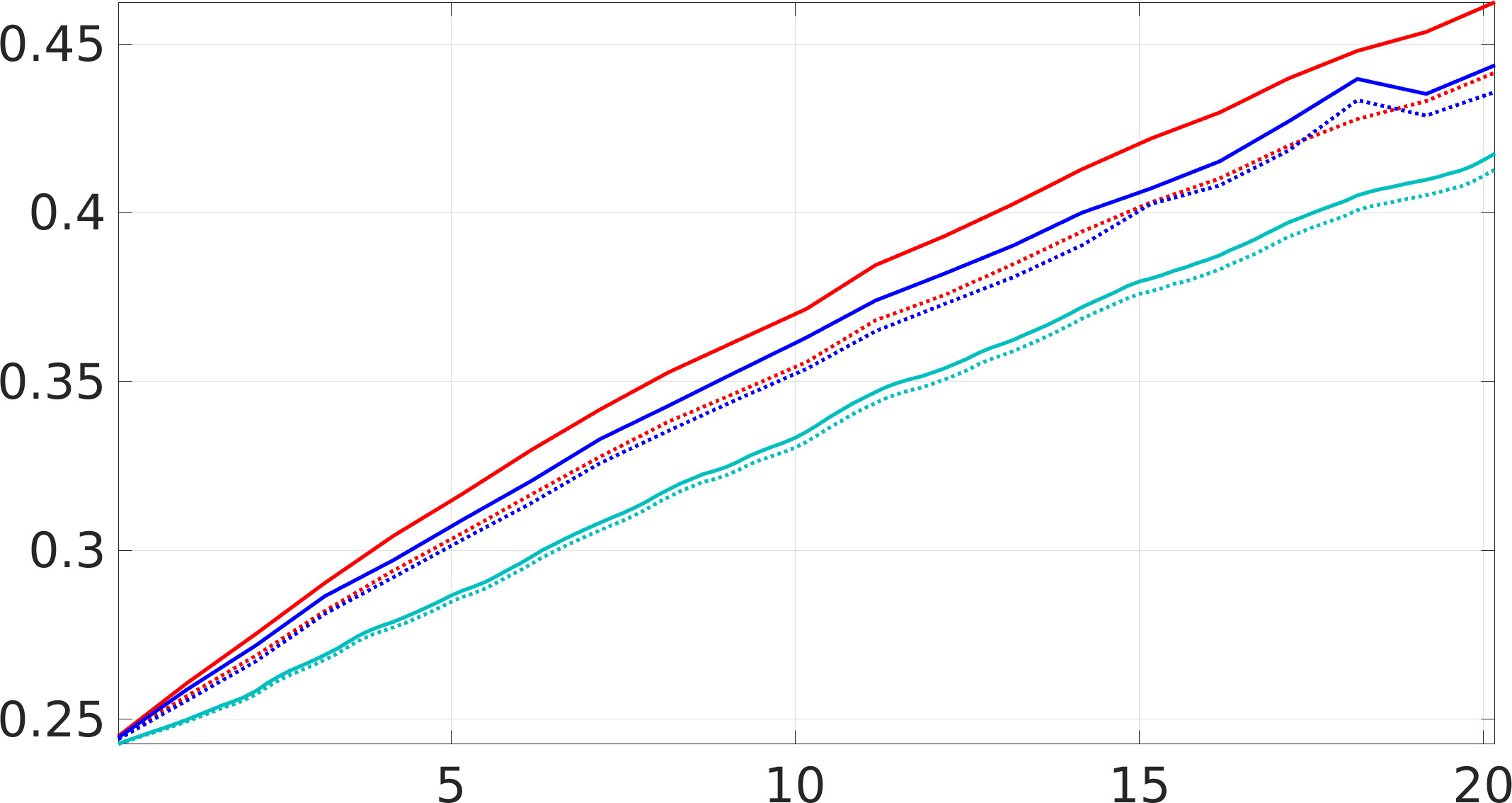}\end{minipage}\\
& &\\[-0.25cm]
&\hspace*{0.5cm}\begin{minipage}{0.1\textwidth}t {\rm [days]}\end{minipage} &\hspace*{0.5cm}\begin{minipage}{0.1\textwidth}t {\rm [days]}\end{minipage}
\end{tabular}
\caption{Evolution of the ensemble mean relative $l_2$-norm error (EME) and relative bias (RB) for \textbf{(a)} all weather stations and \textbf{(b)} the whole domain; $\mathbf{u}^a$ is the true solution, $\mathbf{u}^p$ is the stochastic solution.
In order to assimilate data we use \textit{tempering, jittering, and nudging (Algorithm 5)}; 
the data is assimilated from $M=16$ weather stations every 4 hours; the grid size is $G_s=129\times65$.}
\label{fig:av_relerr_v_129x65_NDG_ws16}
\end{figure}

\begin{figure}[H]
\centering
\hspace*{1.0cm}
\begin{tabular}{cc}
\hspace*{0cm}\begin{minipage}{0.1\textwidth} \large$u^p_1$ \end{minipage} &\hspace*{0cm}\begin{minipage}{0.1\textwidth} \large$v^p_1$ \end{minipage}\\
&\\[-0.35cm]
\hspace*{-0.25cm}\begin{minipage}{0.45\textwidth}\includegraphics[scale=0.25]{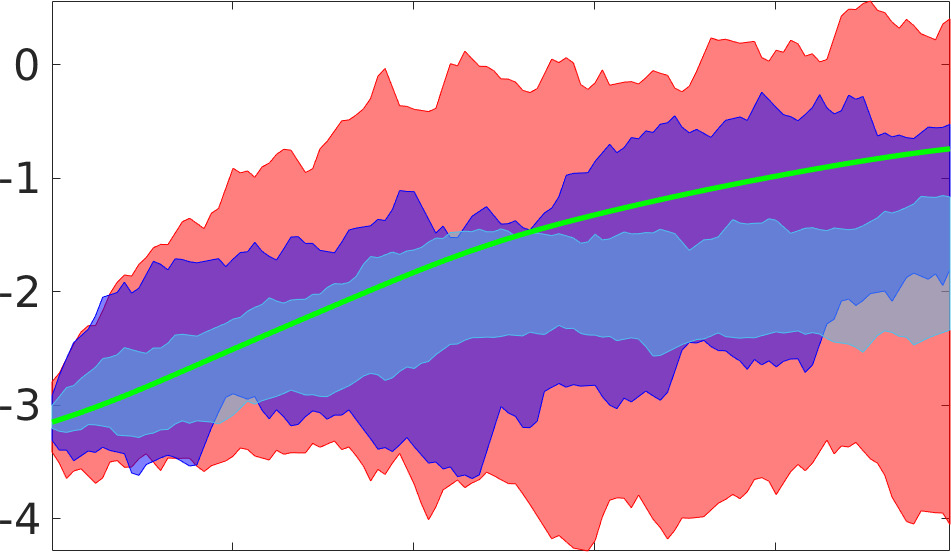}\end{minipage} &
\hspace*{-0.25cm}\begin{minipage}{0.45\textwidth}\includegraphics[scale=0.25]{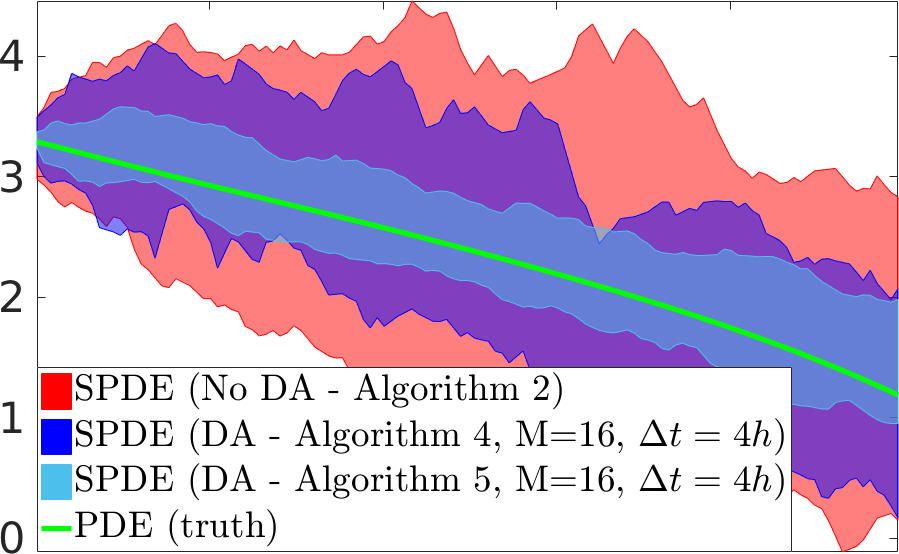}\end{minipage}\\
&\\[-0.25cm]
\hspace*{-0.75cm}\begin{minipage}{0.45\textwidth}\includegraphics[scale=0.25]{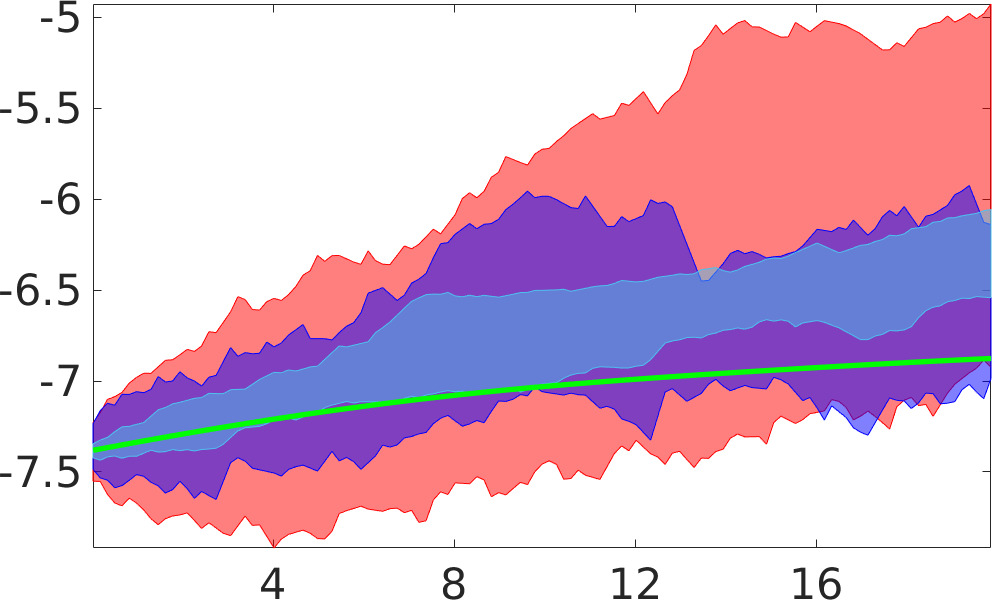}\end{minipage} &
\hspace*{-1cm}\begin{minipage}{0.45\textwidth}\includegraphics[scale=0.25]{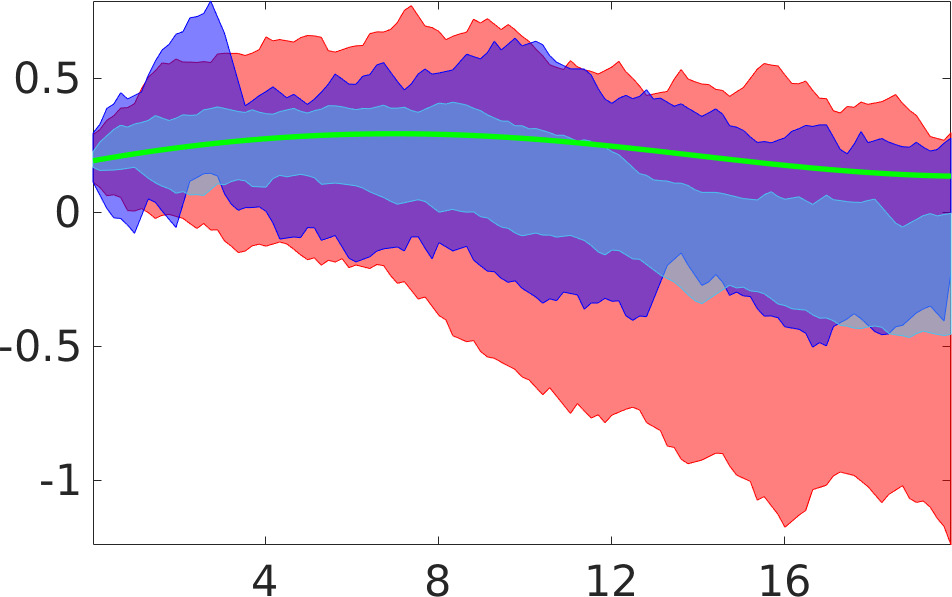}\end{minipage}\\
&\\[-0.25cm]
\hspace*{0cm}\begin{minipage}{0.1\textwidth}t {\rm [days]}\end{minipage} &\hspace*{-0.5cm}\begin{minipage}{0.1\textwidth}t {\rm [days]}\end{minipage}
\end{tabular}
\caption{Shown are typical stochastic spreads of velocity $\mathbf{u}^p_1=(u^p_1,v^p_1)$ at weather stations
located in the slow flow region (upper row) and fast flow region (lower row) 
for the SPDE without (red) and with (light blue) using \textit{tempering, jittering, and nudging (Algorithm 5)};
\textit{Algorithm 4} (blue) is given for ease of comparison.
The green line is the true solution; the grid size is $G_s=129\times65$.
}
\label{fig:spread_129x65_16ws_4h_nudging}
\end{figure}

Again, at the higher resolution ($G_s=257\times129$), the nudged solution is even more accurate than its low-resolution version
(Fig.~\ref{fig:av_relerr_v_257x129_NDG_ws16}), and the uncertainty is further reduced (Fig.~\ref{fig:spread_257x129_16ws_4h_nudging}).
\begin{figure}[H]
\centering
\begin{tabular}{ccc}
& \begin{minipage}{0.45\textwidth}\hspace*{1.4cm}\bf (a): EME for all weather stations\end{minipage} & \begin{minipage}{0.45\textwidth}\hspace*{1.4cm}\bf (b): EME for the whole domain\end{minipage}\\
\begin{minipage}{0.02\textwidth}\rotatebox{90}{EME}\end{minipage} & 
\hspace*{-0.25cm}\begin{minipage}{0.45\textwidth}\includegraphics[width=7.5cm]{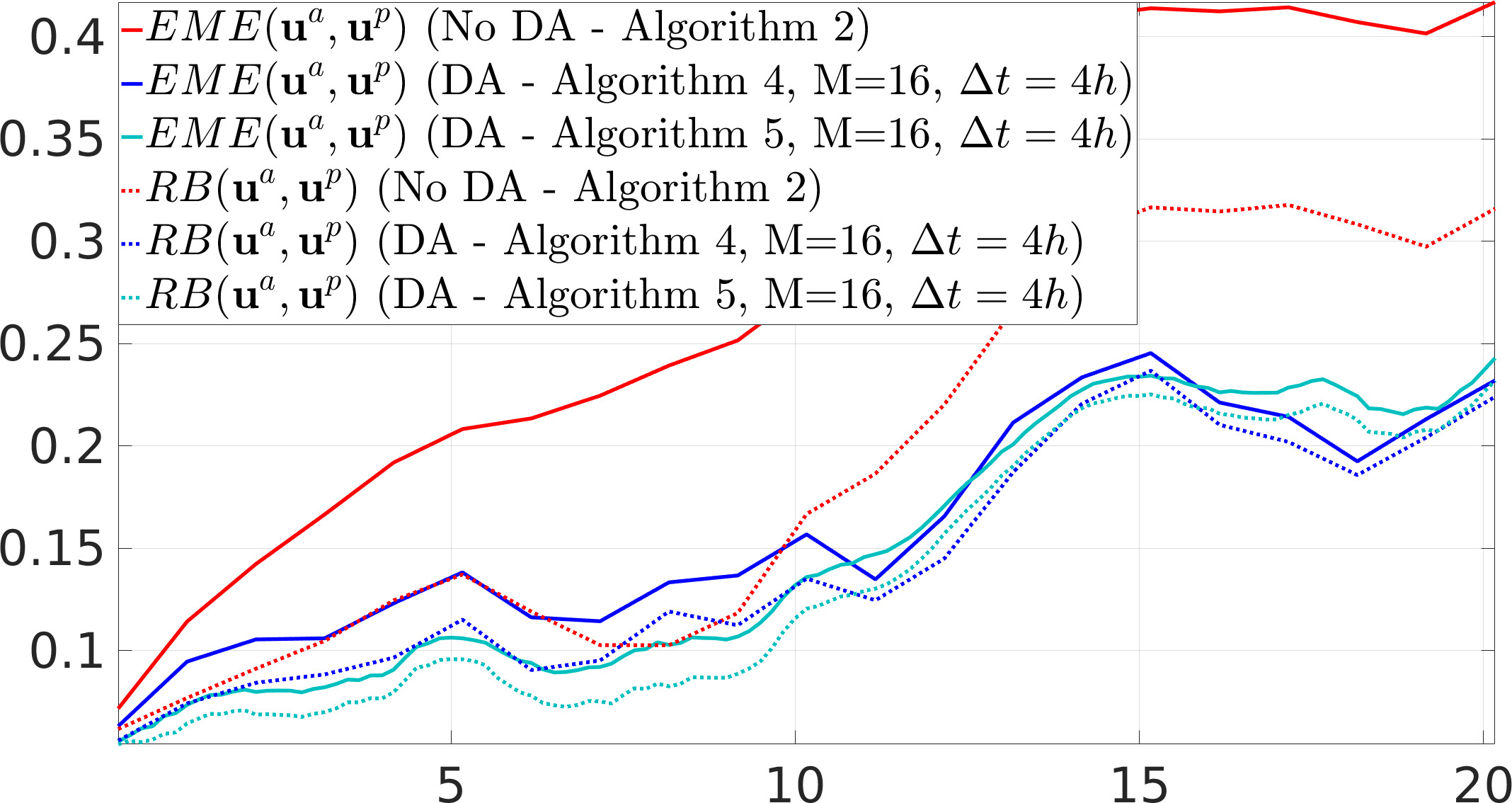}\end{minipage} &
\hspace*{-0.25cm}\begin{minipage}{0.45\textwidth}\includegraphics[width=7.5cm]{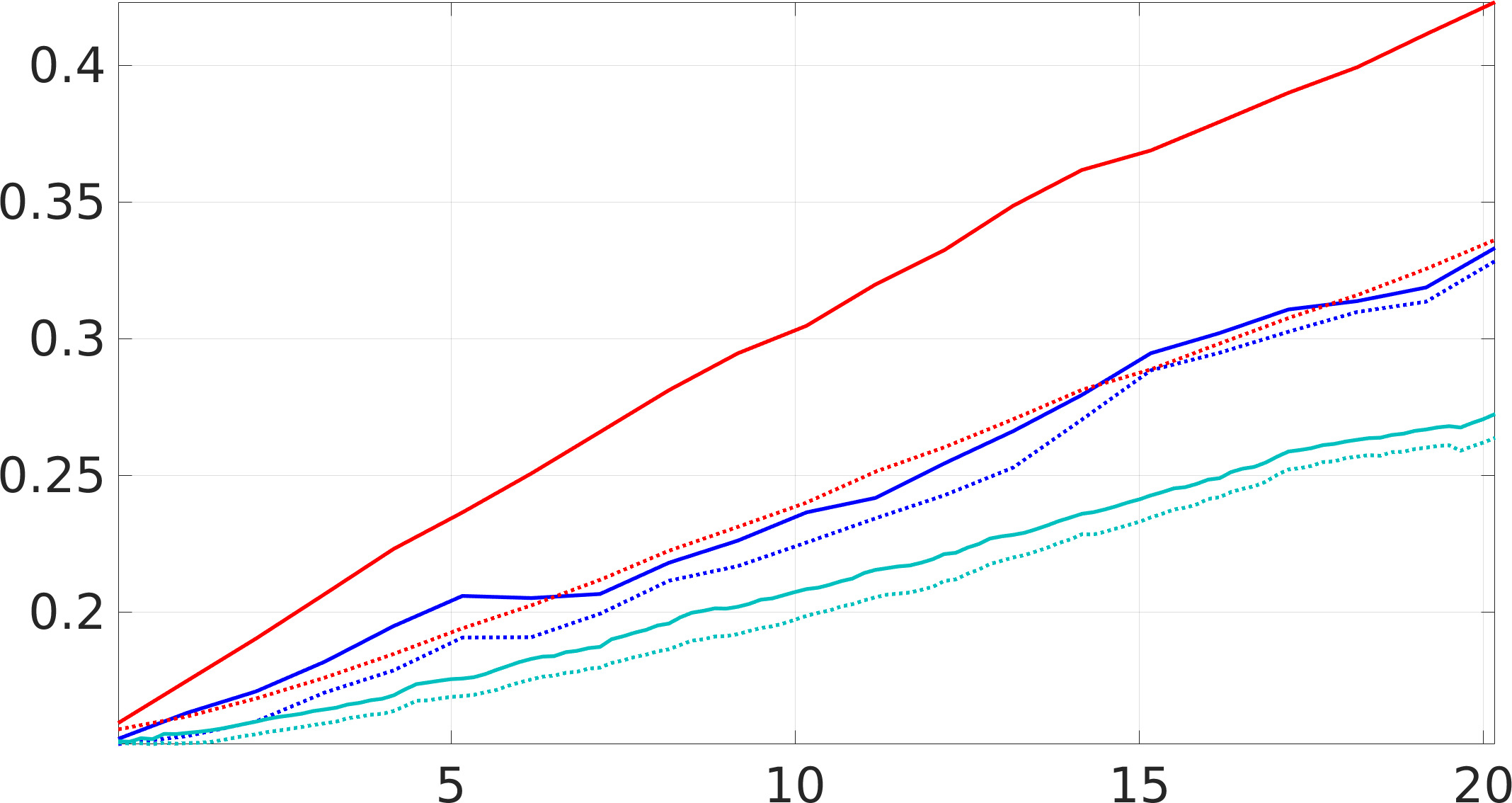}\end{minipage}\\
& &\\[-0.25cm]
&\hspace*{0.5cm}\begin{minipage}{0.1\textwidth}t {\rm [days]}\end{minipage} &\hspace*{0.5cm}\begin{minipage}{0.1\textwidth}t {\rm [days]}\end{minipage}
\end{tabular}
\caption{The same as in Fig.~\ref{fig:av_relerr_v_129x65_NDG_ws16}, but for the signal grid $G_s=257\times129$.}
\label{fig:av_relerr_v_257x129_NDG_ws16}
\end{figure}

\begin{figure}[H]
\centering
\hspace*{1.0cm}
\begin{tabular}{cc}
\hspace*{0cm}\begin{minipage}{0.1\textwidth} \large$u^p_1$ \end{minipage} &\hspace*{0cm}\begin{minipage}{0.1\textwidth} \large$v^p_1$ \end{minipage}\\
&\\[-0.35cm]
\hspace*{-0.25cm}\begin{minipage}{0.45\textwidth}\includegraphics[scale=0.25]{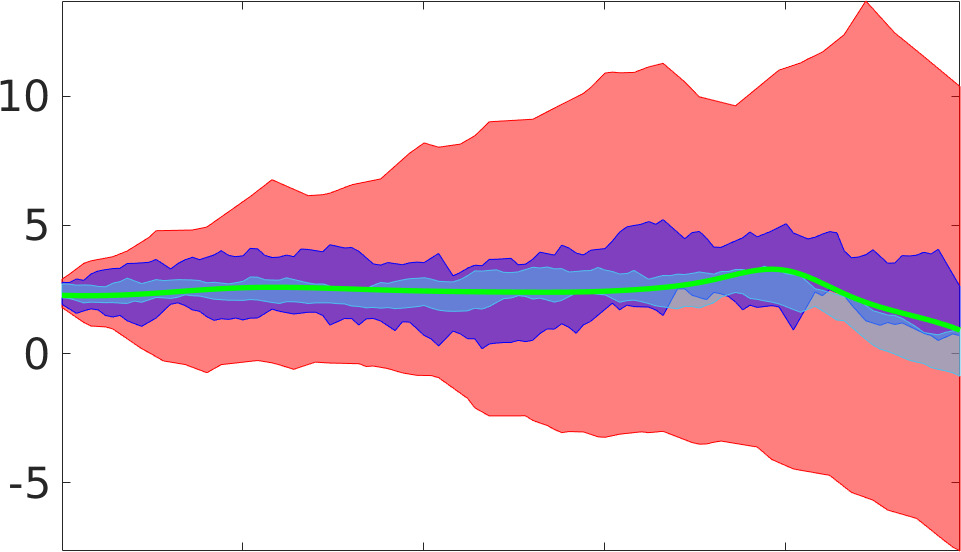}\end{minipage} &
\hspace*{-0.25cm}\begin{minipage}{0.45\textwidth}\includegraphics[scale=0.25]{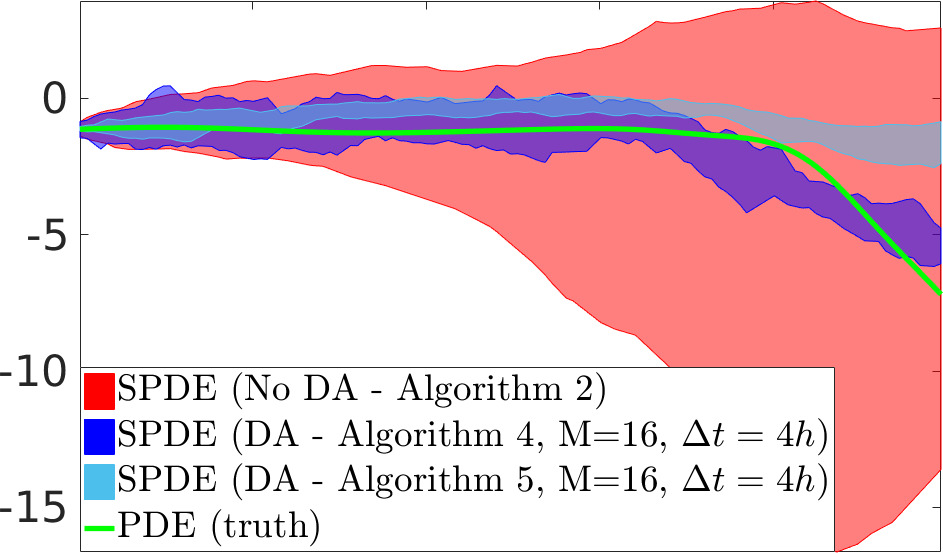}\end{minipage}\\
&\\[-0.25cm]
\hspace*{-0.25cm}\begin{minipage}{0.45\textwidth}\includegraphics[scale=0.25]{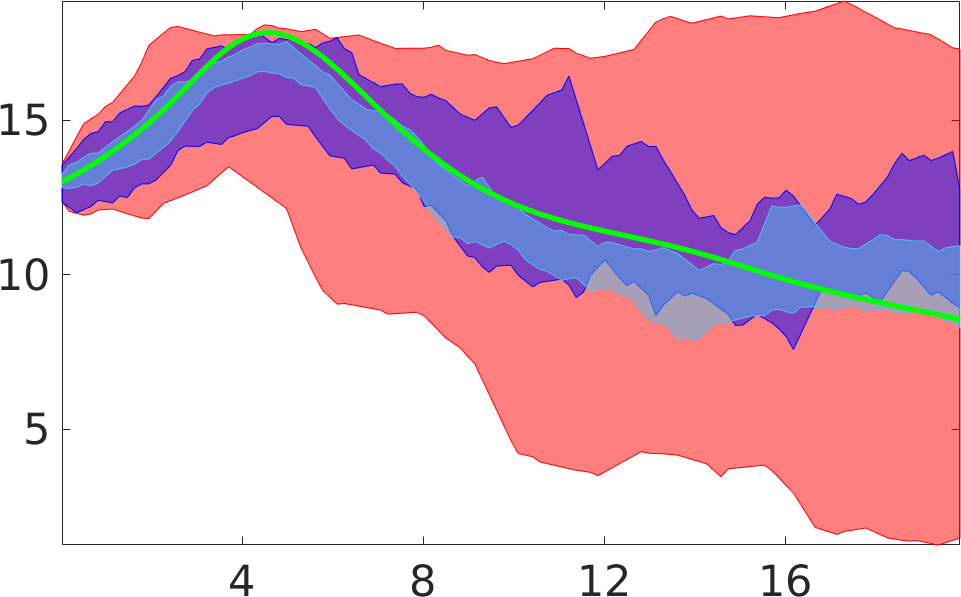}\end{minipage} &
\hspace*{-0.25cm}\begin{minipage}{0.45\textwidth}\includegraphics[scale=0.25]{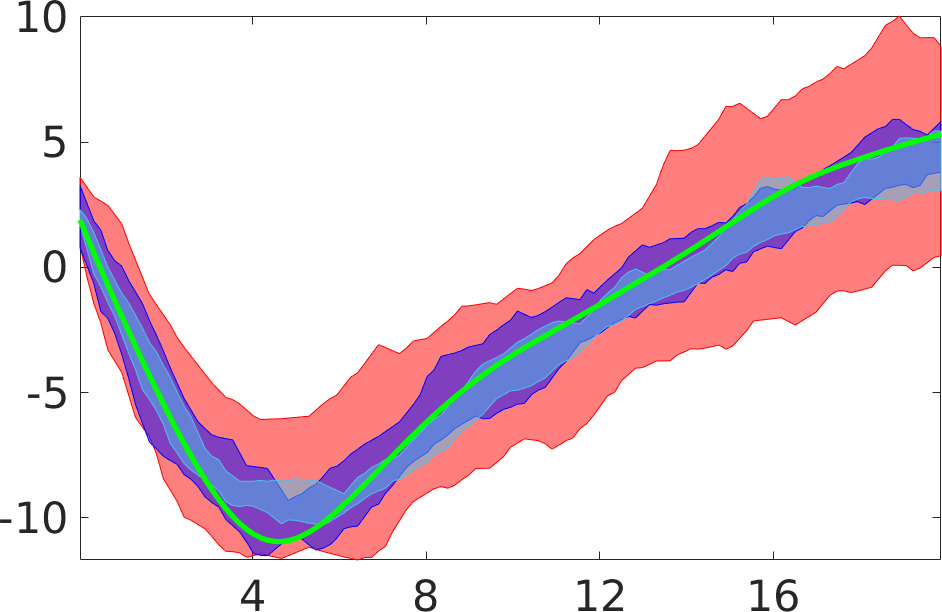}\end{minipage}\\
&\\[-0.25cm]
\hspace*{0cm}\begin{minipage}{0.1\textwidth}t {\rm [days]}\end{minipage} &\hspace*{-0.5cm}\begin{minipage}{0.1\textwidth}t {\rm [days]}\end{minipage}
\end{tabular}
\caption{The same as in Figure~\ref{fig:spread_129x65_16ws_4h_nudging}, but for the signal grid $G_s=257\times129$.}
\label{fig:spread_257x129_16ws_4h_nudging}
\end{figure}

It is important to note that the number of observation locations (weather stations) used in the simulations above is only 0.19\% and 0.05\% of all degrees of freedom 
for the grids $G_s=129\times65$ and $G_s=257\times129$, respectively. 
Obviously, this number of weather stations is not enough to significantly reduce
the uncertainty and decrease the error between the true solution and its parameterisation. Therefore, in the next simulation 
we double the number of observation locations ($M=32$) and compare how Algorithm 5 (data assimilation with the nudging method) 
performs when more observational data is available.

As can be seen in Fig.~\ref{fig:av_relerr_v_129x65_NDG_ws32}, adding more weather stations does not significantly influence the results
for the low-resolution (Fig.~\ref{fig:av_relerr_v_129x65_NDG_ws32}) and higher-resolution (Fig.~\ref{fig:av_relerr_v_129x65_NDG_ws32})  simulation.

\begin{figure}[H]
\centering
\begin{tabular}{ccc}
& \begin{minipage}{0.45\textwidth}\hspace*{1.4cm}\bf (a): EME for all weather stations\end{minipage} & \begin{minipage}{0.45\textwidth}\hspace*{1.4cm}\bf (b): EME for the whole domain\end{minipage}\\
\begin{minipage}{0.02\textwidth}\rotatebox{90}{EME}\end{minipage} & 
\hspace*{-0.25cm}\begin{minipage}{0.45\textwidth}\includegraphics[width=7.5cm]{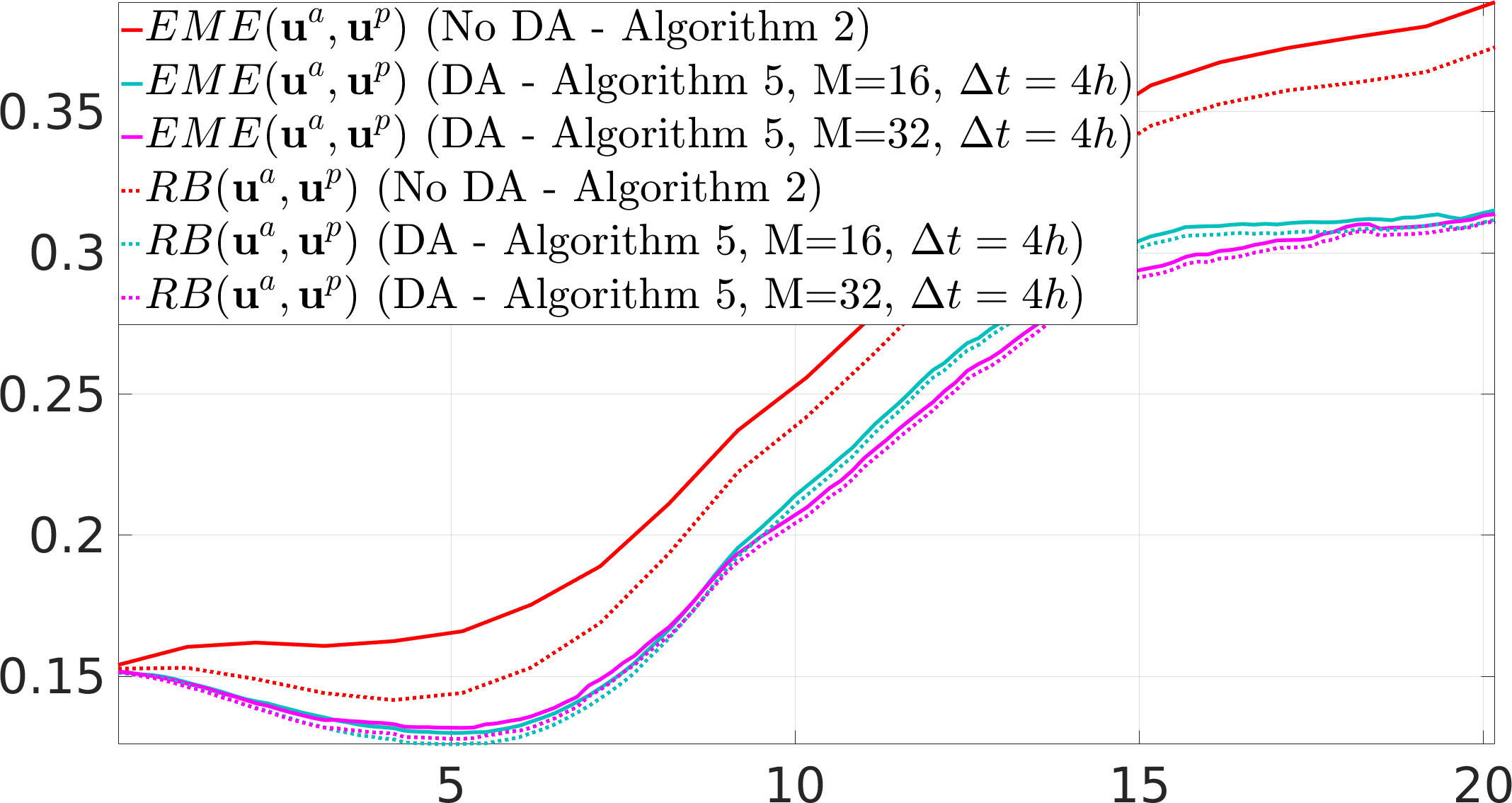}\end{minipage} &
\hspace*{-0.25cm}\begin{minipage}{0.45\textwidth}\includegraphics[width=7.5cm]{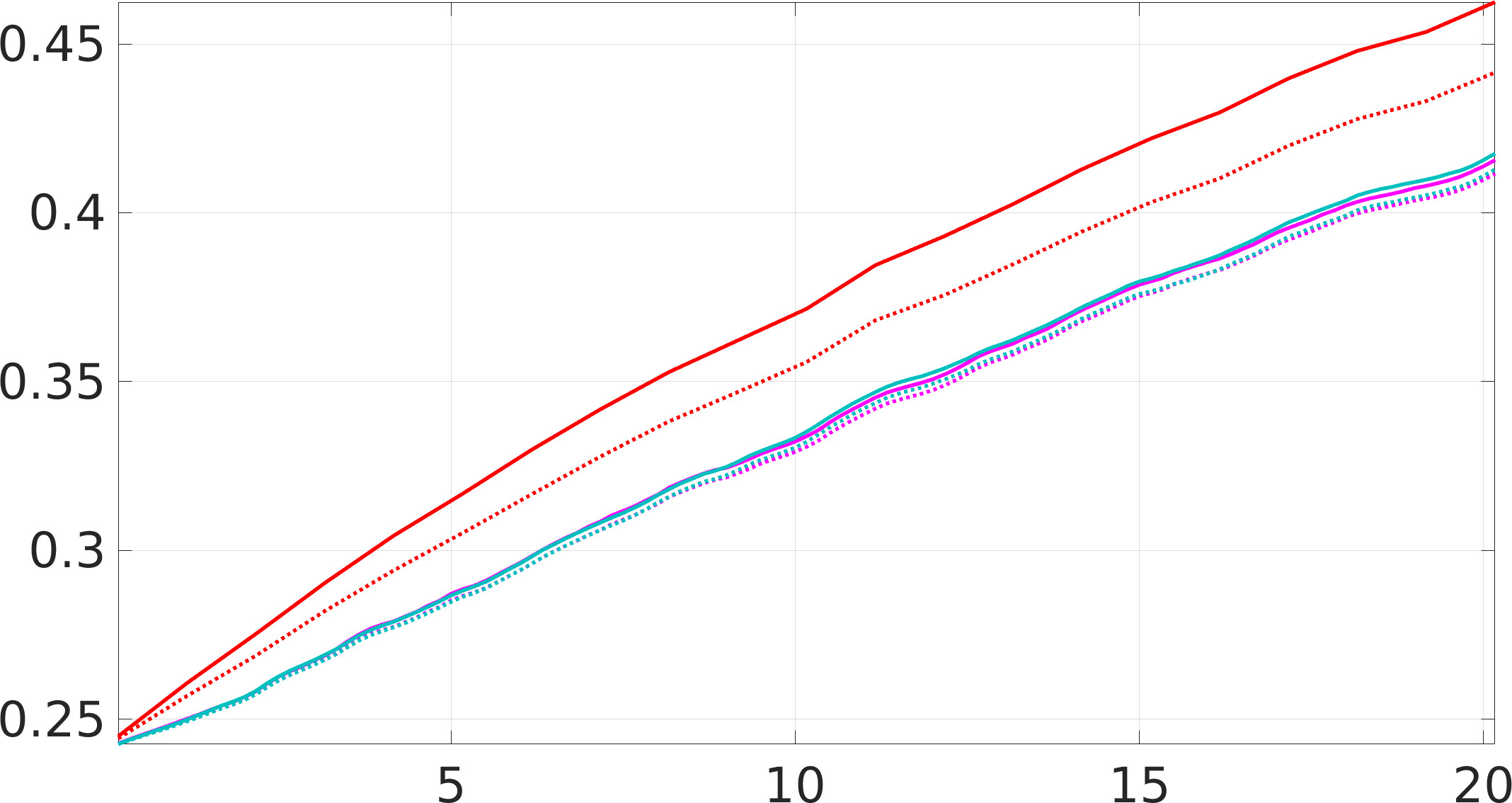}\end{minipage}\\
& &\\[-0.25cm]
&\hspace*{0.5cm}\begin{minipage}{0.1\textwidth}t {\rm [days]}\end{minipage} &\hspace*{0.5cm}\begin{minipage}{0.1\textwidth}t {\rm [days]}\end{minipage}
\end{tabular}
\caption{Evolution of the ensemble mean relative $l_2$-norm error (EME) and relative bias (RB) for \textbf{(a)} all weather stations and \textbf{(b)} the whole domain; $\mathbf{u}^a$ is the true solution, $\mathbf{u}^p$ is the stochastic solution.
In order to assimilate data we use \textit{tempering, jittering, and nudging (Algorithm 5)}; 
the data is assimilated from $M=32$ weather stations every 4 hours; the grid size is $G_s=129\times65$.}
\label{fig:av_relerr_v_129x65_NDG_ws32}
\end{figure}

\begin{figure}
\centering
\begin{tabular}{ccc}
& \begin{minipage}{0.45\textwidth}\hspace*{1.4cm}\bf (a): EME for all weather stations\end{minipage} & \begin{minipage}{0.45\textwidth}\hspace*{1.4cm}\bf (b): EME for the whole domain\end{minipage}\\
\begin{minipage}{0.02\textwidth}\rotatebox{90}{EME}\end{minipage} & 
\hspace*{-0.25cm}\begin{minipage}{0.45\textwidth}\includegraphics[width=7.5cm]{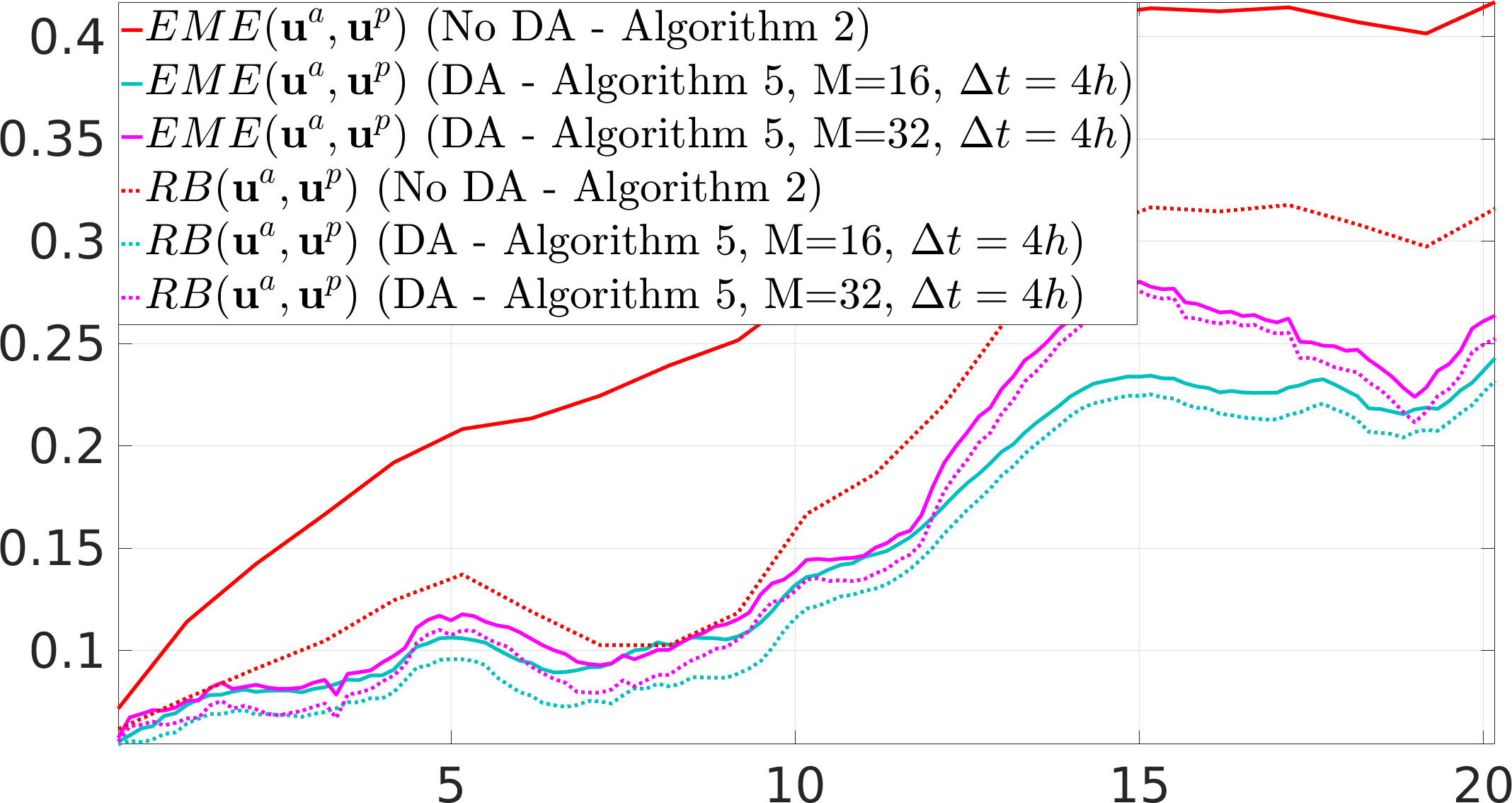}\end{minipage} &
\hspace*{-0.25cm}\begin{minipage}{0.45\textwidth}\includegraphics[width=7.5cm]{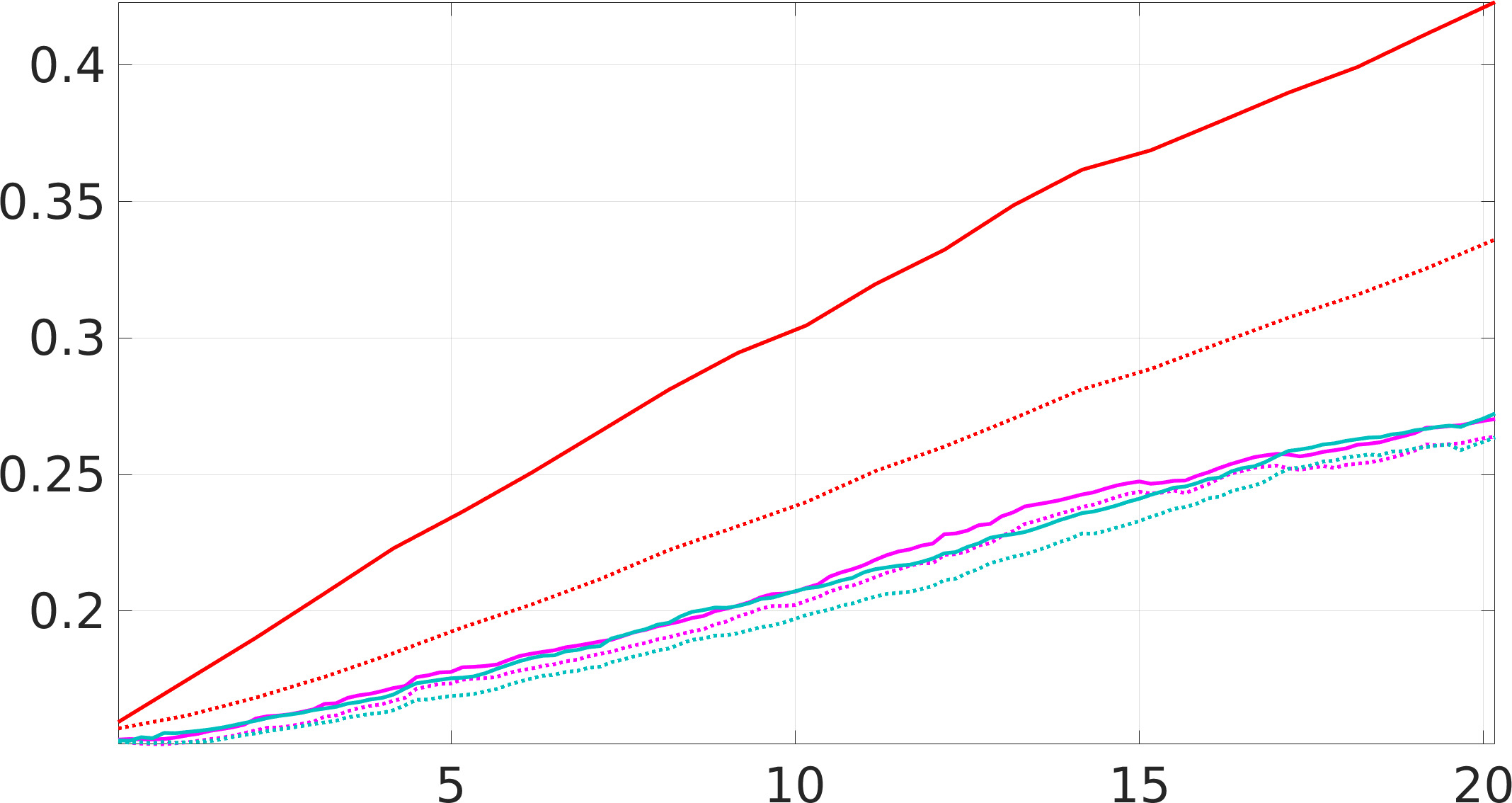}\end{minipage}\\
& &\\[-0.25cm]
&\hspace*{0.5cm}\begin{minipage}{0.1\textwidth}t {\rm [days]}\end{minipage} &\hspace*{0.5cm}\begin{minipage}{0.1\textwidth}t {\rm [days]}\end{minipage}
\end{tabular}
\caption{The same as in Figure~\ref{fig:av_relerr_v_129x65_NDG_ws32}, but for the signal grid $G_s=257\times129$.}
\label{fig:av_relerr_v_257x129_NDG_ws32}
\end{figure}

When more observational data is available (32 weather stations), the uncertainty of the stochastic solution
remains virtually the same compared with the case of using 16 weather stations for both $G_s=129\times65$ 
(Fig.~\ref{fig:spread_129x65_32ws_4h_nudging}) and $G_s=257\times129$ (Fig.~\ref{fig:spread_257x129_32ws_4h_nudging}).

\begin{figure}
\centering
\hspace*{1.0cm}
\begin{tabular}{cc}
\hspace*{0cm}\begin{minipage}{0.1\textwidth} \large$u^p_1$ \end{minipage} &\hspace*{0cm}\begin{minipage}{0.1\textwidth} \large$v^p_1$ \end{minipage}\\
&\\[-0.35cm]
\hspace*{-0.25cm}\begin{minipage}{0.45\textwidth}\includegraphics[scale=0.25]{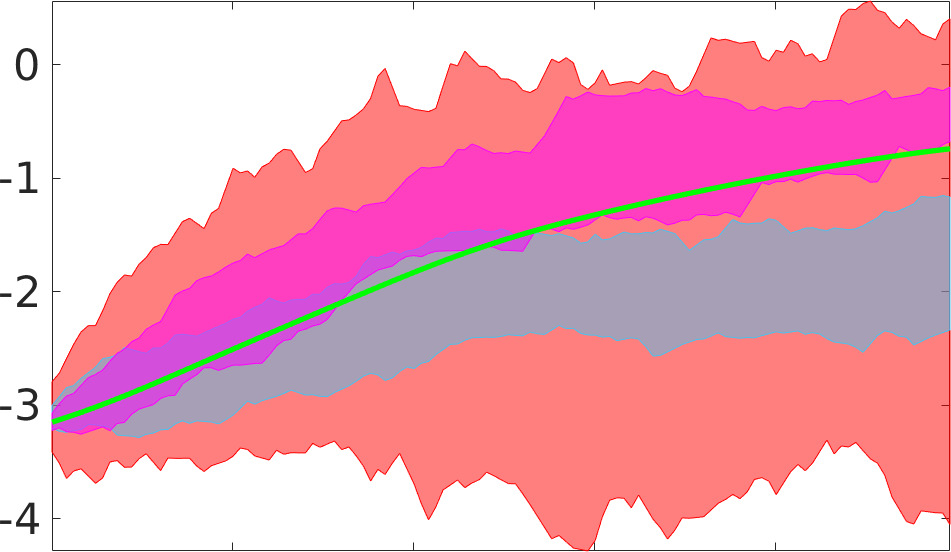}\end{minipage} &
\hspace*{-0.25cm}\begin{minipage}{0.45\textwidth}\includegraphics[scale=0.25]{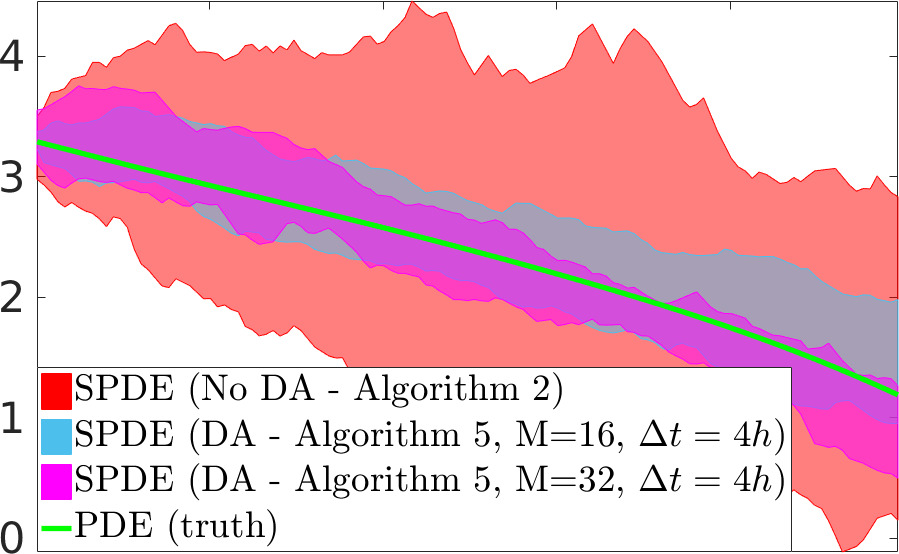}\end{minipage}\\
&\\[-0.25cm]
\hspace*{-0.75cm}\begin{minipage}{0.45\textwidth}\includegraphics[scale=0.25]{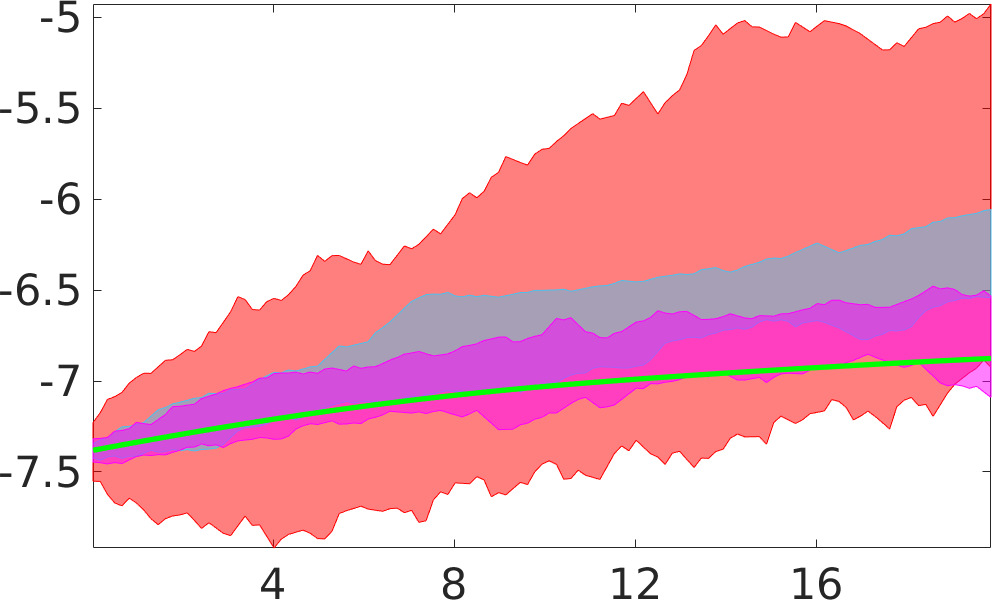}\end{minipage} &
\hspace*{-1cm}\begin{minipage}{0.45\textwidth}\includegraphics[scale=0.25]{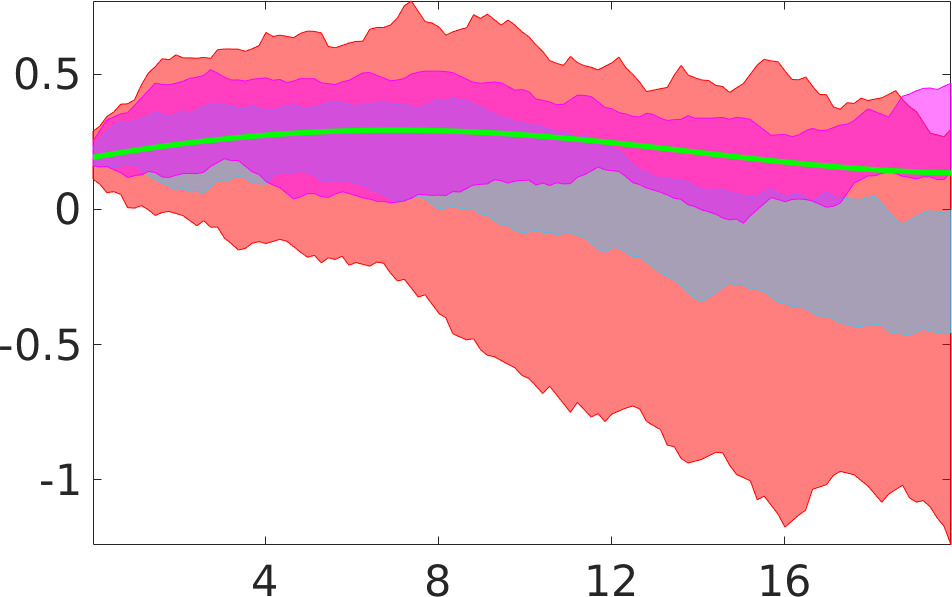}\end{minipage}\\
&\\[-0.25cm]
\hspace*{0cm}\begin{minipage}{0.1\textwidth}t {\rm [days]}\end{minipage} &\hspace*{-0.5cm}\begin{minipage}{0.1\textwidth}t {\rm [days]}\end{minipage}
\end{tabular}
\caption{Shown are typical stochastic spreads of velocity $\mathbf{u}^p_1=(u^p_1,v^p_1)$ at the weather stations
located in the slow flow region (upper row) and fast flow region (lower row) 
for the SPDE without (red) and with (blue) using \textit{tempering, jittering, and nudging (Algorithm 5)}.
The green line is the true solution; the grid size is $G_s=129\times65$.
}
\label{fig:spread_129x65_32ws_4h_nudging}
\end{figure}

\begin{figure}[H]
\centering
\hspace*{1.0cm}
\begin{tabular}{cc}
\hspace*{0cm}\begin{minipage}{0.1\textwidth} \large$u^p_1$ \end{minipage} &\hspace*{0cm}\begin{minipage}{0.1\textwidth} \large$v^p_1$ \end{minipage}\\
&\\[-0.35cm]
\hspace*{-0.25cm}\begin{minipage}{0.45\textwidth}\includegraphics[scale=0.25]{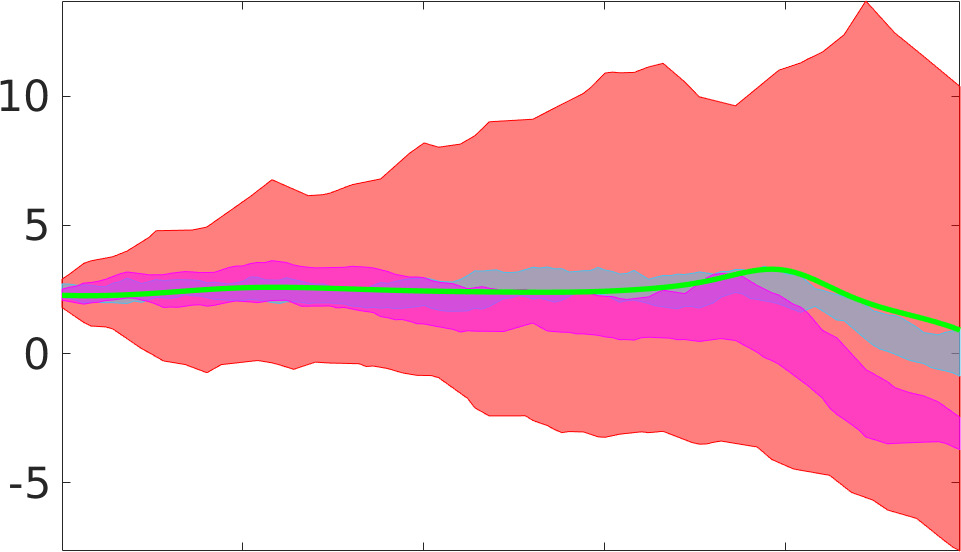}\end{minipage} &
\hspace*{-0.25cm}\begin{minipage}{0.45\textwidth}\includegraphics[scale=0.25]{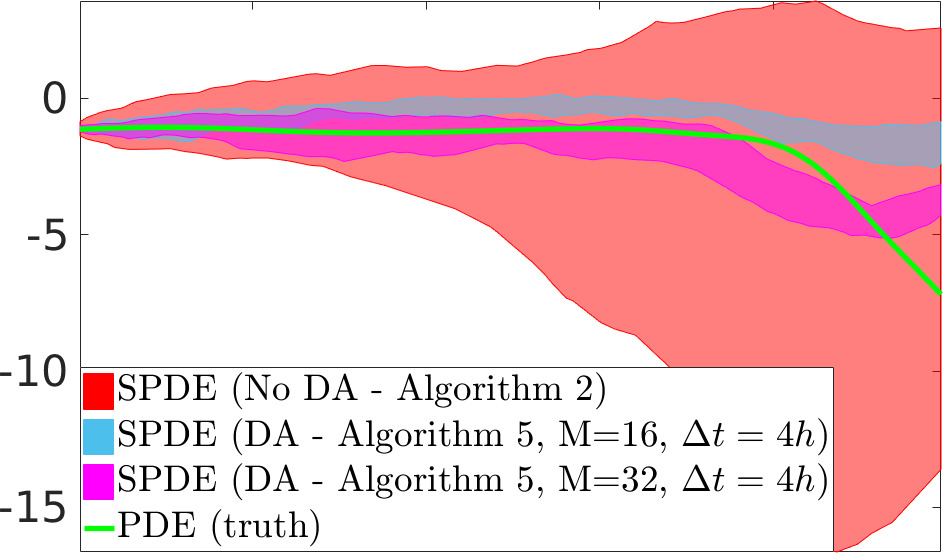}\end{minipage}\\
&\\[-0.25cm]
\hspace*{-0.25cm}\begin{minipage}{0.45\textwidth}\includegraphics[scale=0.25]{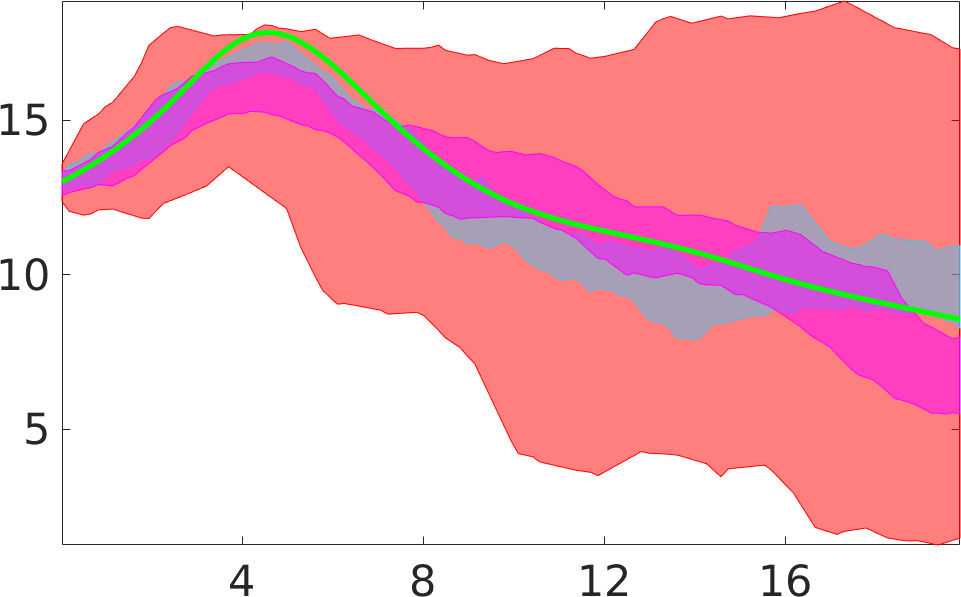}\end{minipage} &
\hspace*{-0.25cm}\begin{minipage}{0.45\textwidth}\includegraphics[scale=0.25]{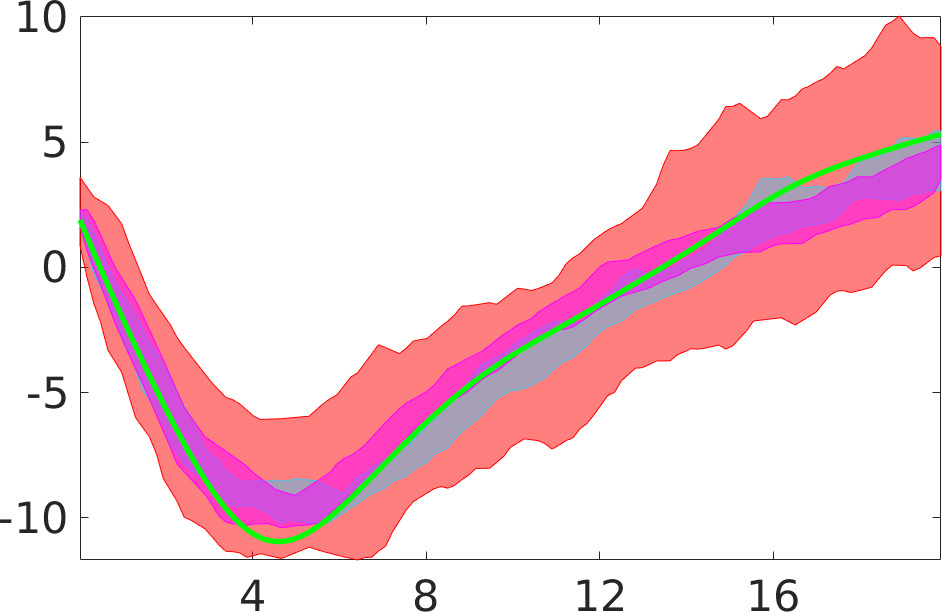}\end{minipage}\\
&\\[-0.25cm]
\hspace*{0cm}\begin{minipage}{0.1\textwidth}t {\rm [days]}\end{minipage} &\hspace*{-0.5cm}\begin{minipage}{0.1\textwidth}t {\rm [days]}\end{minipage}
\end{tabular}
\caption{The same as in Fig.~\ref{fig:spread_129x65_32ws_4h_nudging}, but for the signal grid $G_s=257\times129$.}
\label{fig:spread_257x129_32ws_4h_nudging}
\end{figure}

Again, we conclude that the Algorithm 4
improves the accuracy and reduces the uncertainty of the stochastic spread.
The three parameters that we studied, the data assimilation step, and the size of the grid and the number of weather station, influence the performance of the data assimilation methodology.
In particular, the smaller the data assimilation step is, the higher the accuracy becomes. We have found that
the number of weather stations has a minor effect on the accuracy. 
However,
if the resolution of the signal grid increases so does the
accuracy of the solution computed with using the data assimilation methodology. 
Moreover, increasing the resolution of the signal grid $G_s$ dramatically reduces the uncertainty.
The same conclusions are true for the Algorithm 5 that incorporates the nudging method. 
Moreover, the nudging procedure gives even more accurate solutions and further reduces the spread when compared with Algorithm 4. 

We have also assessed the ensemble reliability on the grid $G_s=257\times129$
by analysing the rank histograms (see, e.g.~\cite{Anderson1996}) 
computed by simulating the stochastic model (without using the data assimilation methodology) (Fig.~\ref{fig:rank_histograms_257x129_32ws_4h}a)
and compared them with the rank histograms for the ensemble produced by Algorithm 5 (Fig.~\ref{fig:rank_histograms_257x129_32ws_4h}b).
As the plots show,  the stochastic ensemble, computed without the regular corrections made through the data assimilation method, has a bias (rank histograms are not flat) at many observation locations. As a result it will makes the ensemble prediction unreliable. 

\begin{figure}[H]
\centering
\hspace*{3.0cm}
\begin{tabular}{cc}
\hspace*{-1.5cm}\begin{minipage}{0.1\textwidth} \large$u^p_1$ \end{minipage} \hspace*{2cm}\begin{minipage}{0.1\textwidth} \large$v^p_1$ \end{minipage} &
\hspace*{3cm}\begin{minipage}{0.1\textwidth} \large$u^p_1$ \end{minipage} \hspace*{2cm}\begin{minipage}{0.1\textwidth} \large$v^p_1$ \end{minipage}\\
&\\[-0.35cm]
\hspace*{-3.5cm}\begin{minipage}{0.45\textwidth}\includegraphics[scale=0.15]{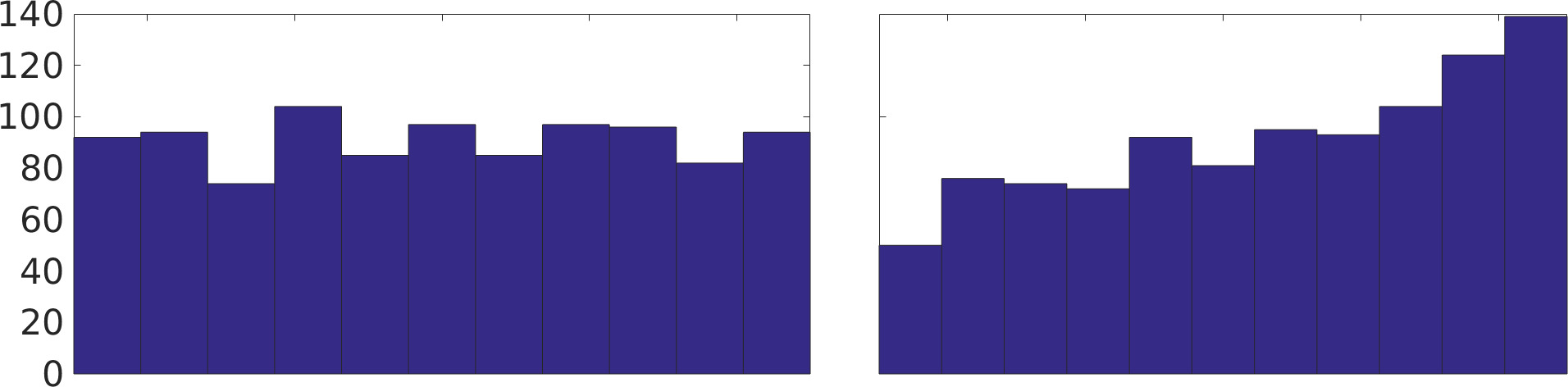}\end{minipage}&
\hspace*{1cm}\begin{minipage}{0.45\textwidth}\includegraphics[scale=0.15]{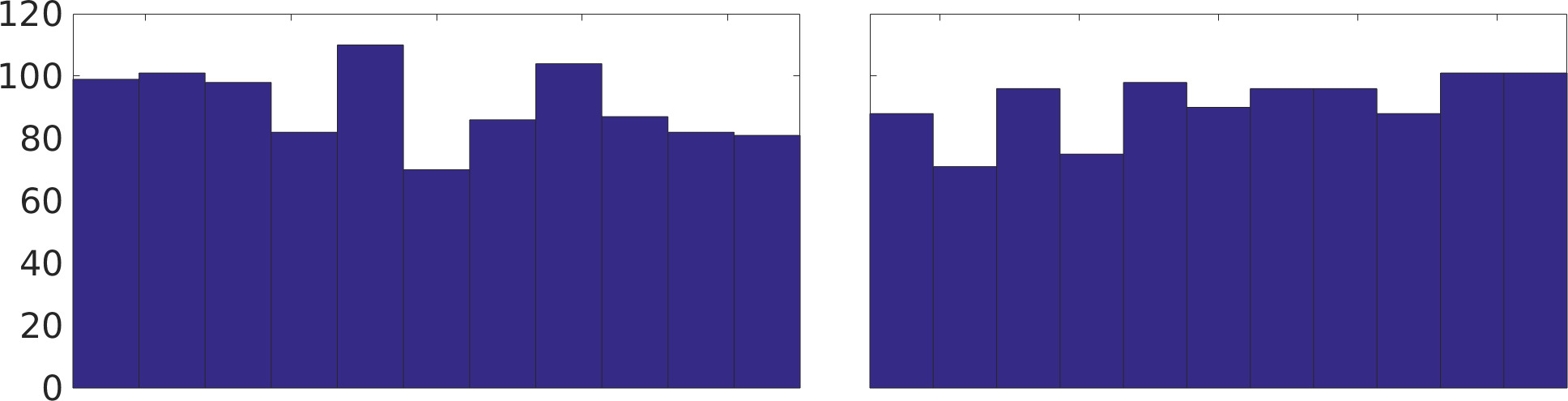}\end{minipage}\\
&\\[-0.125cm]
\hspace*{-3.5cm}\begin{minipage}{0.45\textwidth}\includegraphics[scale=0.15]{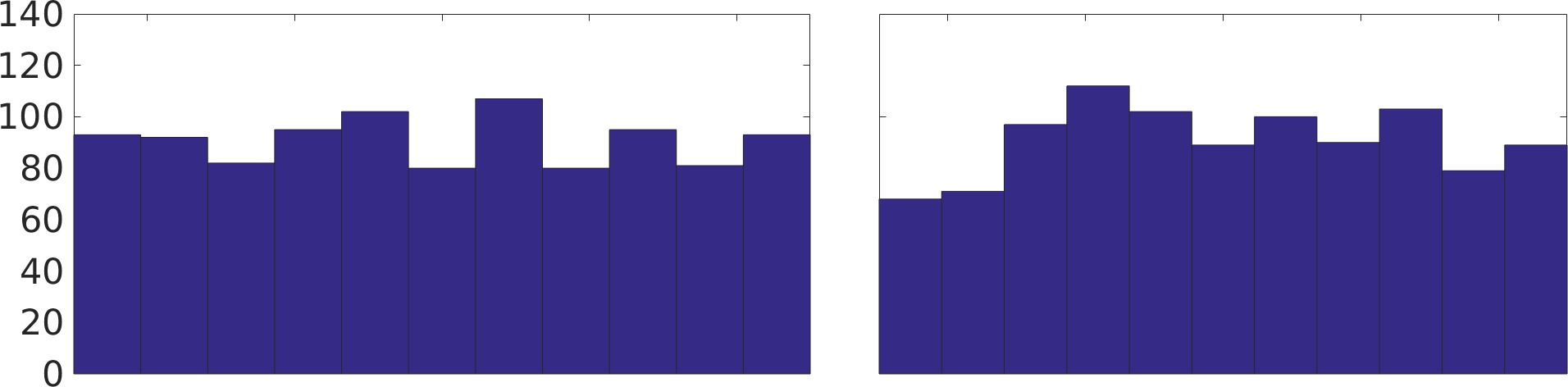}\end{minipage}&
\hspace*{1cm}\begin{minipage}{0.45\textwidth}\includegraphics[scale=0.15]{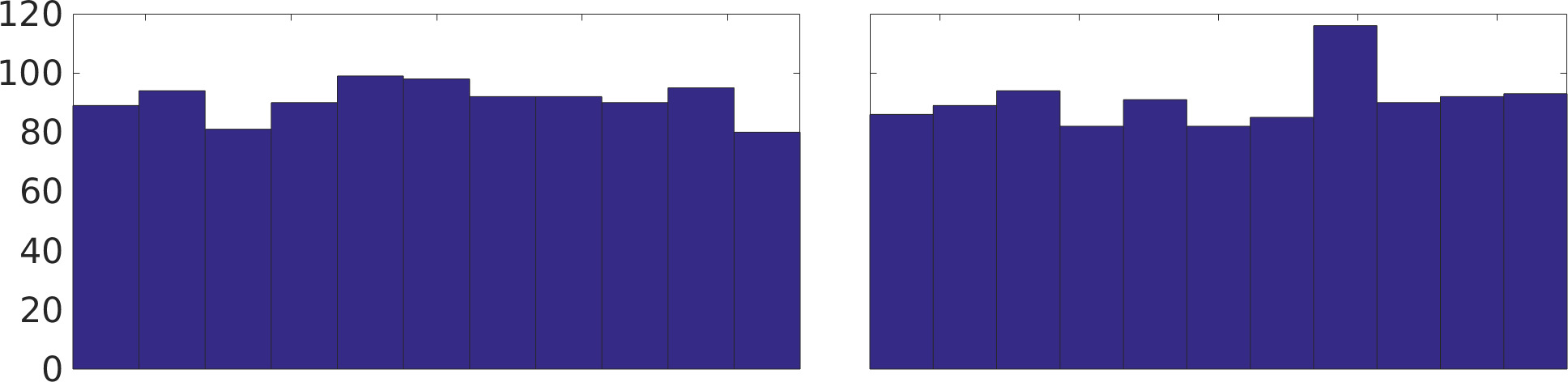}\end{minipage}\\
&\\[-0.125cm]
\hspace*{-3.5cm}\begin{minipage}{0.45\textwidth}\includegraphics[scale=0.15]{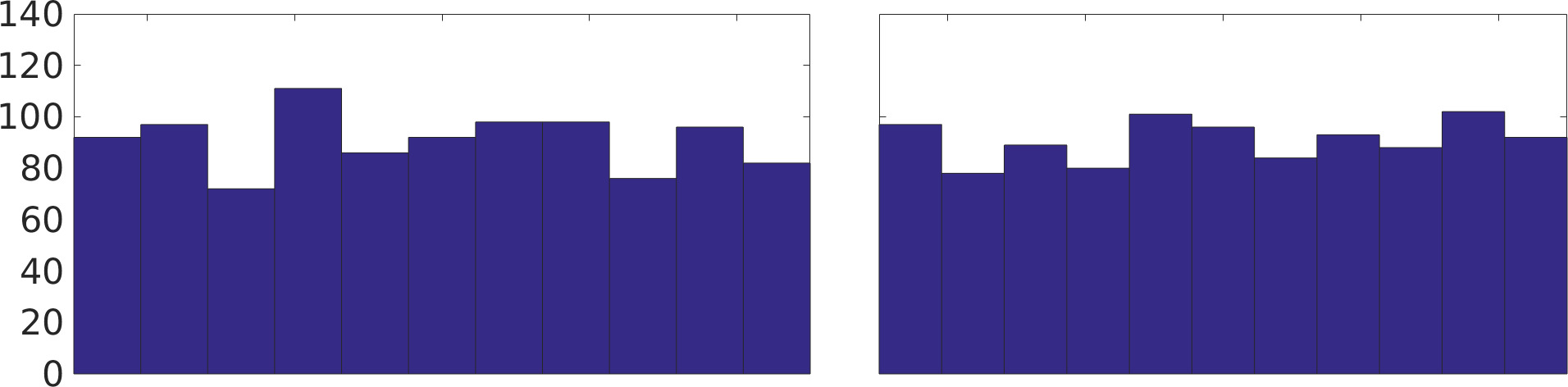}\end{minipage}&
\hspace*{1cm}\begin{minipage}{0.45\textwidth}\includegraphics[scale=0.15]{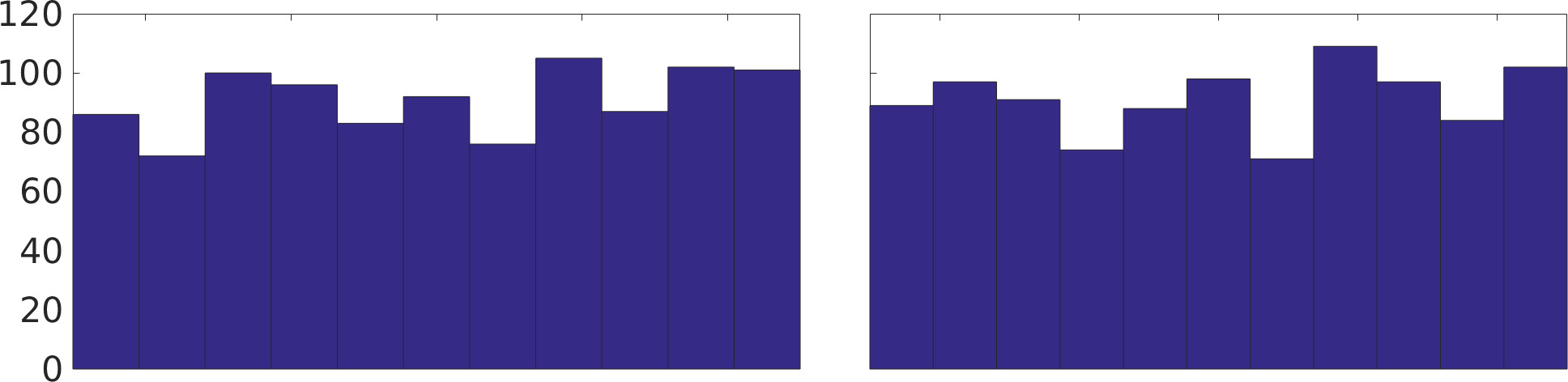}\end{minipage}\\
&\\[-0.125cm]
\hspace*{-3.5cm}\begin{minipage}{0.45\textwidth}\includegraphics[scale=0.15]{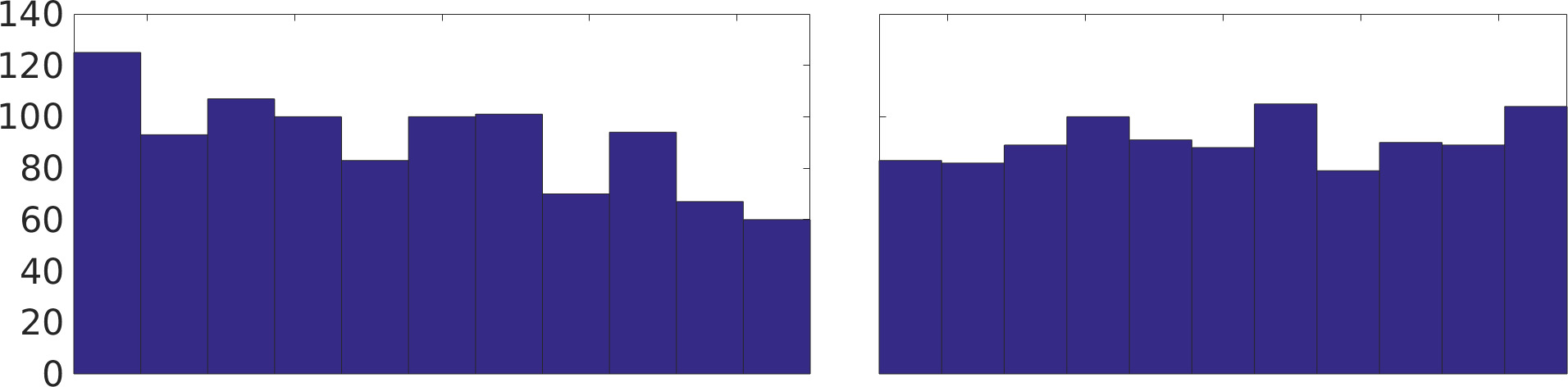}\end{minipage}&
\hspace*{1cm}\begin{minipage}{0.45\textwidth}\includegraphics[scale=0.15]{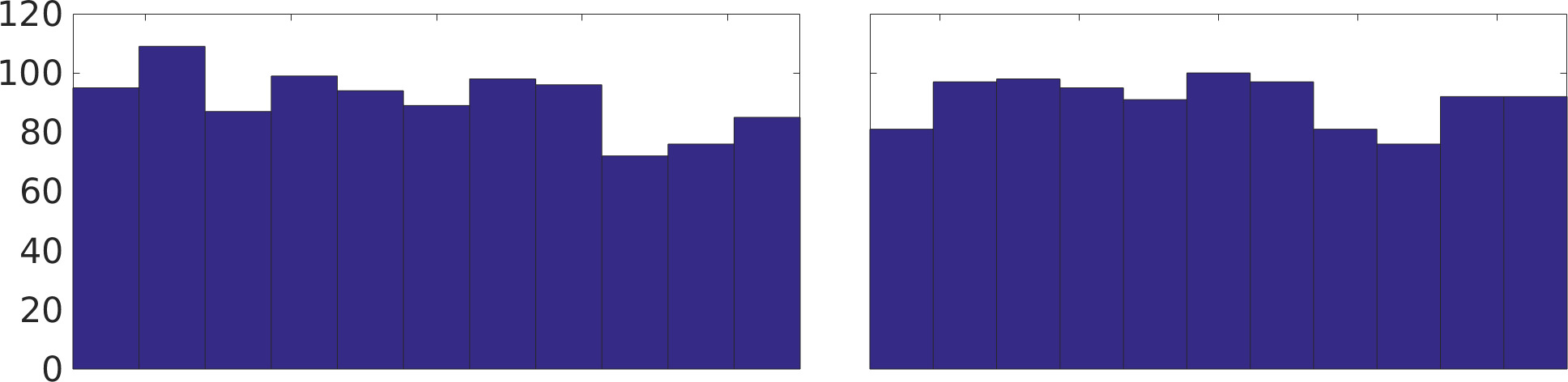}\end{minipage}\\
&\\[-0.125cm]
\hspace*{-3.5cm}\begin{minipage}{0.45\textwidth}\includegraphics[scale=0.15]{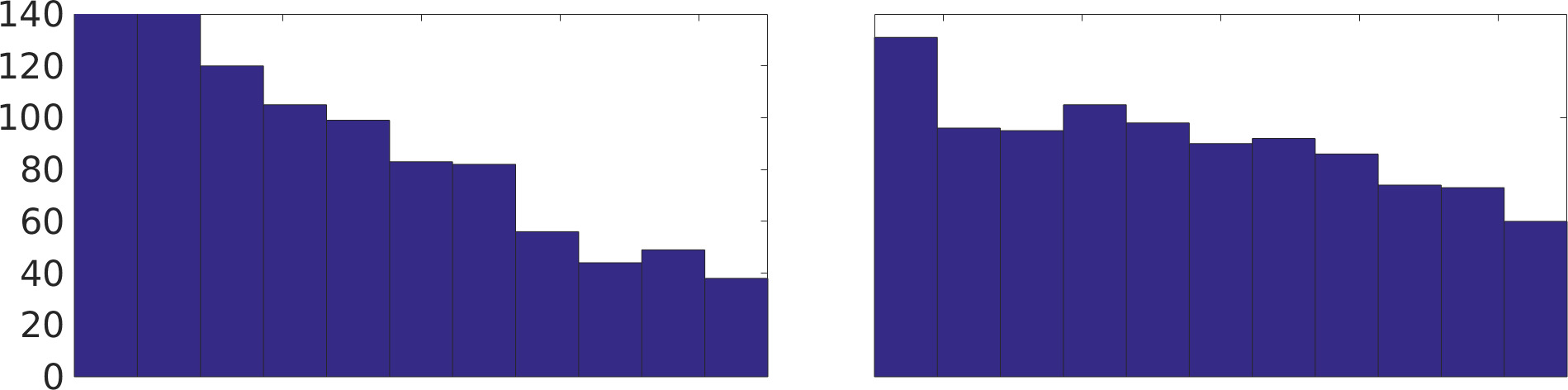}\end{minipage}&
\hspace*{1cm}\begin{minipage}{0.45\textwidth}\includegraphics[scale=0.15]{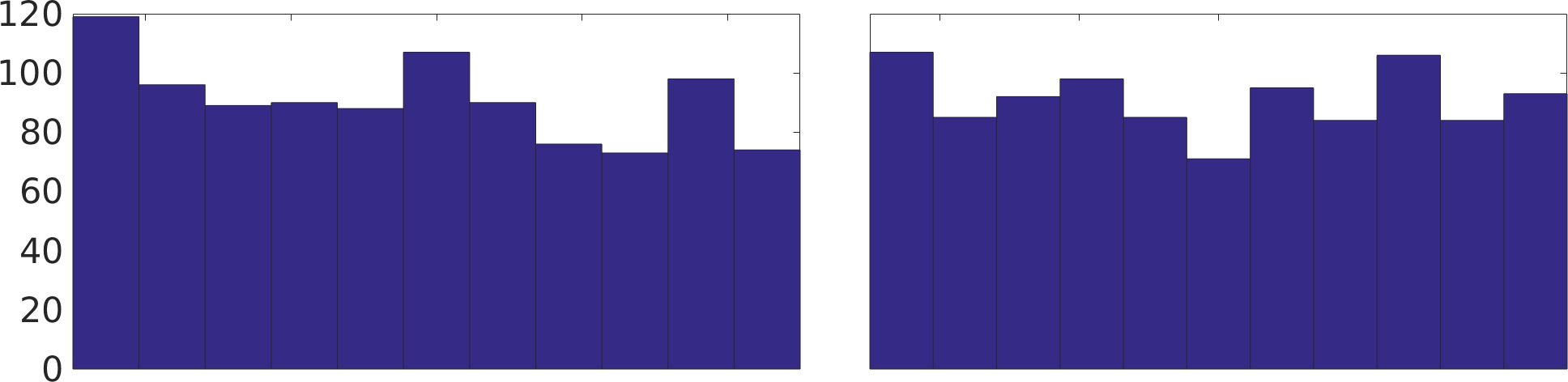}\end{minipage}\\
&\\[-0.125cm]
\hspace*{-3.5cm}\begin{minipage}{0.45\textwidth}\includegraphics[scale=0.15]{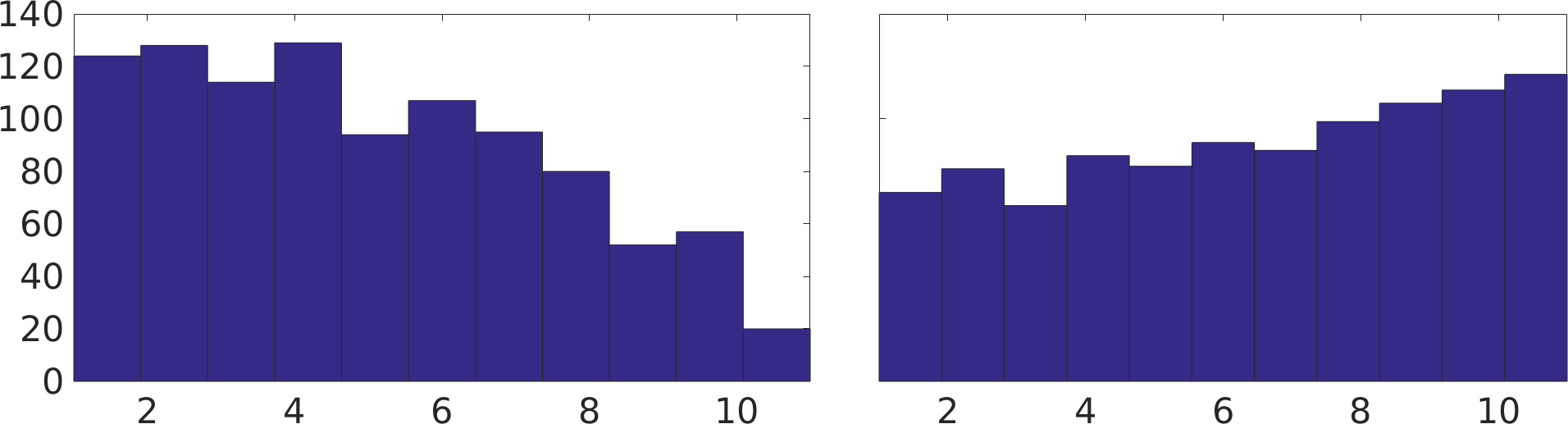}\end{minipage}&
\hspace*{1cm}\begin{minipage}{0.45\textwidth}\includegraphics[scale=0.15]{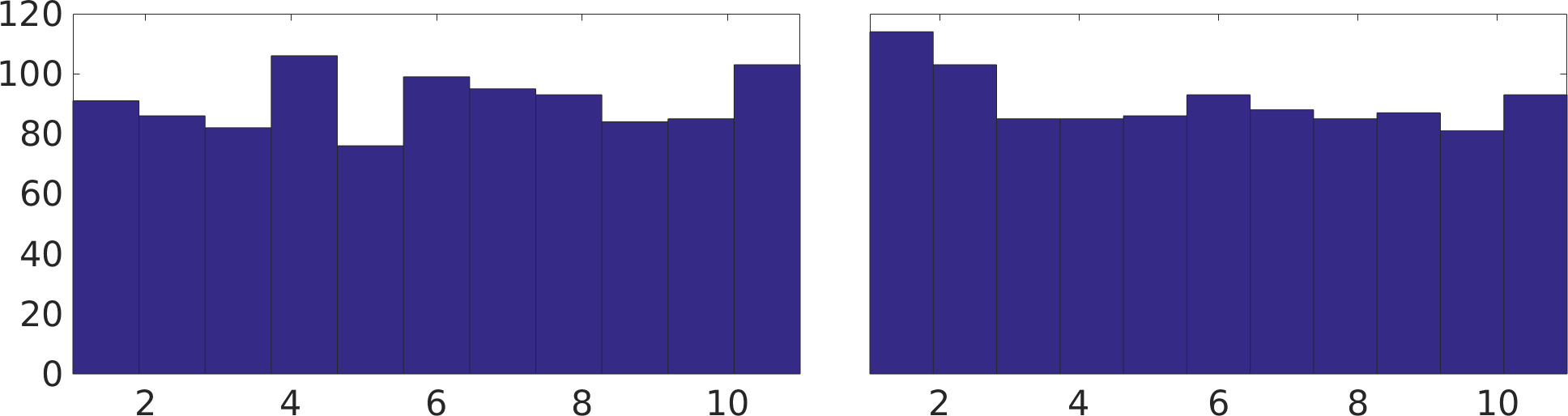}\end{minipage}\\
\hspace*{-2.5cm}\begin{minipage}{0.4\textwidth} \textbf{(a)} Rank histogram without data assimilation \end{minipage} & 
\hspace*{2cm}\begin{minipage}{0.4\textwidth} \textbf{(b)} Rank histograms with data assimilation \end{minipage}\\
\end{tabular}
\caption{Rank histograms for velocity $\mathbf{u}^p_1=(u^p_1,v^p_1)$ at six different locations (not shown).
Three observation points are located in the fast flow within red jets (first three upper rows in the plot), and the other three
observation points are located in the slow flow between the jets (next three rows in the plot).
Each histogram is based on 500 forecast-observation pairs generated by solving the stochastic QG model without (a) and with (b) 
using data assimilation. 
For simulating the stochastic QG model we use $N=100$ (ensemble size), $K=32$ (number of $\xi$'s);
for the data assimilation method 
we use \textit{Algorithm 5 (tempering, jittering, and 
nudging)}, the number of observation locations $M=32$, the data assimilation step $\Delta t=4h$.
Each ensemble member for the rank histogram is selected randomly from the ensemble of 100 members every 4 hours.}
\label{fig:rank_histograms_257x129_32ws_4h}
\end{figure}

The situations changes when one uses the data assimilation method based
one tempering, jittering, and nudging (Algorithm 5); see Fig.~\ref{fig:rank_histograms_257x129_32ws_4h}b.
All the rank histograms now show no sign of a pronounced bias and demonstrate a well-calibrated ensemble. 
In other words, the application of the data assimilation methodology proposed in the paper corrects the bias introduced by the paramaterisation and produces a reliable forecast.

Finally, we have compared true PV anomaly $q$ fields and the ones computed with using the data assimilation
methodology based on tempering, jittering, and nudging (Algorithm 5), see Fig.~\ref{fig:deter_vs_stoch_mu4D-8_257x129}.

\begin{figure}[H]
\centering
\begin{tabular}{ccc}
&  \hspace*{-1cm} $q^a_1$ & \hspace*{-1cm} $\bar{q}^p_1$ \\
\begin{minipage}{0.02\textwidth}\rotatebox{90}{$t=1$ day}\end{minipage} 
& \hspace*{-0.25cm}\begin{minipage}{0.31\textwidth}\includegraphics[width=5cm,height=2.5cm]{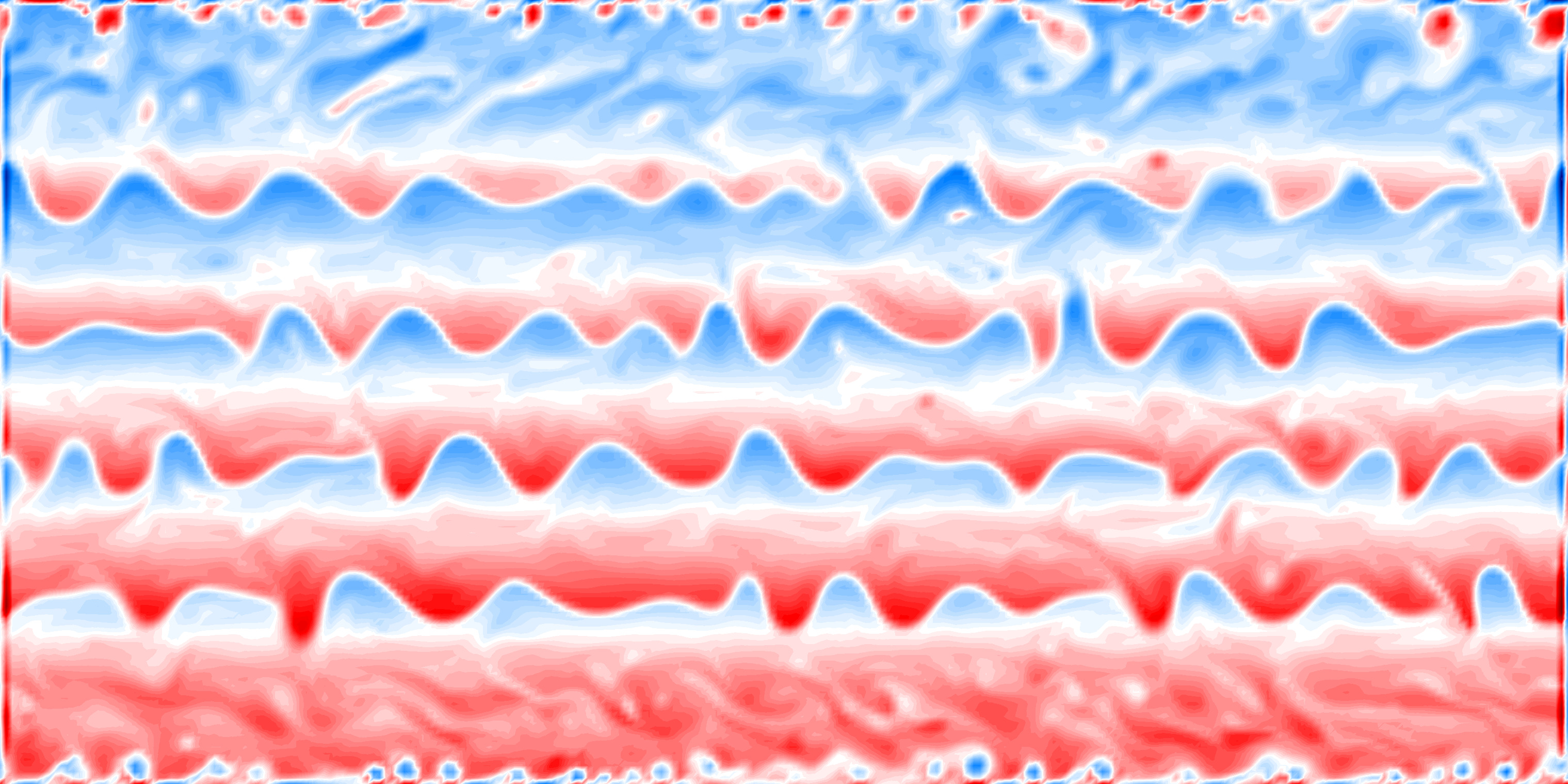}\end{minipage} 
& \hspace*{-0.25cm}\begin{minipage}{0.31\textwidth}\includegraphics[width=5cm,height=2.5cm]{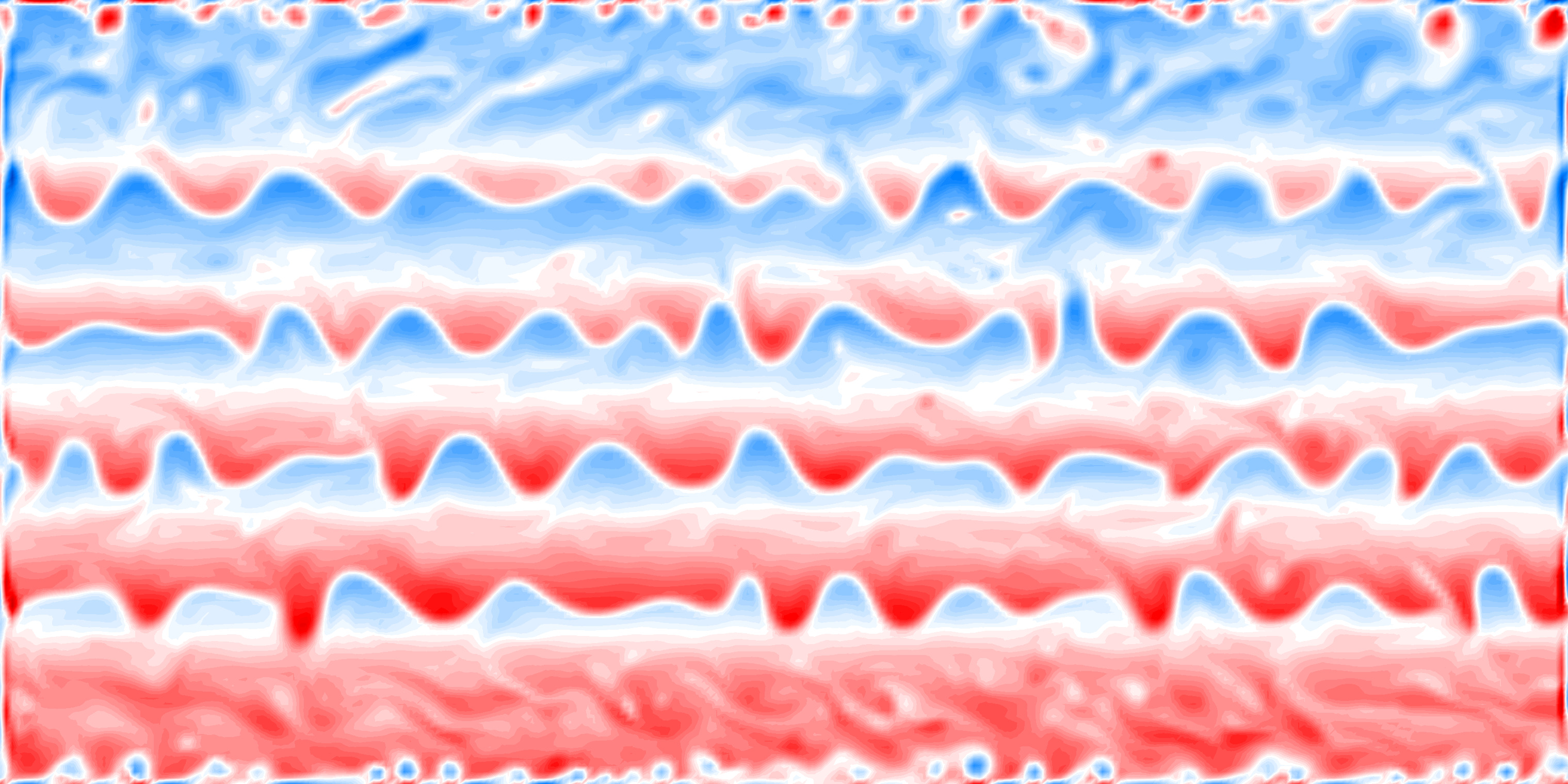}\end{minipage}
\\
\\[-0.25cm]
\begin{minipage}{0.02\textwidth}\rotatebox{90}{$t=10$ days}\end{minipage} 
& \hspace*{-0.25cm}\begin{minipage}{0.31\textwidth}\includegraphics[width=5cm,height=2.5cm]{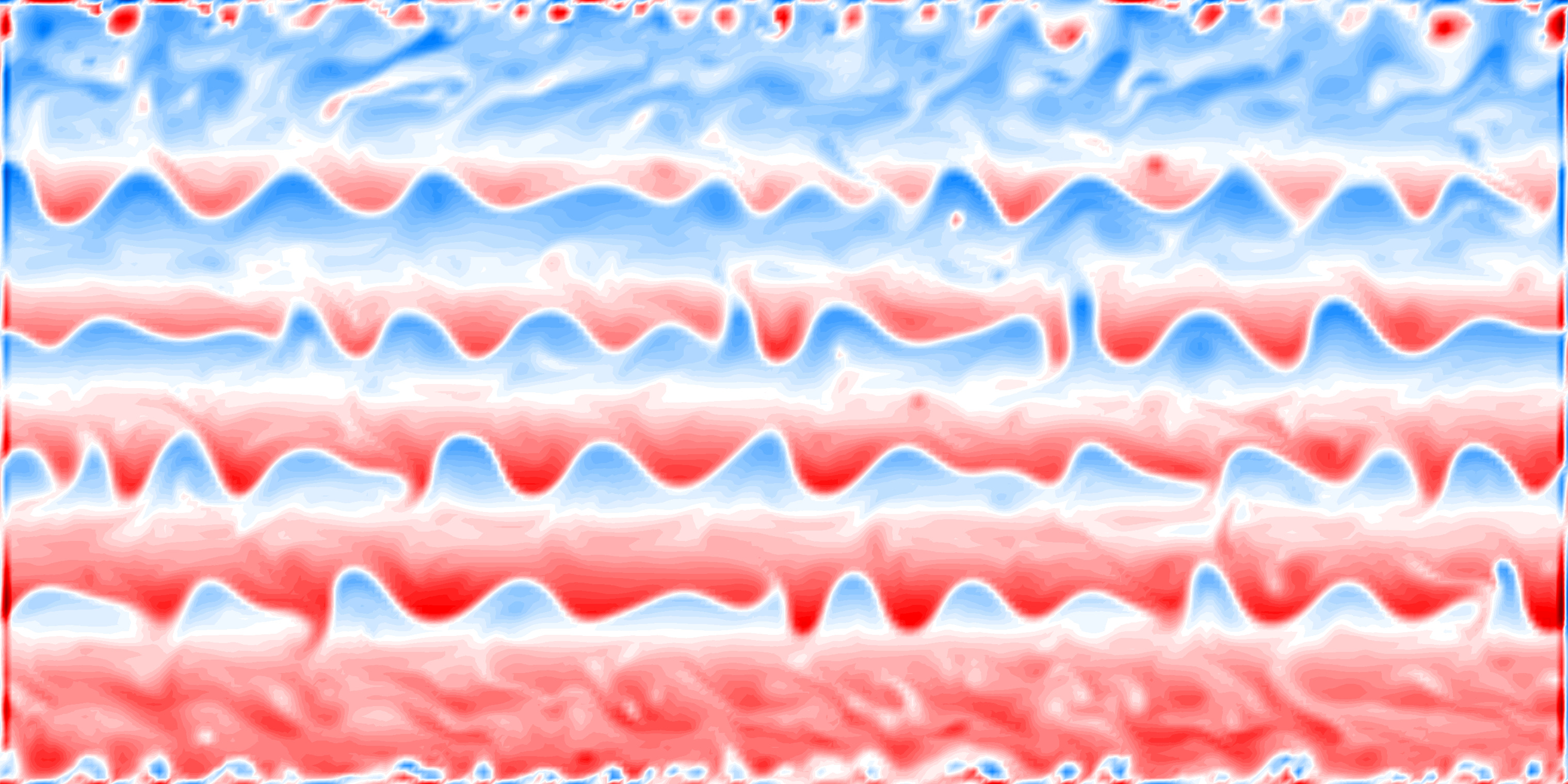}\end{minipage} 
& \hspace*{-0.25cm}\begin{minipage}{0.31\textwidth}\includegraphics[width=5cm,height=2.5cm]{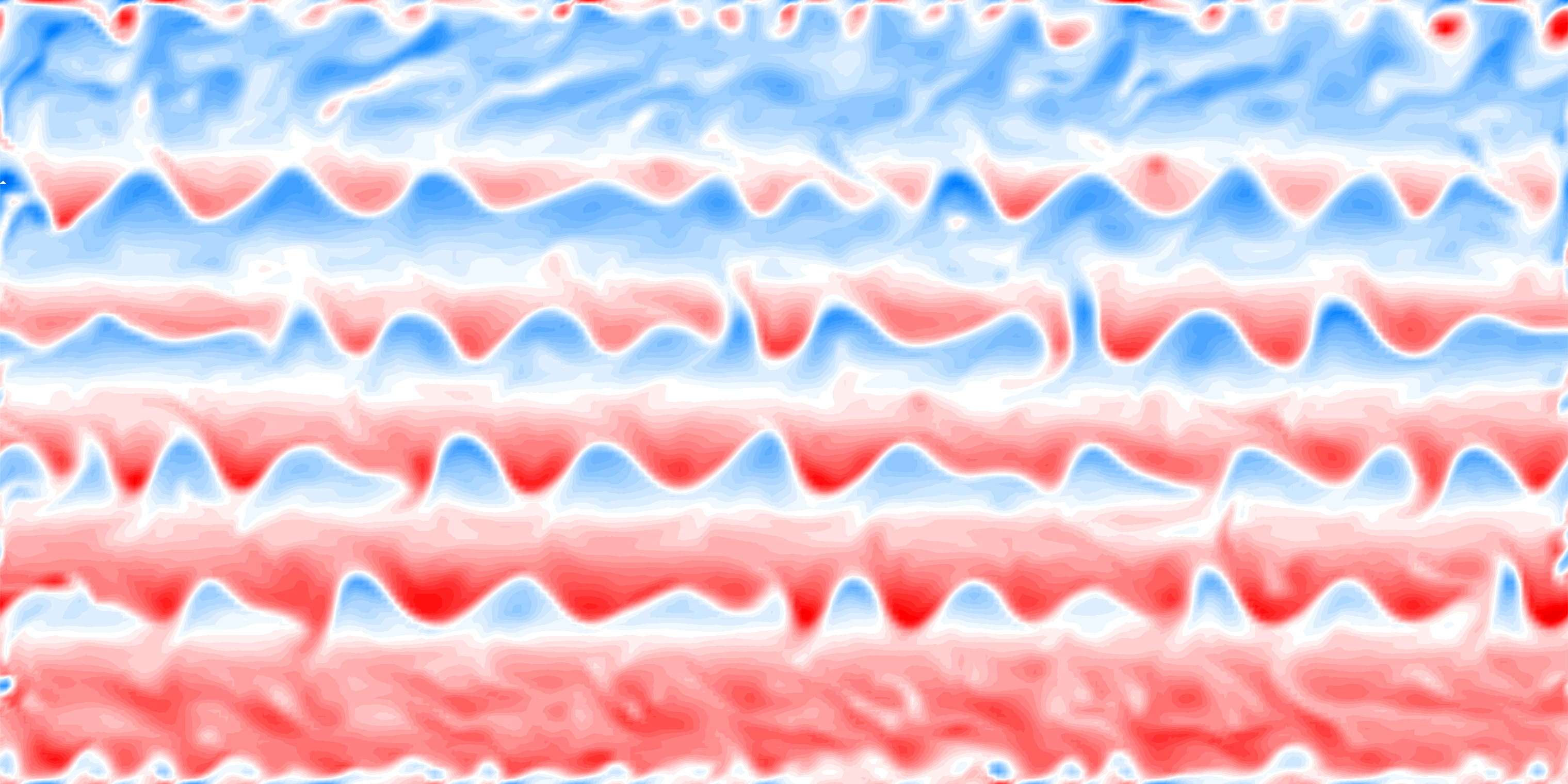}\end{minipage}
\\
\\[-0.25cm]
\begin{minipage}{0.02\textwidth}\rotatebox{90}{$t=20$ days}\end{minipage} 
& \hspace*{-0.25cm}\begin{minipage}{0.31\textwidth}\includegraphics[width=5cm,height=2.5cm]{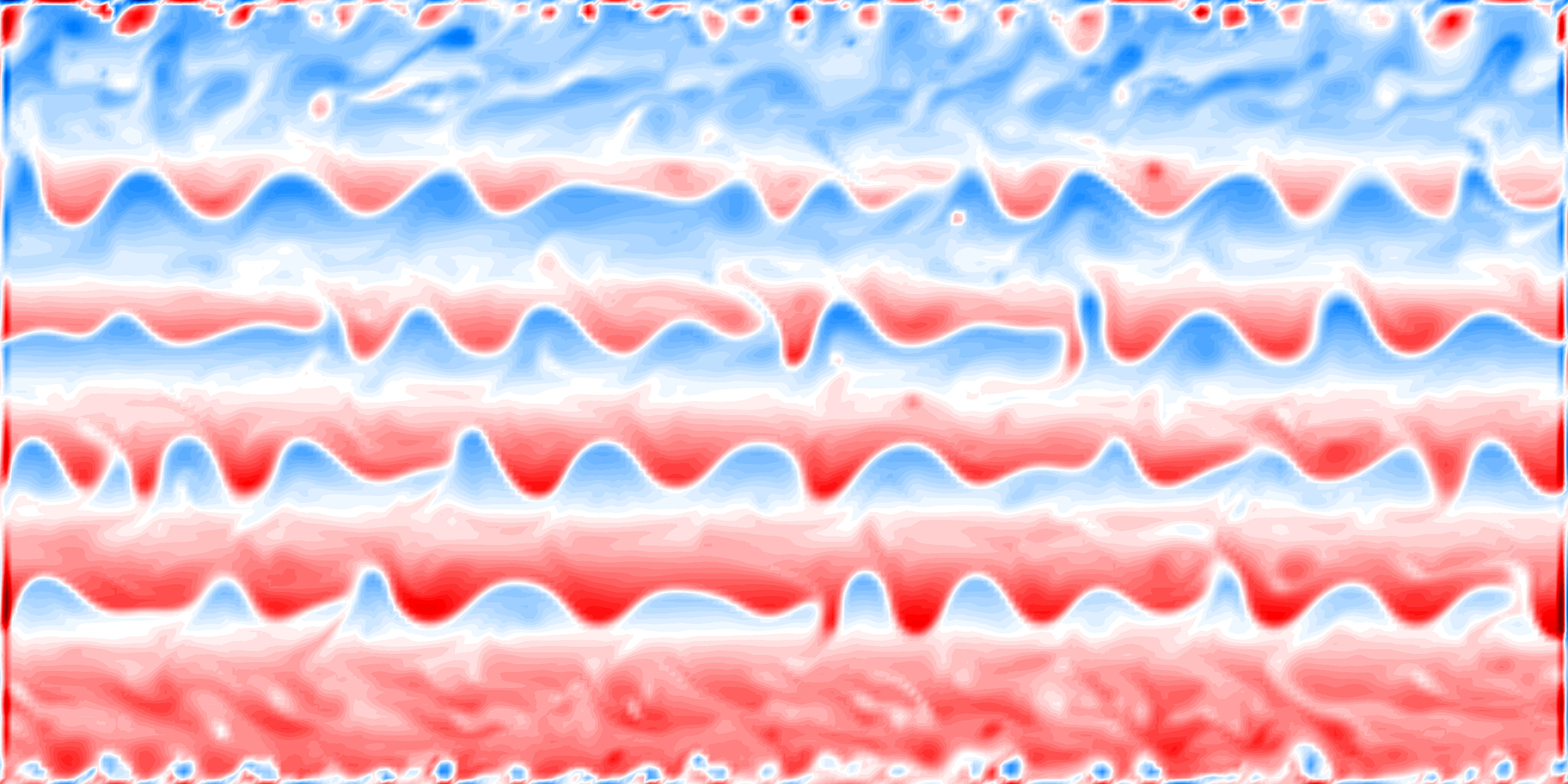}\end{minipage} 
& \hspace*{-0.25cm}\begin{minipage}{0.31\textwidth}\includegraphics[width=5cm,height=2.5cm]{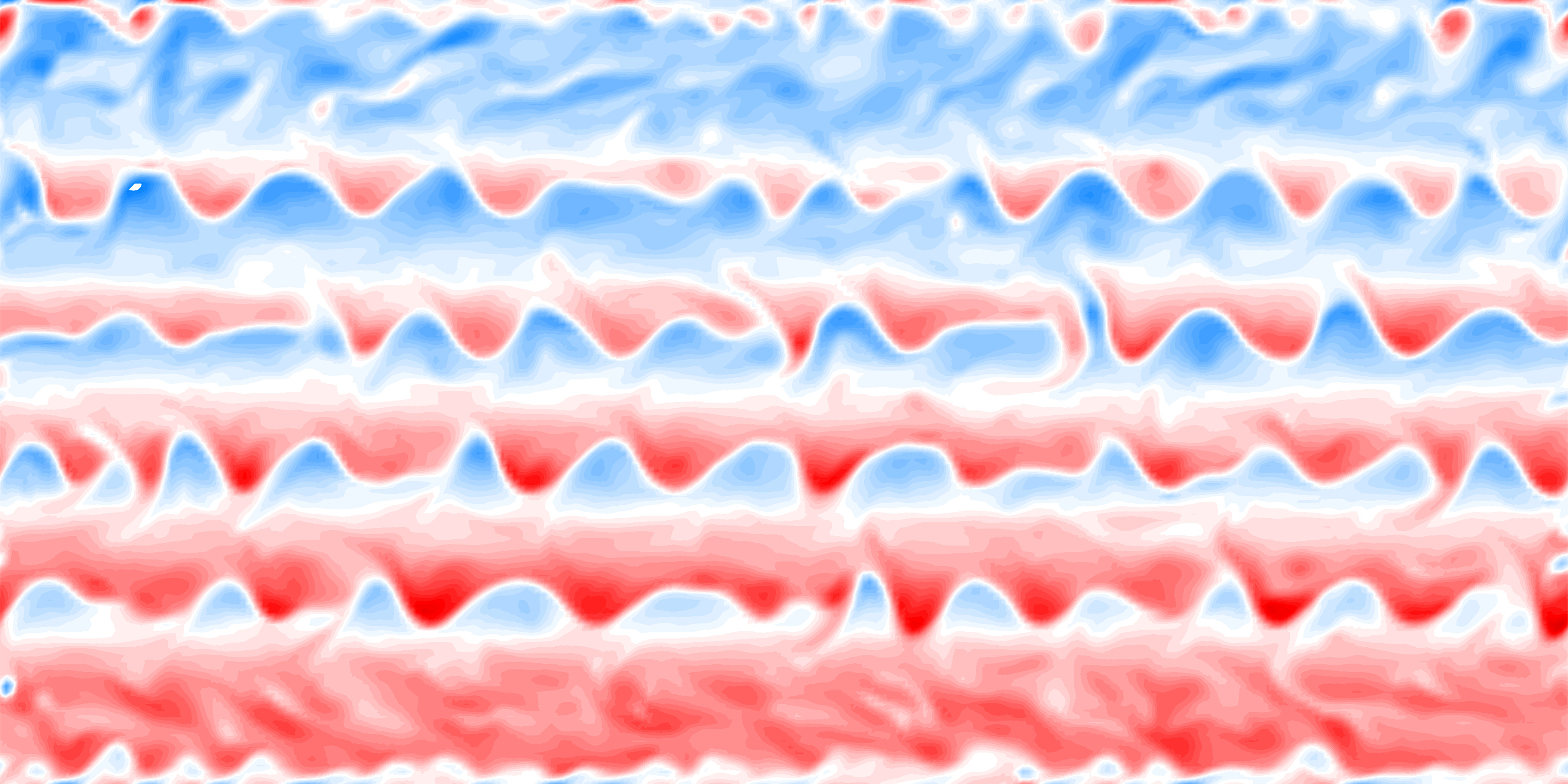}\end{minipage}
\\
& &\\[-0.25cm]
\multicolumn{3}{c}{\hspace*{0.3cm}\includegraphics[width=6cm,height=0.75cm]{figs/colorbar_bwr_mD-5_pD-5.jpg}}\\
\end{tabular} 
\caption{The series of snapshots shows the true deterministic solution $q^a_1$
and parameterised solution $\bar{q}^p_1$ (averaged over the stochastic ensemble of size $N=100$) computed with the data assimilation
Algorithm 5. 
All the solutions were computed on the same coarse grid $G^c=257\times129$, have the same initial condition,
and the parameterised solution uses 32 leading EOFs, and 32 weather stations.
All the fields are given in units of $[s^{-1}f^{-1}_0]$, where $f_0=0.83\times10^{-4}\, {\rm s^{-1}}$ is the Coriolis parameter.
}
\label{fig:deter_vs_stoch_mu4D-8_257x129}
\end{figure}

As Fig.~\ref{fig:deter_vs_stoch_mu4D-8_257x129} shows, the paramaterised solution $\bar{q}^p_1$ using the data assimilation methodology based
on tempering, jittering, and nudging (Algorithm 5) gives an accurate forecast of the true state $q^a_1$ even with a very small number
of weather stations.

\section{Conclusion and future work}
\label{sec:concl}

In this paper, we have reported on our recent progress in developing an ensemble based data assimilation methodology for high dimensional fluid dynamics models. Our methodology involves a particle filter which combines model reduction, tempering, jittering, and nudging.
The methodology has been tested on the two-layer quasi-geostrophic model with $O(10^6)$ degrees of freedom. Only a minute fraction of these are noisily observed (16 and 32 weather stations). The model is reduced by following the stochastic variational approach for geophysical fluid dynamics introduced in~\cite{holm2015variational}. We have also introduced a stochastic time-stepping scheme for the quasi-geostrophic model and have proved its consistency in time. In addition, we have analyzed the effect of different procedures (tempering, jittering, and nudging) on the accuracy and uncertainty of the stochastic spread.
Our main findings are as follows:
\begin{itemize}
    \item The tempering and jittering procedure (Algorithm 4) improves the accuracy by reducing the uncertainty of the stochastic spread;
    \item The nudging procedure (Algorithm 5) brings major improvements to the combinations of the tempering and jittering (Algorithm 4),
    both in terms of the relative bias (RB) and ensemble mean relative $l_2$-norm error (EME)
    \item The number of weather stations has a minor effect on the RB and EME;
    \item The size of the data assimilation step has a substantial effect; namely, the smaller the data assimilation step, the higher the accuracy;
    \item The resolution of the signal grid significantly improves the accuracy and reduces the uncertainty of the stochastic spread;
    \item The proposed data assimilation methodology corrects the bias introduced by the paramaterisation and produces a reliable forecast.
\end{itemize}
We regard the data assimilation method based on tempering and jittering combined with the nudging method proposed here as a potentially valuable addition to data assimilation methodologies. We also expect it to be useful in developing data assimilation methodologies for larger, more comprehensive ocean models.
The combination of the four components presented here (model reduction, tempering, jittering and nudging) can be enhanced by further improvements including localization, space-time data assimilation, etc. Applications of these combined components will form the subject of subsequent work. 

\begin{acknowledgements}
The authors thank The Engineering and Physical Sciences Research Council for the support of this work through the grant EP/N023781/1.
The work of the second author has been partially supported by a UC3M-Santander Chair of Excellence grant held at the Universidad Carlos III de Madrid.
\end{acknowledgements}

%
\section*{Conflict of interest}
The authors declare that they have no conflict of interest.

\bibliographystyle{abbrv}
\bibliography{da_sqg}   


\end{document}